\documentclass{article}

\usepackage{amsmath,amssymb,amsfonts,amsthm,epsf,a4,color,graphicx,eucal,mathrsfs}

\newtheorem{theorem}{Theorem}[section]
\newtheorem{lemma}[theorem]{Lemma}
\newtheorem{definition}[theorem]{Definition}

\newtheorem{proposition}[theorem]{Proposition}
\newtheorem{corollary}[theorem]{Corollary}
\newtheorem{remark}[theorem]{Remark}
\newtheorem{notation}[theorem]{Notation}

\usepackage{color}
\usepackage{amsmath}
\usepackage{amsthm}
\usepackage{amssymb}
\usepackage{eucal}
\usepackage{epsfig}
\usepackage{psfrag}
\usepackage{graphicx}
\usepackage{paralist}
\usepackage{textcomp}
\usepackage{pgf}   

\numberwithin{equation}{section}

\newcommand{\R}{\mathbb{R}}
\newcommand{\N}{\mathbb{N}}

\newcommand\JUMP[1]{\mathchoice
                  {\big[\hspace*{-.3em}\big[#1\big]\hspace*{-.3em}\big]}
                   {[\hspace*{-.15em}[#1]\hspace*{-.15em}]}
                   {[\![#1]\!]}
                   {[\![#1]\!]}}
\newcommand{\GC}{\Gamma_{\!\scriptscriptstyle{\rm C}}}
\newcommand{\GD}{\Gamma_{\!\scriptscriptstyle{\rm D}}}
\newcommand{\GN}{\Gamma_{\!\scriptscriptstyle{\rm N}}}

\newcommand{\barGC}{\overline{\Gamma}_{\!\scriptscriptstyle{\rm C}}}
\newcommand{\barGD}{\overline{\Gamma}_{\!\scriptscriptstyle{\rm D}}}
\newcommand\calE{\mathcal E}
\newcommand\calR{\mathcal R}

\newcommand\calG{\mathcal G}
\newcommand\calK{\mathcal K}

\newcommand\calX{\mathcal X}
\newcommand\calL{\mathcal L}
\newcommand\calH{\mathcal H}

\newcommand\calV{\mathcal V}
\newcommand\calJ{\mathcal J}
\newcommand{\rmD}{\mathrm{D}}

\newcommand\pl\partial

\newcommand{\Var}{\mathrm{Var}}

\newcommand\bbC{\mathbb C}
\newcommand\bbD{\mathbb D}

\newcommand{\supp}{\mathop{\mathrm{supp}}}
\newcommand{\SBV}{\mathrm{SBV}}

\newcommand{\Omegaone}{\Omega_+}
\newcommand{\Omegatwo}{\Omega_-}
\newcommand{\barOmegaone}{\overline{\Omega}_+}
\newcommand{\barOmegatwo}{\overline{\Omega}_-}

\newcounter{myfigure}
\newenvironment{my-picture}[3]{\refstepcounter{myfigure}\label{#3}\setlength{\unitlength}{\textwidth}\begin{picture}(#1,#2)}{\end{picture}}
\newcommand{\Surf}{\mathcal{H}^{d-1}}

\newcommand{\Spu}{\mathbf{V}}
\newcommand{\Spv}{\mathbf{V}}
\newcommand{\Spw}{\mathbf{W}}
\newcommand{\Spz}{\mathbf{Z}}
\newcommand{\Spx}{\mathbf{X}}

\newcommand{\ingrsys}{(\Spu, \Spw, \Spz,\calV,\calK,\calR,\calE)}
\newcommand{\ingrsysk}{(\Spu,\Spw,\Spz,\calV,\calK,\calR_k,\calE_k)}
\newcommand{\ingrsysinf}{(\Spu,\Spw,\Spz,\calV,\calK,\calR_\infty,\calE_\infty)}
\newcommand{\spyn}{\mathbf{Y}_n^{s}}
\newcommand{\spyns}{\widetilde{\mathbf{Y}}_n^{s}}
\newcommand{\eps}{\varepsilon}

\newcommand{\pairing}[4]{ \sideset{_{ #1 }}{_{ #2 }}  {\mathop{\langle #3 , #4
\rangle}}}

\newcommand{\inner}[3]{ \sideset{_{}}{_{ #1 }}  {\mathop{( #2 , #3
)}}}

\newcommand{\foraa}{\text{for a.a. }}

\newcommand{\dissu}{\calV}
\newcommand{\dissr}{\calR}

\newcommand{\dd}{\, \mathrm{d}}
\newcommand{\norm}[2]{\| #1\|_{#2}}
\newcommand{\ene}[3]{\calE(#1,#2,#3)}
\newcommand{\VZ}{\Spu_z}
\newcommand{\BV}{\mathrm{BV}}
\newcommand{\weakto}{\rightharpoonup}
\newcommand{\weaksto}{\overset{*}{\rightharpoonup}}

\definecolor{ddcyan}{rgb}{0,0.1,0.9}
\definecolor{ddmagenta}{rgb}{0.8,0,0.8}
\definecolor{orange}{rgb}{0.6,0.2,0}
\definecolor{vgreen}{rgb}{0.1,0.5,0.2}
\definecolor{dred}{rgb}{.8,0,0}
\definecolor{Turk}{rgb}{0,0.7,0.4}

\newcommand{\EEE}{\color{black}}

\newcommand{\semi}{semistable }
\begin{document}


\title{From Adhesive to Brittle Delamination in Visco-Elastodynamics}

\author{Riccarda Rossi
\thanks{
DIMI (Department of Mechanical and Industrial Engineering), Universit\`a degli studi di
  Brescia, via Branze 38, I--25133 Brescia, Italy.
Email: {\ttfamily riccarda.rossi\,@\,unibs.it}
}
\and 
Marita Thomas
\thanks{ Weierstrass Institute for Applied Analysis and
Stochastics, Mohrenstr.~39, 10117 Berlin, Germany. 
Email: {\ttfamily marita.thomas@wias-berlin.de}
}
}

%
%
%

\date{\today}
\maketitle
\begin{abstract}
In this paper we analyze a  system for \emph{brittle delamination} between two visco-elastic bodies, 
also subject  to inertia, which can be interpreted as a model for dynamic fracture.  
The rate-independent flow 
rule for the delamination parameter is coupled with the momentum balance for the displacement, including inertia. 
This model features a nonsmooth constraint ensuring the continuity of the displacements outside the 
crack set, 
which is marked by the support of the delamination parameter.    
A weak solvability concept, generalizing the notion of energetic solution for rate-independent 
systems to the present mixed 
rate-dependent/rate-independent frame, is proposed. Via refined variational convergence techniques, 
existence of solutions is proved 
by passing to the limit in approximating systems which regularize the nonsmooth constraint by conditions 
for adhesive contact. The presence of the inertial term requires the design of suitable recovery spaces 
small enough to provide compactness but large enough to recover the information on the crack set in the limit.   
\end{abstract}
%


%
\section{Introduction}
%
Over the last two decades, crack propagation has 
been intensively studied from a mathematical viewpoint, 
starting from the seminal paper \cite{FraMar98}. This article proposed a variational 
scheme for the dissipative, rate-independent evolution of fracture, coupled with the 
`static' momentum balance for the \emph{purely elastic} displacement variable. 
Several papers, cf.\ e.g.\ \cite{DT02, Cham03, FraLar03, DFT05, DMLaz10}  
(see  also the survey \cite{BFM08repr}), have ever since  consolidated 
 the existence theory, and the study of the fine properties, for the notion 
of \emph{quasistatic evolution} of fracture due to \textsc{G.\ Dal Maso} and coworkers. 
Also  alternative solution notions  have been advanced \cite{Lars10}. 
In this realm, great generality  as far as the modeling of the crack set 
 has been achieved thanks to the toolbox of Geometric Measure Theory.
 \par
 The study of \emph{dynamic fracture}, with the displacement variable subject to viscosity 
and inertia within Kelvin-Voigt rheology, is at a less refined stage. Indeed, 
 phase-field models for (rate-independent) fracture, coupled with elasto-visco-dynamics, 
have been extensively studied in   \cite{BouLarRic,LarOrtSul}, 
where the evolution of  a volume,  damage-like variable approximating the fracture 
is  governed by the so-called Ambrosio-Tortorelli functional \cite{Ambrosio-Tortorelli} 
of Mumford-Shah type. 
However,  the convergence  of the solutions to this regularized system, to solutions 
of a model for brittle 
fracture, has been proved only in the case of \emph{purely} rate-independent evolution 
(i.e., with the static momentum balance), see  \cite{Giac05ATAQ}. 
While the asymptotic analysis to the 
Mumford-Shah fracture regime has also been carried out  for the \emph{gradient flow} 
of the Ambrosio-Tortorelli functional \cite{BabMil}, the passage to the limit in the  
case where the displacements are subject to the equation of  visco-elastodynamics 
remains open. So is, in fact, the study of dynamic fracture without strong geometric 
assumptions on the cracks, essentially due to the challenges posed by the coupling between 
the rate-independent propagation of the fracture, and the rate-dependent evolution of the 
displacement variable. 
 \par
 The basics for the study of the  dynamic case with arbitrary cracks have been  
established in \cite{DMLar11EWED}, focusing on the analysis of the equations of elastodynamics
for the displacement out of the (arbitrarily growing) crack set, whose evolution 
is assumed to be \emph{given}. 
The existence and uniqueness results from  \cite{DMLar11EWED}  
have been recently extended to the case of mixed 
Dirichlet/Neumann boundary conditions  in \cite{DMLuc}. 
Let us stress that in \cite{DMLar11EWED, DMLuc} the crack evolution is  preassigned.  
To our knowledge, existence results for models on dynamic fracture without  
this restriction 
have only recently been obtained in \cite{DMLarToa,DMLazNar}, and these results are strongly 
based on the $1$ or $2$-dimensional geometry of the problem. 
The work \cite{DMLarToa}  
tackles a $2$D-model for dynamic fracture with prescribed, sufficiently smooth,  connected 
crack path, 
but evolving with unknown speed. In this setting, the evolution of the crack is fully described 
by that of the crack-tip. Restricting the problem to a class of sufficiently smooth 
crack-tip evolutions, 
the evolution criterion for the crack is given by a maximal dissipation condition, 
selecting, within this class, 
the crack-tip evolution that runs as fast as possible consistently with the energy balance, 
and thus preventing stationary cracks 
from always being solutions. 
In \cite{DMLazNar} an existence result for a dynamic $1$D-model without pre-assigned crack evolution  
has been proved in the case of  a \emph{dynamic peeling test} for a thin film, 
initially attached to a planar 
rigid substrate.  The authors provide an existence result for a formulation 
of the model consisting of the wave 
equation on a time-dependent domain. The evolution of the debonding front is given    
by the   Griffith criterion  in terms 
 of a suitable notion of  dynamic energy release rate.  . 
Again, their argument strongly relies on the special, one-dimensional geometry of the problem.
\par
\begin{figure}[ht]
\psfrag{O-}{$\Omega_-$}
\psfrag{O+}{$\Omega_+$}
\psfrag{G}{$\GC$}
\psfrag{n}{$\mathbf{n}$}
\centerline{\includegraphics[width=7cm]{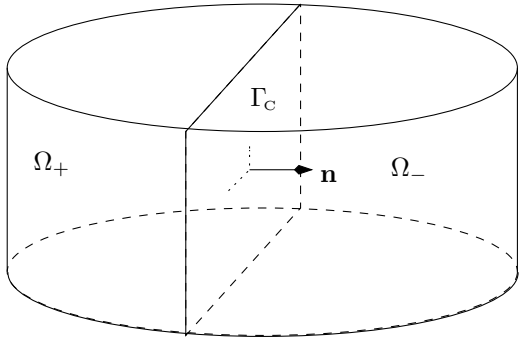}}
\label{Fig:setting}
\caption{A feasible domain $\Omega$ with convex interface $\GC$.}
\end{figure}
In this paper we aim to contribute to the investigation of 
dynamic fracture 
from a yet different perspective. 
We will consider a model describing the  evolution, during a finite time interval $(0,T)$,  
of brittle delamination between two  elastic bodies $\Omega_+$ and $\Omega_-$, 
subject to viscosity and inertia, along a prescribed contact surface $\GC$. 
 Within $\GC$ the crack evolution is not prescribed, but falls into the class of rate-independent evolutions, 
as it is  governed by a unidirectional, positively $1$-homogeneous dissipation potential, 
cf.\ \eqref{diss-en-intro}, 
and a semistability condition, 
cf.\ \eqref{semistab-z-intro}. In our setup the crack set as a subset of $\GC$ need not
be connected and it may even jump with respect to time. Moreover, it should be pointed out that 
our results will be obtained in general space dimension $d\geq 1$. 
We refer to \eqref{ass-domain}  ahead for the precise  statement of our 
conditions on $\Omega $ and $\GC$, cf.\ also Remark \ref{rmk:Gammac}. 
A prototype of feasible domain is the one depicted  in Fig.\ \ref{Fig:setting}.  
\par
Following the approach by \textsc{M.\ Fr\'emond} \cite{FreNed96DGDP,Fre02}, 
within the theory of generalized standard materials \cite{HalNgu75MSG}, 
delamination is described in terms of an internal variable $z: (0,T) \times \GC \to [0,1]$, 
which has in fact the meaning of a damage variable as it describes the fraction of fully 
effective molecular links in the bonding. Namely,
 \begin{equation}
\label{meaning-z}
z(t,x) = \begin{cases}  1 \\ 0  \end{cases}  \text{ means that the bonding is } \  \begin{cases} \text{fully intact} \\ \text{completely broken}   \end{cases} 
\end{equation} 
at the time $t \in (0,T)$, at the material point $x \in \GC$. 
 The rate-independent flow rule for  the delamination parameter 
$z$ is coupled to the \emph{dynamic momentum balance} 
for the displacement field $u: (0,T) \times (\Omega_- {\cup} \Omega_+) \to \R^d$. 
Our model enforces the 
 \begin{equation}
\label{brittleC}
\text{brittle constraint:}\qquad
 \JUMP{u(t)}=0
\quad\text{a.e.\ on } (0,T)\times\supp z(t)
\end{equation}
where $\JUMP{u} = u^+|_{\GC} -   u^-|_{\GC} -$ is the jump of $u$ across $\GC$, 
$ u^{\pm}|_{\GC} $ denoting the traces on $\GC$ 
of the restrictions $u^{\pm}$ of $u$ to $\Omega_\pm$.  
Moreover, $\supp z$ denotes the support of $z\in L^\infty(\GC)$. 
The brittle constraint \eqref{brittleC} ensures the continuity of the displacements ($\JUMP{u(t,x)}=0$) 
in the (closure of the) set of points where (a portion of) the bonding is still active ($z(t,x)>0$), 
and it allows for displacement
jumps only in points $x\in\GC$ where the bonding is completely broken ($z(t,x)=0$). 
In other words, 
\eqref{brittleC} distinguishes between the crack set $\GC\backslash\supp z(t)$, 
where the displacements may jump, and
the complementary set with active bonding, where it imposes a transmission 
condition on the displacements. 
\par
 That is why, the brittle delamination system  can be understood 
as a model for \emph{dynamic fracture}, albeit in a special setting: 
the crack occurs along a prescribed surface, but with \emph{unknown evolution}. 
The main result of this paper states the existence of \emph{energetic}-type solutions, 
obtained by approximation via a model for adhesive contact.  
\par
 Let us now have a closer look at the adhesive contact and brittle delamination systems, 
discuss the analytical difficulties attached to the adhesive-to-brittle limit, 
and illustrate our arguments and results.
 \subsection*{The adhesive contact system}
 The \emph{classical formulation} of the adhesive contact model  
we will at first consider consists of the momentum equation, with viscosity and inertia,   
for the displacement $u$  in the bulk domain, namely 
\begin{subequations}
\label{adh-con-intro}
\begin{equation}
\label{mom-balance-intro}
\varrho \ddot{u} - \mathrm{div} \left(\bbD \dot e + \bbC e  \right)  
= F \quad \text{ in } (0,T) \times ( \Omega_+ {\cup} \Omega_-),
\end{equation}
with  $\varrho>0 $ the (assumed constant, for simplicity) mass density of the body,  $e  = e(u):=\frac{1}{2}(\nabla u+\nabla u^\top)$ the linearized strain 
tensor (throughout the paper, we shall often write $\dot e$ as a short-hand for $e(\dot u)$), 
and $F$  a  time-dependent applied volume force.  
Equation \eqref{mom-balance-intro} is  supplemented with homogeneous Dirichlet 
boundary conditions on the Dirichlet part  $\GD$  of the boundary $\partial\Omega$, where 
$\Omega := \Omega_+ \cup \GC \cup \Omega_-$,  
and subject to an applied traction $f$ on the Neumann part $\GN = \partial\Omega \setminus \GD$, i.e. 
\begin{equation}
\label{bc-intro}
u=0 \quad \text{  on } (0,T) \times \GD,  \qquad \qquad  
\left(\bbD \dot e + \bbC e  \right)|_{\GN} \nu = f  \quad \text{  on } (0,T) \times \GN,
\end{equation} 
with $\nu$  the outward unit normal to $\partial\Omega$.  The evolutions of $u$ and of the delamination 
parameter $z$ from \eqref{meaning-z} are coupled through the following boundary 
condition on the contact surface $\GC$ 
\begin{equation}
\label{cont-surf-intro}
  \left(\bbD \dot e + \bbC e  \right)|_{\GC}  \mathbf{n} + k z \JUMP{u} =0 
\quad \text{ on } (0,T) \times \GC,
\end{equation}
with $\mathbf{n}$ the unit normal to $\GC$ oriented from $\Omega_+$ 
to $\Omega_-$ and $k$  a positive constant: The adhesive-to-brittle limit passage results from letting $k\to\infty$.  In the adhesive contact model, 
the flow rule for $z$ reads
\begin{equation}
\label{adh-flow-rule-intro}
\partial I_{(\infty, 0]}(\dot z ) + \partial\calG(z) -a_0^k - a_1^k \ni - \tfrac{1}2  k \left| \JUMP{u} \right|^2 \quad \text{ on } (0,T) \times \GC. 
\end{equation}
\end{subequations}
In \eqref{adh-flow-rule-intro}, $I_{(\infty, 0]}$ denotes the indicator function of 
the half-line $(-\infty,0]$, by means of which the unidirectionality $\dot z \leq 0$ of 
the debonding phenomenon is imposed, and $\partial I_{(\infty, 0]} $ is its subdifferential 
in the sense of convex analysis. The positive coefficients $a_0^k$ and $a_1^k$,  
which we shall consider depending on the parameter $k$ in view of a discussion of 
different scalings in the adhesive-to-brittle limit $k\to\infty$, are the 
phenomenological specific energies per area stored and, respectively, 
dissipated by disintegrating the adhesive. Finally,
 $\partial\calG$ is the (formally written) subdifferential of the gradient term
\begin{equation}
\label{gradient-regulariz-intro}
\calG_k(z): = \begin{cases}
\mathrm{b}_k |\rmD z |(\GC)  &  \text{ if } z \in \SBV (\GC; \{0,1\}), \\
\infty  & \text{ otherwise,} 
\end{cases}
\end{equation}
where $\mathrm{b}_k>0$ again depends on $k$,   $\SBV (\GC; \{0,1\})$ is the space of the special 
bounded variation functions on $\GC$, taking values in $\{0,1\}$, and $|\rmD z |(\GC) $ denotes the 
variation on $\GC$ of the Radon measure $\rmD z$. Indeed, we are imposing that $z$ only takes the 
values $0$ and $1$, so  that it can be  identified with the characteristic function  of a  set 
$Z\subset\GC$  with finite perimeter $P(Z; \GC) = |\rmD z |(\GC)$.    Thus our  adhesive contact model 
(and the limiting brittle delamination system) accounts for just two states of the bonding 
between $\Omega_+$ and $\Omega_-$, i.e.\ the fully effective and the completely ineffective ones. 
While postponing to  the following lines some comments on the analytical advantages of the 
gradient regularization from \eqref{gradient-regulariz-intro}, let us mention here that, 
the constraint $z \in \{0,1\}$ makes ours akin to a model
for crack propagation (along a prescribed $(d{-}1)$-dimensional interface). 
\par
Due to the expected poor time regularity of the delamination variable $z$, 
the adhesive contact system \eqref{adh-con-intro} has to be weakly formulated in  a 
suitable way, reflecting its  mixed rate-independent/rate-dependent  character. 
For this, we shall  resort to an \emph{energetic}-type  solvability concept, 
generalizing the notion of (global) energetic solution to a purely rate-independent system,  
cf.\ \cite{MieRouBOOK}.  Our notion  was first introduced   in \cite{Roub08}  and 
has been recently analyzed from a  more abstract viewpoint in \cite{RosTho15CEx}.  
We shall recall this solution  concept  in a general and abstract setting in the 
upcoming Definition \ref{def:energetic-sol}; in the specific context of the adhesive contact system, 
we call  a  pair $(u,z)$ with suitable 
 temporal and spatial regularity (cf.\ Def.\ \ref{def:energetic-sol})  
a \emph{\semi energetic solution} to system   \eqref{adh-con-intro}  if it fulfills 
 the weak formulation of the momentum balance 
\begin{equation} \label{weak-mom-intro}
\int_\Omega \varrho \ddot{u}(t) v \dd x 
+ \int_{\Omega\backslash\GC} \left( \mathbb{D}e(\dot{u}(t)) \colon e(v) 
{+} \mathbb{C}e(u(t)):e(v) \right) \,\mathrm{d}x+\int_{\GC} kz\JUMP{u}\JUMP{v}\,\mathrm{d}\Surf  
= \langle\mathbf{f}(t),v\rangle_{H^1(\Omega;\R^d)}
\end{equation}
for almost all $t\in (0,T)$ and for every $v \in H^1(\Omega;\R^d)$ with 
$ v=0$ on $(0,T) \times \GD$ (with $\Surf $ the $(d{-}1)$-dimensional
Hausdorff measure, and  the function $\mathbf{f}:(0,T) \to H^1(\Omega;\R^d)^*$ 
subsuming the bulk force $F$  and the applied traction $f$),  and the weak formulation  
of the flow rule \eqref{adh-flow-rule-intro}. 
The latter is akin to the (global) energetic formulation for rate-independent systems, 
in that it features
\begin{compactitem}
\item[-]
 an  \emph{energy-dissipation  (in)equality}, involving the stored energy 
of the adhesive contact system 
\begin{equation}
\label{stored-energy-adh-intro}
\calE_k(t,u,z):=\int_{\Omega\backslash\GC}\tfrac{1}{2}\mathbb{C}e(u):e(u)
 -   \pairing{}{H^1(\Omega;\R^d)}{\mathbf{f}(t)}{u} 
+\int_{\GC} \left(\tfrac{k}{2}z|\JUMP{u}|^2-a^0_kz \right)
\,\mathrm{d}\Surf+ \calG_k(z)
\end{equation}
and the dissipated energy 
\begin{equation}
\label{diss-en-intro}
\calR_k (\dot z) : = \begin{cases}
\int_{\GC} a_k^1 | \dot z | \dd x & \text{ if } \dot z \leq 0 \text{ a.e.\ in } \GC,
\\
\infty & \text{ otherwise},
\end{cases}
\end{equation}
\item[-] coupled with the \emph{semistability} condition
\begin{equation}
\label{semistab-z-intro} 
\calE_k(t,u(t),z(t)) \leq \calE_k(t,u(t),\tilde z) +
\calR_k(\tilde z {-}z(t)) \qquad \text{for all } \tilde z \in L^1(\GC) \ \ 
\text{for all } t \in [0,T].
 \end{equation} \end{compactitem}
In fact,  \eqref{semistab-z-intro}  reflects  the mixed character of the evolution, 
in that stability is only tested for $z$, while the rate-dependent variable $u$ is kept 
fixed as a solution of \eqref{weak-mom-intro}. 
\par 
The first result of this paper, \underline{\textbf{Theorem \ref{th:exist-mech}}}, 
states the existence of \semi energetic solutions to (the Cauchy problem for) 
the adhesive contact system \eqref{adh-con-intro}, in fact satisfying the 
energy-dissipation \emph{balance} along any interval $[s,t]\subset [0,T]$. 
We shall derive Thm.\   \ref{th:exist-mech} from a general existence result 
for  damped inertial  systems with a mixed rate-dependent/rate-independent character, 
which was proved in \cite{RosTho15CEx}. With Thm.\  \ref{th:exist-mech} we will also provide a 
series of a priori estimates on families of \semi energetic solutions $(u_k,z_k)_k$, uniform with respect 
to the parameter $k\in\N$ and preliminary to the limit passage $k\to\infty$. 
%
\subsection*{The adhesive-to-brittle limit: analytical challenges and our results}
%
The asymptotic analysis as $k\to\infty$ for the \emph{purely rate-independent} 
adhesive contact system, coupling  the flow rule \eqref{adh-flow-rule-intro}  
(with no regularizing gradient term), with the \emph{static} momentum balance, 
was carried out in \cite{RoScZa09QDP} by resorting to the  evolutionary $\Gamma$-convergence 
results for rate-independent processes from \cite{MRS06}. Loosely speaking, 
the main observation is that the adhesive contact contribution  $\int_{\GC}\tfrac{k}{2}z|\JUMP{u}|^2
\,\mathrm{d}\Surf  $ to $\calE_k$ \eqref{stored-energy-adh-intro} 
penalizes displacement jumps in points with positive $z$, and leads 
as $k\to\infty$ to the brittle constraint  
 $z|\JUMP{u}|=0$ a.e.\ in $\GC$, incorporated in the  
$\Gamma$-limit of the  energy functionals $(\calE_k)_k$ (cf.\ \eqref{brittle-energy-brittle-intro}   
below).   
\par
The adhesive-to-brittle asymptotics is remarkably more complicated in the case of  mixed 
rate-in\-de\-pen\-dent/rate-dependent evolution, 
where one  has to pass to the limit separately in the momentum balance \eqref{weak-mom-intro}  
featuring the  semistable delamination variables, 
and in the  semistability inequality \eqref{semistab-z-intro} featuring the solution \
of the momentum balance.   
This problem was tackled in \cite{RosTho12ABDM} for a system also encompassing the 
temperature equation, but without inertia in the momentum balance. Analogous arguments 
were used   in  the \emph{purely rate-independent} case in \cite{RoThPa15SDLS}, 
to address  the adhesive-to-brittle limit 
combined with time discretization and leading to \emph{local solutions} 
(in the sense  of \cite[Sec.\ 3]{RouACVB13}), of the brittle delamination system.  
\par
In what follows, we will   illustrate these analytical difficulties and hint 
at our methods, which could in fact be adapted to handle the coupling  with the temperature 
equation, as well. We have however chosen to confine our analysis to the isothermal case, 
in order to highlight the techniques specifically developed  in the present paper 
to deal with inertia in the momentum balance.
\par
The very first problem  is due to the 
\begin{enumerate}
\item[(1)]
blow-up of the bounds on the adhesive contact term $ kz\JUMP{u}$ 
in \eqref{weak-mom-intro}  as $k\to\infty$.
\end{enumerate}
This reflects the fact that, for the limiting brittle system the momentum balance has 
to be tested with  test functions encompassing the brittle constraint \eqref{brittleC}, 
which will be  satisfied by the limiting displacement $u$. We will in fact prove  
that any pair $(u,z)$, arising in the limit as $k\to\infty$ of a sequence of \semi 
energetic solutions $(u_k,z_k)_k$ of the adhesive contact system \eqref{adh-con-intro}, 
$k\in\N,$ complies with 
\begin{equation}
\label{weak-mom-brittle-intro}
\begin{aligned}
\int_\Omega \varrho \ddot{u}(t) v \dd x   + \int_{\Omega\backslash\GC}  
& \left( \mathbb{D}e(\dot{u}(t)) \colon e(v) {+} \mathbb{C}e(u(t)):e(v) \right) \,\mathrm{d}x   
= \langle\mathbf{f}(t),v\rangle_{H^1(\Omega;\R^d)}
\\ &
\text{for all } v \in  H^1(\Omega;\R^d) \text{ with } v=0 \text{ on } \GD  
\text{ and } \JUMP{v}=0 \text{ on } \supp z(t)\subset \GC
\end{aligned}
\end{equation}
and for  almost all $t \in (0,T)$. 
In order to obtain  \eqref{weak-mom-brittle-intro}, 
we shall resort to the arguments from \cite{RosTho12ABDM} and provide for every 
admissible test function $v$ for \eqref{weak-mom-brittle-intro} a \emph{recovery sequence}  
$(v_k)_k$, suitably converging to $v$, 
fulfilling the brittle constraint \eqref{brittleC} already at level $k\in\N,$ i.e.
\begin{equation}
\label{recovery-constraint}
 \JUMP{v_k} =0 \quad \text{ a.e.\ on }\supp z_k(t)\,, 
\end{equation}
 with $t \in (0,T)$ fixed out of a negligible set. This will allow us to bypass problem (1).  
 The key ingredient in the construction of the sequence $(v_k)_k$, 
starting from a  test function $v$ such that $\JUMP{v}=0$ on $\supp z(t)$, is a relation 
between the supports of the approximate, semistable delamination variables $z_k$, and the support of 
the semistable limit $z$. This is provided by the property 
 of  \emph{support convergence} \begin{equation}
\label{suppconv-intro}
 \supp z_k(t)\subset \supp z(t)+B_{\rho(k,t)}(0)
\quad\text{and}\quad
\rho(k,t)\to0\text{ as }k\to\infty,
\end{equation} 
that was proved in  \cite{RosTho12ABDM} via arguments from geometric measure theory. 
In turn, these arguments  heavily rely on the fact that the delamination 
variables $z_k $ take value in  $\{0,1\}$, and on the regularizing perimeter 
term from \eqref{gradient-regulariz-intro} contributing to the energy functional 
\eqref{stored-energy-adh-intro} driving the adhesive contact system. 
\par
In \cite{RosTho12ABDM}, addressing the case without inertia, the above arguments were 
sufficient to pass to the limit in the momentum balance \eqref{weak-mom-intro},  
tested with the recovery test functions $v_k$ complying with \eqref{recovery-constraint}. 
In the present case, we have to face  an  additional difficulty, 
clearly related to problem (1), 
 namely  the 
\begin{enumerate}
\item[(2)] blow-up  as $k\to\infty$ of the estimates (by comparison) 
on the inertial terms $\ddot{u}_k$ in \eqref{weak-mom-intro}.
\end{enumerate}
We will overcome this by a careful refinement of the method from  \cite{RosTho12ABDM}. 
This will lead  us to construct a sequence of \emph{recovery spaces} for the space of test functions in the weak momentum balance \eqref{weak-mom-brittle-intro} for the brittle system. The crucial point will then be to observe that the terms $(\ddot{u}_k)_k$ are in fact suitably estimated in these spaces, which will allow for compactness arguments and, ultimately, the limit passage in \eqref{weak-mom-intro}. The limit passage in the energy-dissipation 
inequality for the adhesive contact system will  essentially follow from lower semicontinuity, 
while for the semistability condition we will make use of the by now standard 
\emph{mutual recovery sequence} 
argument from \cite{MRS06}.
\par
 In this way we will obtain  the  \textbf{main result of our paper},  
\underline{\textbf{Theorem \ref{thm:main}}}, 
stating the convergence of \semi energetic solutions to the adhesive contact 
systems to a \semi energetic solution of the brittle one, fulfilling 
\begin{compactitem}
\item[-] the weak momentum balance \eqref{weak-mom-brittle-intro},
\item[-] the energy-dissipation (in)equality, 
\item[-] the semistability condition.
 \end{compactitem}
 The latter  two relations feature the dissipation potential
 $\calR_\infty$  arising in the  limit of the energies 
$(\calR_k)_k$ from  \eqref{diss-en-intro}, and the 
  energy functional  \begin{equation}
\label{brittle-energy-brittle-intro}
\calE_\infty(t,u,z):=\int_{\Omega\backslash\GC}\tfrac{1}{2}\mathbb{C}e(u):e(u)
 -   \pairing{}{H^1(\Omega;\R^d)}{\mathbf{f}(t)}{u} 
+\int_{\GC} \left(J_\infty (\JUMP{u},z) -a^0_kz \right)
\,\mathrm{d}\Surf+ \calG_\infty(z)
\end{equation}
with $ J_\infty (\JUMP{u},z) $ the indicator function of the 
 brittle constraint \eqref{brittleC}, i.e.\ $ J_\infty (\JUMP{u},z)=0 $ 
if \eqref{brittleC} is satisfied and 
$ J_\infty (\JUMP{u},z)=\infty$ otherwise, and  $\calG_\infty$ the $\Gamma$-limit 
as $k\to\infty$ of the perimeter  energies $(\calG_k)_k$ 
from \eqref{gradient-regulariz-intro}.
Let us stress that, adapting some arguments from   \cite{DMLar11EWED},  
we shall prove that along \semi energetic solutions of the brittle system, 
the energy-dissipation inequality  actually holds as a balance, along any arbitrary interval 
$[s,t] \subset [0,T]$ for almost all $ s\leq t \in (0,T)$, and for $s=0$. 
\par
 Let us finally mention that our ansatz for $\calG_k$ 
and $\calR_k,$ cf.\ \eqref{gradient-regulariz-intro} 
and \eqref{diss-en-intro}, 
will allow for different scalings of the parameters $a_k^0$,  $a_k^1$,  
and $\mathrm{b}_k,$ cf.\ \eqref{coeffscale}. 
In this way, we can obtain different fracture models in the brittle limit. 
We will discuss the different options in Section \ref{MainResult}. 
\paragraph{\bf Plan of the paper.}
In Section \ref{s:2} we give our weak solvability notion for damped inertial  systems  with a 
mixed rate-independent/rate-dependent character. 
In particular,  in  Sec.\ \ref{ss:2.1} we  specify  it  in the context of the adhesive 
contact model and then state the existence of \semi energetic solutions to the adhesive system. 
In Sec.\ \ref{ss:2.3} we give the notion of \semi energetic solutions to the brittle system, while in Sec.\ 
\ref{MainResult} we present our main result, Theorem \ref{thm:main}, 
which  
 provides the existence of \semi energetic solutions for the 
brittle model in terms of an approximation result via the adhesive contact  systems. 
We also compare our result with other existing results on dynamic fracture. 
\par
The existence of \semi energetic solutions to the adhesive contact system is proved 
in Section \ref{s:3}, while the proof of Theorem \ref{thm:main} is carried out in Section \ref{s:4}. 
%
\section{Setup, solution concepts for the adhesive and brittle problems,  and preliminary results}
\label{s:2}
%
We start by fixing some general notation that will be used throughout the paper. 
 \begin{notation}\upshape
 We will denote by   $\norm{\cdot}{X}$   the norm of a 
Banach space   $X$,  and by $\pairing{}{X}{\cdot}{\cdot}$
the duality pairing between $X^*$ and $X$. If $X$ is a Hilbert space,
its inner product shall be denoted by $\inner{X}{\cdot}{\cdot}$. The symbols
(1) $\mathrm{B}([0,T]; X)$,  (2) $ \BV[0,T]; X)$, (3)
$ \mathrm{C}^0_{\mathrm{weak}}([0,T]; X)$ shall denote the spaces of functions with values 
in $X$ that are 
 defined at every $t\in [0,T]$ and are (1) bounded and measurable, (2) with bounded variation, (3)
  continuous with respect to the weak topology, respectively. 
\par
Moreover,   we shall often denote by the symbols
$c,\,\tilde{c},\, C,\,\tilde{C}$  various positive constants, whose meaning may vary from line to line,  depending only on
known quantities. 
\end{notation} 
  \paragraph{\bf Setup \& \semi energetic solutions for  damped inertial systems.}
\noindent Let us now specify the concept of  \emph{abstract damped inertial system}, and the 
associated notion of \emph{\semi energetic solution}, that will later apply both to the 
adhesive contact, and to the brittle systems.  We draw the following definitions 
from \cite{RosTho15CEx}, where 
the \semi energetic solution concept,  originally introduced in 
\cite{Roub08} for a class of  mixed rate-dependent/rate-independent systems 
in continuum  mechanics, was generalized to an abstract setting. 
Let us mention that, in  \cite{RosTho15CEx}
 a fairly broad class of  damped inertial  systems was tackled, in particular encompassing 
a dissipation potential $\dissu$ with general superlinear growth at infinity, 
and a non-convex (but still with appropriate properties)  dependence
$u\mapsto \ene tuz$. However, in view of the target adhesive contact and brittle systems, 
it will be sufficient to confine the discussion to a \emph{quadratic} dissipation potential, 
and to the case the  mapping    
$u\mapsto \ene tuz$ is \emph{convex}. 
 \par
In what follows, we will consider an 
abstract  damped inertial system
$\ingrsys$ to be  
given by:
\begin{subequations}
\label{basic-conds}
\begin{compactitem}
\item two Hilbert spaces 
\begin{equation}
\label{spaces-V-W}
\Spu  \text{ and } \Spw, \text{ $\Spw$ identified with its dual $\Spw^*$, such that  } 
\Spu \Subset \Spw 
\text{ compactly and densely},
  \end{equation}
  so that 
  $\Spu \subset \Spw = \Spw^* \subset \Spv^*$ continuously and densely, and $\pairing{}{\Spu}{w}{u} =
(w,u)_{\Spw} $ for all $ u\in \Spu  $  and $ w \in
\Spw$; 
\item a separable  Banach space $\Spz$;
\item a dissipation potential $\calV: \Spu \to [0,\infty)$ of the form 
\begin{equation}
\label{diss-V}
\dissu(v) = \tfrac12 a(v,v) \qquad \text{with } a : \Spu \times \Spu \to \R \text{ a continuous coercive bilinear form};
\end{equation}
\item a dissipation potential  $ \calR:\Spz\to[0,\infty]$, with domain 
$\mathrm{dom}(\calR)$,  lower semicontinuous, convex,  positively $1$-homogeneous and coercive i.e., 
\begin{equation}
\label{diss-R}
\begin{aligned}
&
\calR(\lambda \zeta)=\lambda\calR(\zeta)\quad \text{ for all }\zeta\in\Spz\text{ and }\lambda\geq0,
\\
&
\exists\, C_R>0 \ \ \forall\, \zeta \in \Spz \qquad \dissr(\zeta) \geq C_R \norm{\zeta}{\Spz};
\end{aligned}
\end{equation}
\item 
a \emph{kinetic energy} $\calK:\Spw\to[0,\infty),$ $\calK(v):=\frac{1}{2}\|v\|_\Spw^2,$ 
\item an \emph{energy functional} $\calE : [0,T] \times \Spv \times \Spz \to \R \cup \{\infty\}$,  
with proper domain 
$\mathrm{dom}(\calE)=[0,T]\times\mathbf{D}_u\times\mathbf{D}_z$, such that 
\begin{equation}
\label{basic-ass-E}
\begin{aligned}
&
\text{
$t\mapsto \calE(t,u,z)$ is differentiable} && \text{for all } (u,z)  \in \mathbf{D}_u\times\mathbf{D}_z,
\\
&
\text{$(u,z)\mapsto \ene tuz$ is lower semicontinuous} && \text{for all } t \in [0,T],
\\
& 
\text{$u \mapsto \ene tuz$  is convex} && \text{for all } (t,z)\in [0,T] \times\mathbf{D}_z\,.
\end{aligned}
\end{equation}
\end{compactitem}
In what follows, we shall denote by 
$\partial_u \calE : [0,T] \times \Spu \times \Spz \rightrightarrows \Spu^*
$  the  subdifferential  of the functional $\calE(t, \cdot, z)$ in the sense of convex analysis.  
We  postpone  to Section \ref{s:3} ahead the precise statement of
the  further conditions on $\calE$ required in the existence result from   \cite{RosTho15CEx} 
that we shall apply to deduce the existence of solutions to the adhesive contact system. 
Let us only  mention here that 
the assumptions on $z\mapsto \ene tuz$ (cf.\  the coercivity requirement
\eqref{assDisscoercE} ahead)
 also involve a second space 
$\Spx$ such that 
\begin{equation}
\label{sp-X-Z}
\Spx \text{ is the dual of a separable Banach space  and } \Spx \Subset \Spz \text{ compactly}. 
\end{equation}
\end{subequations}
\par
We are now in the position to state precisely the \semi energetic 
solution concept for  the damped inertial system  $\ingrsys$, which has been developed  in 
 \cite[Def.\ 3.1]{RosTho15CEx}
 based on a time-discrete scheme with \EEE alternating (decoupled) \EEE minimization
w.r.t.\ the variables $u$ and $z$. 
\begin{definition}[Semistable  energetic solution]
\label{def:energetic-sol}
We call a pair $(u,z) : [0,T] \to \Spu \times \Spz$  a \emph{\semi energetic solution} 
to the  damped inertial  system $\ingrsys$ if
\begin{subequations}
\label{regularity}
\begin{align}
&
\label{reg-u}
u \in 
  W^{1,1}(0,T;\Spu)\,, \quad  \dot u \in L^\infty (0,T; \Spw)\EEE\,,
 \quad \ddot u\in L^2(0,T;\Spu^*)\,,
\\
&
\label{reg-z}
z \in  
\mathrm{B}([0,T];\Spx) \cap  \mathrm{BV}([0,T];\Spz)
\end{align}
\end{subequations}
fulfill the
\begin{compactitem}
\item[-] the subdifferential inclusion
\begin{equation}
\label{eq-u}
\ddot u(t) + \partial\calV(\dot u(t)) + \partial_u\calE(t,u(t),z(t)) \ni 0 
\qquad \text{in } \Spu^* \quad \text{ for a.a.\ } t \in (0,T), 
\end{equation} 
\item[-] the  semistability condition 
\begin{equation}
\label{semistab-z} 
\calE(t,u(t),z(t)) \leq \calE(t,u(t),\tilde z) +
\calR(\tilde z {-}z(t)) \qquad \text{for all } \tilde z \in \Spz \
\text{for all } t \in [0,T];
 \end{equation}
\item[-]  the energy-dissipation  inequality 
 \begin{equation}
 \label{enineq}
 \begin{aligned}
\tfrac{1}{2}\|\dot u(t)\|_{\Spw}^2  & 
+ \int_0^t  2\calV(\dot u(s)) 
\,\mathrm{d} s + \Var_{\calR}(z, [0,t])+
 \calE(t,u(t),z(t))\\ 
&  \leq \tfrac{1}{2} \|\dot u(0)\|_{\Spw}^2 
+  \calE(0,u(0),z(0)) + \int_0^t \partial_t\calE(s,u(s),z(s))\,\mathrm{d}s 
\qquad \text{for all } t \in [0,T]\,,
 \end{aligned}
 \end{equation}
 with $\xi(s)$ a selection  in $\partial_u\calE(s,u(s),z(s))  $ 
fulfilling \eqref{eq-u}  for almost all $s\in (0,T)$  and 
 $\Var_{\calR}$ the notion of total variation induced by the dissipation potential $\calR$, i.e. 
\[
\Var_{\dissr}(z; [s,t]) := \sup\left\{ \sum_{j=1}^{N}
\dissr(z(r_j) - z(r_{j-1}))\, : \quad s= r_0<r_1<\ldots<r_{N-1}<r_N=t
\right\} 
\] for a given 
 subinterval $[s,t]\subset [0,T]$.
\end{compactitem}
\end{definition}
\begin{remark}[The energy-dissipation balance]
\upshape
In fact, for the adhesive contact system \eqref{adh-con-intro} 
we will prove in Thm.\ \ref{th:exist-mech} the existence of \semi energetic solutions fulfilling 
the energy-dissipation \emph{balance} along any interval $[s,t]\subset [0,T]$. Also for 
the brittle system, in Thm.\ \ref{thm:main}, 
we will show that any \semi energetic solution in fact complies with the energy-dissipation balance 
in any  interval $[s,t]\subset [0,T]$, \emph{for almost all} $s\leq t \in (0,T)$ 
and for $s=0$. 
\end{remark}
\paragraph{\bf Basic assumptions.}
Before specifying  the above notions in the context of  
the adhesive contact and brittle systems, 
let us establish some  basic conditions on the domains $\Omega$ and $\GC$, and on 
the problem data, in common to the adhesive and brittle models. 
\\
\textbf{Assumptions on the reference domain: } 
We suppose that
\begin{subequations}
\label{ass-domain}
\begin{eqnarray}
\label{ass-domain1}
\hspace*{-3em}
&&    
\text{$\Omega\subset\R^d$, $d \geq 2$, is
bounded, $\Omegatwo,\, \Omegaone,\, \Omega$ are Lipschitz domains,
$\Omegaone \cap \Omegatwo=\emptyset$}\,,
\\
\label{ass-domain2}
\hspace*{-3em}
&&
\partial \Omega = \GD\cup \GN,
\text{ $\GD,\,\GN$ open subsets in $\partial\Omega$,} \
\\
\label{ass-domain2+}
\hspace*{-3em}
&&\GD\cap \GN=\emptyset, \ \barGD\cap\barGC=\emptyset, \ 
\mathcal{H}^{d-1}(\GD\cap\barOmegaone)>0\,,\
\mathcal{H}^{d-1}(\GD\cap\barOmegatwo)>0\,,
\\
\label{flatness}
\hspace*{-3em}
&& \hspace*{-.4em}
\begin{array}{l}
\text{$\GC=\barOmegaone\cap\barOmegatwo\subset\R^{d-1}$ is a 
convex ``flat'' surface, i.e.\ contained in a hyperplane of }\R^d,\\
\hspace{10em}\text{such that, in particular, }\;
\mathcal{H}^{d-1}(\GC) = \mathcal{L}^{d-1}(\GC) >0\,,
\end{array}
\end{eqnarray}
where $\mathcal{H}^{d-1},$ resp.\ $\mathcal{L}^{d-1},$ denotes 
the $(d{-}1)$-dimensional Hausdorff, resp.\ Lebesgue measure.
\end{subequations}
 In what follows, we will use the notation
\[
H^1_\mathrm{D}(\Omega\backslash\GC;\R^d):= \{ v \in H^1(\Omega\backslash\GC;\R^d)\, : \ v=0\text{ a.e.\ on }\GD\}\,.
\] 
\begin{remark}
\label{rmk:Gammac}
\upshape
The condition that $\GC$ is contained in a hyperplane has no substantial 
role for our analysis but to simplify  arguments and  notation. 
Instead, the convexity of $\GC$ is essential for  the proof of the adhesive-to-brittle 
limit passage (whereas it is not needed in the analysis of the adhesive contact system). 
Indeed,  it is at the basis of a \emph{uniform relative isoperimetric inequality}  
from \cite[Thm.\ 3.2]{Tho13UPSI}, which in turn is the basic ingredient for the proof of 
the support convergence \eqref{suppconv-intro}, cf.\ also Sec.\ \ref{FineProps}. 
\end{remark}  
\noindent
\textbf{Assumptions on the given data: }  
For  the tensors $\mathbb{C},\mathbb{D}\in\R^{d\times d\times d\times d}$ 
and the function $\mathbf{f}$ in \eqref{weak-mom-brittle-intro}, we require that  
\begin{subequations}
\label{assdata} 
\begin{equation}
\label{assCD} 
\begin{split}
&\mathbb{C},\mathbb{D}\in\R^{d\times d\times d\times d}
\text{ are symmetric and positive definite, i.e., }\\
&\exists\,C_\mathbb{C}^1,C_\mathbb{C}^2,C_\mathbb{D}^1,C_\mathbb{D}^2>0,\forall\,
\eta\in\R^{d\times d}:\;
C_\mathbb{C}^1|\eta|^2\leq\eta:\mathbb{C}\eta\leq C_\mathbb{C}^2|\eta|^2
\;\text{and}\;
C_\mathbb{D}^1|\eta|^2\leq\eta:\mathbb{D}\eta\leq C_\mathbb{D}^2|\eta|^2\,,
\end{split} 
\end{equation}
\begin{equation}
\label{assRegf}
\mathbf{f}\in \mathrm{C}^1([0,T];\Spu^*)\text{ and }
\sup_{t\in[0,T]}\big(\|\mathbf{f}(t)\|_{\Spu^*}
+\|\dot{\mathbf{f}}(t)\|_{\Spu^*}\big)\leq C_\mathbf{f}\,.
\end{equation}
\end{subequations}
Moreover, to keep notation and arguments simple, we prescribe \emph{homogeneous Dirichlet data} on $\GD$. 
%
\subsection{Semistable  energetic solutions  to the evolutionary adhesive contact\\ system}
\label{ss:2.1}
%
The adhesive contact evolutionary system falls within the class of  damped inertial  systems, with 
the following choices of 
\par
\noindent
\textbf{Function spaces: } \\[-8mm]
\begin{subequations}
\label{function-spaces}
\begin{eqnarray}
\Spu & 
&  H^1_\mathrm{D}(\Omega\backslash\GC;\R^d)\,, 
\\
\Spw&=&L^2(\Omega;\R^d)\;\text{ endowed with the norm }
\|v\|_\Spw:=\Big(\int_\Omega\varrho|v|^2\,\mathrm{d}x\Big)^{1/2}\,,\\
\Spz&=&L^1(\GC)\,,\\
\Spx&=&\SBV(\GC;\{0,1\})\,,
\end{eqnarray}
\end{subequations}
 where the space $\Spx$ is related to the perimeter regularizing term contributing 
to the energy functional $\calE_k$, cf.\ \eqref{defEk} below. 
 Observe that, due to the positivity and boundedness of the mass density $\varrho,$ 
the space $\Spw$ is identified with $L^2(\Omega;\R^d)$. \\
\par
\noindent
\textbf{Dissipation potentials and energy   functionals for the adhesive case, $k\in\N$: }
For each $k\in\N$ the adhesive systems $\ingrsysk$ 
are governed by the functionals corresponding to  the  kinetic energy $\calK,$ 
the viscous dissipation $\calV,$ the rate-independent 
dissipation $\calR_k,$  and the mechanical energy $\calE_k$ 
defined as follows:
\begin{eqnarray}
\label{defEkin}
&&\calK(\dot u):=\tfrac{1}{2}\|\dot u\|^2_\Spw\,,\\
\label{defV}
&&\calV:\Spu\to[0,\infty)\,,\;
\calV(\dot u):=\int_{\Omega\backslash\GC}\tfrac{1}{2}\mathbb{D} e(\dot u): e(\dot u)\,\mathrm{d}x\,,  \\
\label{defRk}
&&\calR_k:\Spz\to[0,\infty]\,,\;
\calR_k(\dot z):=\int_{\GC} R_k(\dot z)\,\mathrm{d}\Surf\,,\;
R_k(\dot z):=\left\{
\begin{array}{ll}
a_k^1|\dot z|&\text{if }\dot z\leq 0\,,\\
\infty&\text{otw.}\,;  
\end{array}
\right.
\\
\nonumber
&&\calE_k:[0,T]\times\Spu\times\Spz\to\R\cup\{\infty\}, \;\text{ defined for all }t\in[0,T]:\\
\label{defEk}
&&\calE_k(t,u,z):=\left\{
\begin{array}{ll}
\widetilde\calE_k(t,u,z) +\calJ_k (u,z) &
\text{if }(u,z)\in\Spv\times\Spx\,,\\
\infty&\text{otw.}
\end{array}
\right.\qquad\text{with }\\
\nonumber
&&\widetilde\calE_k(t,u,z):=\int_{\Omega\backslash\GC}\tfrac{1}{2}\mathbb{C}e(u):e(u) \dd x  
- \pairing{}{\Spv}{\mathbf{f}(t)}{u} 
+\int_{\GC}\left( I_{[0,1]}(z)
{-} a^0_kz \right)  
\,\mathrm{d}\Surf+\mathrm{b}_kP(Z,\GC)\,,  \text{ and }
\\
\label{JKfunc} 
&& 
\calJ_k (u,z)   : = \int_{\GC} \tfrac{k}2 z \left| \JUMP{u}\right|^2 \dd \Surf\,. 
\end{eqnarray}
As already mentioned in the Introduction, hereafter 
$Z$ shall denote a subset of $\GC$ with finite perimeter 
$P(Z,\GC)$ in $\GC$ such that  $z=\chi_Z\in\{0,1\}$ is  its characteristic function. 
\par 
Observe that, for every $k\in \N$ the functional 
$u\mapsto \calE_k(t,u,z)$ is G\^ateaux-differentiable, in addition to being convex. 
Therefore, at every $(t,u,z) \in [0,T] \times \mathbf{D}_u\times\mathbf{D}_z
$ its subdifferential
$\partial_u \calE_k(t,u,z)$ reduces to a singleton, whose unique element is  still denoted 
by $\partial_u \calE_k$ with a slight abuse of notation,  and given  
for all $v \in \Spu$ by  
\begin{equation}
\label{GderEk}
\begin{aligned}
& 
\langle\partial_u\calE_k(t,u,z),v\rangle_\Spu=
\int_{\Omega\backslash\GC}\mathbb{C}e(u):e(v)\,\mathrm{d}x
-\langle\mathbf{f}(t),v\rangle_\Spu  
+  \langle\partial_u\calJ_k(u,z),v\rangle_\Spu,  \ \text{with }   
\\ &  \langle\partial_u\calJ_k(u,z),v\rangle_\Spu=
\int_{\GC} kz\JUMP{u}\JUMP{v}\,\mathrm{d}\Surf  \,.  
\end{aligned}
\end{equation}
Therefore, taking into account of the form \eqref{defV} of the 
dissipation potential $\calV$, the subdifferential inclusion  \eqref{eq-u}  
yields the momentum equation \eqref{weak-mom-intro}. 
\noindent
Moreover, let us point out that our 
analysis will cover the following two cases for the coefficients
\begin{subequations}
\label{coeffscale} 
\begin{eqnarray}
&&
\label{coeffscale-a}
a_k^0=a^0=\mathrm{const.}\,,\; a_k^1=a^1=\mathrm{const.}\,,\; 
\mathrm{b}_k=\mathrm{b}=\mathrm{const.}\,,\quad\text{or}\\
&&
\label{coeffscale-b}
a_k^0=\tfrac{a^0}{k}\,,\; a_k^1=\tfrac{a^1}{k}\,,\; 
\mathrm{b}_k=\tfrac{\mathrm{b}}{k}\,.
\end{eqnarray}
\end{subequations}
 We postpone to Sec.\ \ref{MainResult} 
a  discussion of the scaling \eqref{coeffscale-b}. 
Let us only mention here that, since with the scaling \eqref{coeffscale-b} 
for the coefficients $\mathrm{b}_k$, the constraint $z\in \{0,1\}$ 
is no longer ensured in the brittle limit $k\to\infty$, the term $ I_{[0,1]}$ 
contributing to $\calE_k$ (and to $\calE_\infty,$  cf.\  \eqref{defEinfty} below) 
has the role to enforce, for $k\in \N\cup\{\infty\}$, that $z$ 
takes values in $[0,1]$. This is crucial not only for  
the physical consistency of the model, but also for technical reasons 
related to the construction of the recovery sequence for the limit passage as $k\to\infty$ 
in the semistability condition. 
\par
The existence of \emph{\semi energetic solutions}  $(u_k,z_k)$
to the  evolutionary  adhesive contact systems $\ingrsysk$ 
will be deduced from the abstract 
results of \cite{RosTho15CEx} 
in Section \ref{s:3} ahead, 
where we will also derive estimates \eqref{unifbdsk} \& \eqref{kdepbd} on the functions  
$(u_k,z_k)_k$. Let us  mention in advance that, independently from  
the bound \eqref{unifbdDiss} on the total variation of $z_k$ 
induced by $\calR_k$, we also need to derive an estimate 
for $z_k$ in $\BV ([0,T];L^1(\GC))$, due to the possible 
degeneracy of the coefficients $a_k^1$ when  scaled as in \eqref{coeffscale-b}. 
\begin{theorem}[Existence of \semi energetic solutions for $k\in\N$ fixed, uniform bounds, uniqueness]
\label{th:exist-mech}
Assume \eqref{ass-domain}--\eqref{assdata}. 
For each $k\in\N$   the  damped inertial  system $\ingrsysk$ 
admits a \semi energetic solution  
$(u_k,z_k)$ in the sense of Def.\ \ref{def:energetic-sol} 
starting from initial data  $(u_0,u_1,z_0)\in\Spu  \times\Spw  \times\Spz$
 fulfilling the semistability \eqref{semistab-z} at $t=0$ with $\calE_k$ 
and $\calR_k$, cf.\ \eqref{init-semistab}.
 \par
In addition, for every $k\in\N$ the energy-dissipation inequality even holds as an 
equality along any interval 
$[s,t]\subset[0,T]$:
\begin{equation}
 \label{eneq}
 \begin{aligned}
\tfrac{1}{2}\|\dot u_k(t)\|_{\Spw}^2  & 
+ \int_s^t  2\calV(\dot u_k(\tau)) 
\,\mathrm{d} \tau + \Var_{\calR_k}(z_k, [s,t])+
 \calE_k(t,u_k(t),z_k(t))\\ 
&  = \tfrac{1}{2} \|\dot u_k(s)\|_{\Spw}^2  
+  \calE_k(s,u_k(s),z_k(s)) + \int_s^t \partial_t\calE_k(\tau,u_k(\tau),z_k(\tau))\,\mathrm{d}\tau \,.
 \end{aligned}
 \end{equation}
Furthermore, for a given $z\in L^\infty(0,T;\SBV(\GC;\{0,1\}))$,  if 
 $(u,z)$ and $(\tilde u,z)$ both satisfy the adhesive momentum balance  
with the same initial data $u_0$ and $u_1$, 
then $\tilde u=u$.   
\par 
Finally, there exists a constant $C>0$, only depending on the initial data 
$(u_0,u_1,z_0)$ and on the given data, such that the functions  
 $(u_k,z_k)_k$ satisfy the following 
 bounds,  uniform  in  $k\in\N$: 
\begin{subequations}
\label{unifbdsk}
\begin{eqnarray}
\label{unifbdE}
\sup_{t\in[0,T]}\big(\calE_k(t,u_k(t),z_k(t))+\partial_t\calE_k(t,u_k,z_k)\big)&\leq&C\,,\\
\label{unifbdDiss}
\int_0^T\calV(\dot u_k(s))\,\mathrm{d}s+\Var_{\calR_k}(z_k,[0,T])&\leq&C\,,\\
\label{unifbdu}
 \|u_k\|_{H^1(0,T;\Spu)} + \| \dot{u}_k \|_{L^\infty (0,T;\Spw)} 
&\leq&C\,,\\
\label{unifbdzDiss}
\| z_k\|_{\BV([0,T]; L^1(\GC))} 
 &\leq&C\,,\\
\label{unifbdz}
\sup_{t\in[0,T]}\big(P(Z_k(t),\GC)+\|z_k(t)\|_{L^\infty(\GC)}\big)&\leq&C\,,\\
\label{unifbdcombi}
\|\ddot u_k+ \partial_u\calJ_k(\cdot,u_k,z_k)  \|_{L^2(0,T;\Spu^*)}&\leq& C\,.
\end{eqnarray}
\end{subequations}
Furthermore, $(u_k,z_k)$ satisfy the following $k$-dependent bounds for a.a.\ $t\in(0,T)$: 
\begin{subequations}
\label{kdepbd}
\begin{eqnarray}
\label{kdepsubgr}
\exists\, c,C>0,\,\forall\,k\in\N:\quad
& \| \partial_u \calJ_k (u_k,z_k) \|_{L^2(0,T;\Spu^*)} &  \leq    \sqrt{k}C + c\,,\\
\label{kdepddu}
\exists\, \tilde c,\widetilde C>0,\,\forall\,k\in\N:\quad
&  \|\ddot u_k\|_{L^2(0,T;\Spu^*)}  &\leq\sqrt{k}\widetilde C+\tilde c\,.
\end{eqnarray}
\end{subequations}
\end{theorem}
\begin{remark}[The non-penetration condition]
\upshape
\label{rmk:no-nonpen}
Observe that the adhesive contact model so far considered does not include the non-penetration 
constraint ensuring that the two parts of the body, 
$\Omega_-$ and $\Omega_+$, 
 cannot interpenetrate along the contact surface $\GC$, namely
 \begin{equation}
 \label{non-penetration}
 \JUMP{u} \cdot \mathbf{n} \geq 0,
 \end{equation}
 with $\mathbf{n} $  the unit normal to $\Gamma$ oriented from $\Omega_+$ to $\Omega_-$. 
Condition \eqref{non-penetration} would be rendered by 
  an additional  contribution to the energy functional $\calE_k$ \eqref{defEk} of the form 
  $\int_{\GC} I_{K}(\JUMP{u}) \dd \Surf$, with 
   $ I_{K}(\JUMP{u})  = I_{K(x)} ((\JUMP{u(x)}))$ for 
   $\Surf$-a.a.\ $x\in \GC$ and
    and $I_{K(x)} $ the indicator function of the cone 
$K(x): = \{ v\in \R^d\, : \ v \cdot \mathrm{n}(x) \geq 0 \}$, and it 
would give rise to the so-called \emph{Signorini conditions} on the contact surface.
\par
Indeed, even in the realm of adhesive contact,  the  analysis of 
the momentum balance equation with inertia and Signorini conditions poses remarkable challenges: 
in particular,  the existence 
 of  solutions  to the dynamic problem for unilateral contact, possibly complying with an  energy 
balance, seems to be
  an open problem, in the case of bounded domains (whereas for unbounded domains we refer to 
the results in 
  \cite{Pe-Scha1, Pe-Scha2}). 
  \par
  Very recently, 
  a technique based  on duality methods 
in Sobolev-Bochner spaces has emerged in \cite{Sca-Schi}, leading to existence results for    
a suitable weak notion solution (with the energy balance  still missing, though). 
Possibly  relying on this approach,  we intend to address the adhesive-to-brittle 
limit in the dynamic case 
with Signorini conditions 
in a forthcoming study.
\end{remark}
\par 
\begin{remark}[Alternatives to the  non-penetration condition]
\label{rmk:alternative}
\upshape
 The existence result  from Theorem \ref{th:exist-mech} can be extended to the case  where 
suitable boundary conditions are imposed 
on the contact surface $\GC$, \emph{alternative} to the non-penetration constraint
\eqref{non-penetration}. 
Namely, as proposed in 
\cite{RosRou10TARA} and arguing in the very same way as therein, we could 
include in the adhesive contact energy functional $\calE_k$ from \eqref{defEk} the  term 
$\int_{\GC} I_{K(x)}(\JUMP{u(x)}) \dd \Surf$, with  the 
$x$-dependent \emph{linear subspaces} $K(x)$ e.g.\ given by 
\[
\begin{aligned}
K(x) : = \{ v\in \R^d\, : \ v \cdot \mathbf{n}(x) =0 \}.
\end{aligned}
\] 
This would prescribe  a  zero normal jump of the displacement, and thus, 
in this way, we would allow only for tangential slip along 
$\GC$.  
\end{remark}
%
\subsection{Semistable energetic solutions for the brittle system}
\label{ss:2.3}
%
 Let us now specify the functional analytic setting for the brittle system:
\\
\textbf{Function spaces:}
In addition to the spaces $\Spu,\, \Spw,\, \Spz,\, \Spx$ from \eqref{function-spaces}, 
we will work with the following family of time-dependent spaces, defined for  all $t\in [0,T]$:  
\begin{equation}
\label{spu-z}
\begin{gathered}
\Spu_z(t)
 =\{v\in H^1_\mathrm{D}(\Omega\backslash\GC;\R^d)\, : \ 
 \JUMP{v}=0
\text{ a.e.\ on }\supp z(t)\subset\GC\}\,,
\\
  \text{for a given } z \in L^\infty (0,T; \SBV(\GC;\{0,1\}))   \cap \BV ([0,T]; L^1(\GC))     
\text{ such that }
  \\
 z(t_2) \leq z(t_1) \text{ for all } 0\leq t_1 \leq t_2 \leq T\,. 
\end{gathered}
\end{equation} 
These
will be the spaces for the test functions in the weak formulation of the momentum balance for the 
brittle system.
Observe that, with these spaces we are enforcing a constraint slightly stronger than 
$z |\JUMP{v}|=0,$  cf.\ also Remark \ref{rmk:subtle} ahead.
\par
As we will see (cf.\ Proposition \ref{PropsSpuM} later on), 
$\Spu_z(t)$, endowed with the norm induced by $H^1(\Omega\backslash\GC;\R^d) $,  
is a closed subspace of $\Spu$.
Hence, by the Hahn-Banach theorem every $\xi\in  \Spu_z(t)^*$ can be extended 
to  a functional $\tilde \xi \in \Spu^*$ such that 
\begin{equation}
\label{HB}
\pairing{}{\Spu_z(t)}{\xi}{v} = \pairing{}{\Spu}{\tilde \xi}{v} \quad 
\text{for all } v \in \Spu_z(t) \quad \text{and} 
\quad \| \xi \|_{\Spu_z(t)^*}= \|\tilde\xi\|_{\Spu^*} 
\end{equation}
Moreover, $\Spu_z(t)$ is 
 continuously and densely embedded  in $\Spw,$ cf.\ \eqref{function-spaces}.  
Hence, $\Spw$ is continuously and densely embedded in the dual space 
$\Spu_z(t)^*$ and for every $t\in [0,T]$ there holds
\begin{equation}
\label{to-gelfand}
\pairing{}{\Spu_z(t)}{\xi}{v} = \inner{\Spw}{\xi}{w} \qquad \text{for all } v \in \Spu_z(t) 
\text{ and all } \xi \in \Spw.
\end{equation}
Moreover, due to  the monotonicity 
of the function  $z(\cdot,x):[0,T]\to\{0,1\}$ for a.a.\ $x\in\GC$, 
we have that
$\supp(z(t_2))) \subset \supp (z(t_1))$ for every $0 \leq t_1 \leq t_2 \leq T$, and therefore
 \begin{equation}
 \label{monotonicity-Vz-later-used}
 \Spu_z(t_1) \subset \Spu_z(t_2) \qquad \text{for all $0\leq t_1 \leq t_2 \leq T$.}
 \end{equation}
  Accordingly, for every $\xi \in \Spu_z(t_2)^*$ we can consider its restriction to 
$\Spu_z(t_1)$, which gives an element of $\Spu_z(t_1)^*$ defined by 
$\pairing{}{\Spu_z(t_1)}{\xi|_{\Spu_z(t_1)}}{v} = \pairing{}{\Spu_z(t_2)}{\xi}{v}$. 
The restriction map  is continuous and is indeed the adjoint of the embedding  
$\Spu_z(t_1) \subset \Spu_z(t_2)$.
Therefore, there holds
\begin{equation}
\label{agreed?}
\Spu_z(t_2)^* \subset \Spu_z(t_1)^* \text{
continuously for all $0\leq t_1 \leq t_2 \leq T$.}
\end{equation}
\par
 We will also work with the space
\begin{align}
&
\label{L2VZ}
L^2(0,T; \VZ): = \{ v \in L^2(0,T;\Spu)\, : \ v(t) \in \Spu_z(t) \ \foraa\, t \in (0,T)\},
\\
\end{align}
endowed with the norm $\|\cdot\|_{L^2(0,T;\Spu)}$,   and with 
\begin{subequations}
\label{L2VZ-dual}
\begin{equation}
\label{L2VZ-dual-1}
\begin{aligned}
L^2(0,T; \VZ^*): = &  \{\xi \in L^2(0,T; \Spu_z(0)^*) \, : 
\\
& \quad 
\xi(t) \in  \Spu_z(t)^* \ \foraa t \in (0,T),\,
t \mapsto \|\xi(t)\|_{\Spu_z(t)^*} \in L^2(0,T)  \},
\end{aligned}
\end{equation}
endowed with the norm 
\begin{equation}
\label{L2VZ-dual-2}
\| \xi \|_{L^2(0,T;\VZ^*)}:=\sup_{v\in L^2(0,T;\VZ)} 
\left|  \int_0^T \pairing{}{\Spu_z(t)}{ \xi(t)}{v(t)} \dd t \right |.
\end{equation} 
\end{subequations}
 Observe that, underlying definition 
\eqref{L2VZ-dual}
is the fact that $ \Spu_z(t)^* \subset  \Spu_z(0)^*$ for all $t\in [0,T];$ we also refer to 
Prop.\ \ref{PropsSpuM} for more details.    
Finally, let us  also introduce the Sobolev space 
\begin{equation}
\label{Sobolev-VZ}
H_{\#}^2(0,T; \VZ^*): = \{ u \in H^{1}(0,T; \Spu^*)\, : \ddot{u} \in L^2(0,T; \VZ^*)\}, 
\end{equation}
with  $\ddot{u}$  the  (second-order in time) \emph{weak} distributional derivative, to be understood 
at almost all $t\in (0,T)$
as the weak limit of the difference quotients  $\frac{\dot{u}(t+h)-\dot{u}(t)}{h}$ 
in $\Spu_z(t)^*$, namely
\begin{equation}
\label{weak-ddot-u}
 \pairing{}{\Spu_z(t)}{\ddot{u}(t)}{v} = 
\lim_{h\to 0}   \pairing{}{\Spu_z(t)}{\frac{\dot{u}(t+h)-\dot{u}(t)}{h}}{v} 
\qquad \text{for all } v \in \Spu_z(t).
\end{equation}
The basic properties of these spaces are collected in  Proposition \ref{PropsSpuM} ahead.
\par
\noindent
\textbf{Dissipation potentials and energy   functionals for the brittle  case, $k=\infty$: }
The brittle system $\ingrsysinf$ 
is  governed by the  kinetic energy $\calK$ as in \eqref{defEkin}, 
the (quadratic) viscous dissipation $\calV$ as in \eqref{defV}, 
and by the following rate-independent dissipation potential and mechanical energy:
\begin{eqnarray}
\label{defRinfty}
&&\calR_\infty:\Spz\to[0,\infty]\,,\;
\calR_\infty(\dot z):=\int_{\GC}   R_\infty (\dot z)\,\mathrm{d}\Surf\,,\;
R_\infty(\dot z):=\left\{
\begin{array}{ll}
a_\infty^1|\dot z|&\text{if }\dot z\leq 0\,,\\
\infty&\text{otw.}\,,  
\end{array}
\right.
\\
\nonumber
&&\calE_\infty:[0,T]\times\Spu\times\Spz\to\R\cup\{\infty\}\;\text{ defined for all }t\in[0,T]:\\
\label{defEinfty}
&&\calE_\infty(t,u,z):=
\left\{
\begin{array}{cl}
\widetilde\calE_\infty(t,u,z)+\calJ_\infty(u,z)&\text{if }
(u,z)\in\Spu\times\Spx_\infty,\qquad\text{where }\\
\infty&\text{otw.}
\\
\end{array}
\right.\\
\label{brittle-constr-funct}
&&\calJ_\infty(u,z):=\int_{\GC} J_\infty(\JUMP{u},z)\,\mathrm{d}\Surf\;\text{ with }J_\infty
\left(\JUMP{u},z\right)= \left\{
\begin{array}{ll}
0 & \text{if } \JUMP{u}=0\text{ a.e.\ on }\supp z,
\\
+\infty & \text{otw.}
\end{array}
\right.  
\\
\nonumber
&&\widetilde\calE_\infty(t,u,z):=\int_{\Omega\backslash\GC}\tfrac{1}{2}\mathbb{C}e(u):e(u)
 -  \pairing{}{\Spv}{\mathbf{f}(t)}{u} +\int_{\GC}\!\!
\big(I_{[0,1]}(z)- a^0_\infty  z\big)
\,\mathrm{d}\Surf+ \mathrm{b}_\infty  P(Z,\GC).
\end{eqnarray} 
Here, $\mathbb{C}$ and $\mathbf{f}$  are as in \eqref{assCD} and \eqref{assRegf}, 
the support of $z$ is defined in a measure-theoretic sense by 
\begin{equation}
\label{defsupp}
   \supp z :=\bigcap\{A\subset\R^{d-1};\,A\text{ closed },\,\calH^{d-1}(Z{\backslash} A)=0\},
   \end{equation}
   based on the identification of $z$ with the set $Z$ such that $z= \chi_Z$,  and, 
in correspondence with \eqref{coeffscale}, the coefficients and the space $\Spx_\infty$ 
comply with the following 
\begin{subequations}
\label{coefflim} 
\begin{eqnarray}
&&
\hspace*{-2cm}
\Spx_\infty=\Spx=\SBV(\GC;\{0,1\}),\;a_\infty^0=a^0=\mathrm{const.}\,,\; 
a_\infty^1=a^1=\mathrm{const.}\,,\; 
\mathrm{b}_\infty=\mathrm{b}=\mathrm{const.}\,,\quad\text{or}\\
&&
\hspace*{-2cm}
\Spx_\infty=L^\infty(\GC),\hspace*{2.05cm} a_\infty^0=a_\infty^1=\mathrm{b}_\infty=0\,.
\end{eqnarray}
\end{subequations} 
Observe that the functional $\calE_\infty(t,\cdot,z):\Spu\to\R\cup\{\infty\}$ is convex and that 
its proper domain is $\VZ$. Its subdifferential with respect to $u$, 
which appears in the subdifferential inclusion \eqref{eq-u}, 
takes the form $\partial_u \calE_\infty = \partial_u \widetilde{\calE}_\infty + \partial_u \calJ_\infty$ by the sum rule. 
Now, 
$\partial_u  \widetilde{\calE}_\infty$ is the singleton given 
by the G\^ateaux-differential of $u \mapsto \widetilde{\calE}_\infty(t,u,z)$, while 
$ \partial_u \calJ_\infty: \Spu \rightrightarrows \Spu^*$ is a multi-valued operator. 
 But we check that 
 for a.a.\ $t\in(0,T)$: 
\begin{equation}
\label{DuJinfty} 
\forall\,\zeta \in\partial_u\calJ_\infty(u(t),z(t)),\;\forall\,v\in\Spu_z(t):\quad 
\langle\zeta,v\rangle_\Spu=0\,. 
\end{equation}
In fact, this relation can be verified directly from the definition  
of  $\partial_u\calJ_\infty(u(t),z(t))$, 
which reads 
$\langle\zeta(t),v-u(t)\rangle_\Spu\leq0$ for any $v\in\Spu_z(t)$. Using the test 
functions $v=2u(t)\in\Spu_z(t)$ and $v=0\in\Spu_z(t)$ we first deduce 
that $\langle\zeta(t),u(t)\rangle_\Spu=0$. 
Thus, by testing with $v$ and $-v\in\Spu_z(t)$ we also find that 
$\langle\zeta(t),v\rangle_\Spu=0$ for any $v\in\Spu_z(t)$.   
\par
In view of observation \eqref{DuJinfty}, 
the notion of \semi energetic solution for the brittle system  is not stated  
 with the general subdifferential inclusion \eqref{eq-u} in  $\Spu^*,$ 
but with its restriction to the domain 
$\VZ(t)\subset\Spu,$ which in fact increases 
with $t\in[0,T]$ since $z$ monotonically decreases in time. 
This restriction results in the 
momentum balance \eqref{weak-mom-brittle} below. 
\begin{definition}[Semistable  energetic solution for the brittle system]
\label{def:energetic-sol-brittle}
Let $\varrho\geq 0$. Given $(u_0.u_1,z_0) \in \Spu \times \Spw \times \Spz$, 
we call a pair $(u,z) : [0,T] \to \Spu \times \Spz$  a \emph{\semi energetic solution} 
to the evolutionary brittle  system $\ingrsysinf$ if 
\begin{align}
&
\label{reg-u-britt;e}
u \in 
   H^1(0,T;\Spu)\,, \quad  \dot u \in L^\infty (0,T; \Spw)\,, \quad
u \in   H^2_\#(0,T;\VZ^*)\,, 
\end{align}
$z$ fulfills 
\eqref{reg-z}, and the pair $(u,z)$ fulfill the  Cauchy conditions
\begin{equation}
\label{Cauchy}
u(0)=u_0, \ \dot{u}(0) = u_1, \ z(0)= z_0,
\end{equation}
and the 
\begin{compactitem}
\item[-] weak formulation of the mometum balance in the brittle case
\begin{subequations}
\label{weak-mom-brittle}
 \begin{align}
 & 
   \label{u-constraint}
 u(t) \in \VZ(t) \qquad
  \text{ for every } t \in [0,T], 
\\
 & 
\label{weak-momentum-brittle}
 \begin{aligned}
 \int_\Omega \varrho \ddot{u}(t) v \dd x  & + \int_{\Omega\backslash\GC} 
\left( \mathbb{D}e(\dot{u}(t)) \colon e(v) {+} \mathbb{C}e(u(t)):e(v) \right) \,\mathrm{d}x 
= \langle\mathbf{f}(t),v\rangle_\Spu
\\
& \qquad
 \qquad   \text{for every } v \in \VZ(t) \quad \foraa\, t \in (0,T); 
\end{aligned}
\end{align}
\end{subequations}
\item[-] the  semistability condition
\eqref{semistab-z} with   $\calR_\infty$  from \eqref{defRinfty} 
and $\calE_\infty$ from  \eqref{defEinfty};
\item[-]  the energy-dissipation  inequality 
\eqref{enineq}
 with $\calV$ from \eqref{defV},   $\calR_\infty$   \eqref{defRinfty},  
and $\calE_\infty$   \eqref{defEinfty}. 
\end{compactitem}
\end{definition}
\par
We conclude this section by   fixing some properties of  $L^2(0,T;\VZ)$ and  of related spaces.
The following statement is given in terms of a time-dependent set $M(t)$ which  
applies to the  closed  
set
$\supp z(t)$,  but also  to  suitable enlargements of $\supp z(t)$, cf.\ \eqref{recsp}  
and Proposition \ref{Propsrecsp}  later on.
\begin{proposition}
\label{PropsSpuM}
Let $(M(t))_{t\in [0,T]}$ be a family of 
closed subsets of $\GC$  
and set 
\begin{equation}
\label{VMt}
\Spu_M(t):=\{v\in H^1_\mathrm{D}(\Omega\backslash\GC;\R^d)\, : \ \JUMP{v}=0
\text{ a.e.\ on }M(t)\subset\GC\}
=H^1_\mathrm{D}\big((\Omega\backslash\GC)\cup M(t);\R^d\big)\,.
\end{equation}
 Then, $\Spu_M(t)$ is a closed subspace of 
 $\Spu$, and thus it is a 
  reflexive and separable Banach space, 
and so is its dual  $\Spu_M(t)^*,$ which is  isometrically  isomorphic   to the quotient space 
\begin{align}
\nonumber
\Spu/\Spu_M(t)^\bot\quad\text{with }   \Spu_M(t)^\bot:  
&  = \{ \tilde \xi \in \Spu^*\, : \ \pairing{}{\Spu}{\tilde \xi}{v} = 0 
\ \text{for every } v \in M(t)  \} 
\end{align}
the annihilator of $\Spu_M(t)$. 
Furthermore, $\Spu_M(t)$ is  dense in $\Spw$. 
  \par
Also   the space
\begin{equation}
L^2(0,T;\Spu_M):=\big\{v\in L^2(0,T;\Spu):\,v(t) \in \Spu_M(t)\,\foraa  t \in (0,T) \big\}\quad
\text{ with }
\|\cdot\|_{L^2(0,T;\Spu)}
\end{equation}
is reflexive and separable. 
\par
   Finally, suppose that the sets $(M(t))_{t\in [0,T]]} $ are monotonically decreasing, i.e.\ 
  \begin{equation}
  \label{monotonicity-Mt}
  M(t_2) \subset M(t_2) \qquad \text{ for all $0 \leq t_1\leq t_2 \leq T$.}
  \end{equation} 
Then, $L^2(0,T;\Spu_M)^* \ \text{ endowed with the norm } \|\cdot\|_{L^2(0,T;\Spu_M)^*}
:=\sup_{v\in L^2(0,T;\Spu_M)}|\langle\cdot, v \rangle_{L^2(0,T;\Spu)}|,$
 is isometrically isomorphic to 
\begin{equation}
\label{L2VMdual}
\begin{aligned}
L^2(0,T;\Spu_M^*): & =\{ \xi  \!\in\! L^2(0,T; \Spu_M(0)^*):
\\ & \quad 
\,\xi(t)\!\in\! \Spu_M(t)^*\,\foraa t \!\in\! (0,T) 
\text{ and } t \mapsto \|\xi(t)\|_{\Spu_M(t)^*} \!\in\! L^2(0,T)  \},
\end{aligned}
\end{equation}
with the norm 
$
\| f \|_{L^2(0,T;\Spu_M^*)}:=\sup_{v\in L^2(0,T;\Spu_M)} \left|  \int_0^T 
\pairing{}{\Spu_M(t)}{ f(t)}{v(t)} \dd t \right |.$ 
\end{proposition}
\begin{proof}
Let $(v_n)_n \subset \Spu_M(t)$ with $v_n \to v$ in $\Spu$. 
Taking into account  that the jump operator
$\JUMP{\cdot}: \Spu \to H^{1/2}(\GC;\R^d)$ is continuous by the trace theorem, 
one immediately checks that 
$\JUMP{v} =0$ a.e.\ on  the \emph{closed} set $M(t)$. 
Thus $\Spu_M(t)$ is  a closed, linear subspace of  $\Spu$, whence its reflexivity and separability.
Then, its dual  $\Spu_M(t)^*$ is also reflexive and separable. 
Since $\Spu_M(t)$ is a closed subspace of $\Spu$,  
one of the corollaries of the Hahn-Banach theorem applies, 
 yielding $\Spu_M(t)^*$ is isometrically isomorphic to the quotient space 
$\Spu^*/\Spu_M(t)^\bot$ through the operator
 $L : \Spu_M(t)^* \to \Spu^*/\Spu_M(t)^\bot$ which maps an element $\xi$ of 
$ \Spu_M(t)^*$ to $\xi + \Spu_M(t)^\bot$, 
where we denote by the same symbol the extension of $\xi$ to $\Spu$.
\par 
The density of $\Spu_M(t)$ in the space $\Spw$ can be concluded from the fact that 
$H^1_\mathrm{D}(\Omega;\R^d)\subset\Spu_M(t)$ and $H^1_\mathrm{D}(\Omega;\R^d)$ 
is dense in $\Spw$ given that $\Omega$ is a Lipschitz domain.  
\par
Observe that 
$L^2(0,T;\Spu_M)$ is a closed subspace of $L^2(0,T;\Spu)$:  Given a sequence 
$(v_n)_n \subset L^2(0,T;\Spu_M)$ with $v_n \to v$ in $L^2(0,T;\Spu)$, there 
holds (for a not relabeled subsequence) 
$v_n(t) \to v(t)$ in $\Spu,$ whence $v(t) \in \Spu_M(t)$  for a.a.\ $t\in (0,T)$. 
Thus, $L^2(0,T;\Spu_M)$ inherits the reflexivity and separability of $L^2(0,T;\Spu)$. 
Clearly, also its dual $L^2(0,T;\Spu_M)^*$ 
is reflexive and separable. 
\par
Finally, in order to verify the equivalence 
$L^2(0,T;\Spu_M)^*=L^2(0,T;\Spu_M^*)$ stated along with \eqref{L2VMdual}, we first observe that 
 $  L^2(0,T;\Spu)^*  \cong  L^2(0,T;\Spu^*) $.  We will now show that 
the annihilator  of $ L^2(0,T;\Spu_M)$, 
namely 
\[
 L^2(0,T;\Spu_M)^\bot = \Big\{ \xi \in L^2(0,T; \Spu)^*:\, 
\textstyle{\int_0^T}\!\!\! \pairing{}{\Spu}{\xi(t)}{v(t)} \dd t =0\,\text{ for all } v \in L^2(0,T;\Spu_M)\Big\}, 
\]
 is isometrically isomorphic to
 \[
L^2(0,T;\Spu_M^\bot) : = \big\{ \xi \in L^2(0,T; \Spu^*):\, \xi(t)  
\in  \Spu_M(t)^\bot \ \foraa\, t \in (0,T)\big\}.
\]
To this aim, we 
 observe that, due to the monotonicity property \eqref{monotonicity-Mt}, there holds
\begin{equation}
\label{monot-VM}
\Spu_M(t_1) \subset \Spu_M(t_2) \quad \text{and} \quad  \Spu_M(t_2)^\bot 
\subset \Spu_M(t_1)^\bot \quad \text{for all } 0 \leq t_1 \leq t_2 \leq T. 
\end{equation}
Then, we have that 
\begin{equation}
\label{localization-prop}
\xi \in  L^2(0,T;\Spu_M)^\bot \;\Leftrightarrow\;
\Big(\xi(t) \in L^2(0,T; \Spu^*) \text{ \& } 
\forall\, \mathsf{v} \!\in\!  \Spu_M(t):\pairing{}{\Spu}{\xi(t)}{\mathsf{v}} =0\;\,
\foraa t \!\in\! (0,T)\Big).
\end{equation}
The right-to-left implication is obvious. As for the converse one, we point out that, 
for all $\xi\in  L^2(0,T;\Spu_M)^\bot$ there holds 
\begin{equation}
\label{ad-Leb-points}
\frac1h \int_{t}^{t+h}   \pairing{}{\Spu}{\xi(s)}{\mathsf{v}} \dd s =0 \qquad \text{for every $h>0$,  
$\mathsf{v} \in \Spu_M(t) $ and  $t\in [0,T-h]$.}
\end{equation}
This follows from choosing 
 $v(t): =\tfrac1h \chi_{(t,t+h)} \mathsf{v}$, which satisfies 
$ v\in L^2(0,T;\Spu_M)$ thanks to \eqref{monot-VM}, 
in the identity $\int_0^T \pairing{}{\Spu}{\xi(t)}{v(t)} \dd t =0$
 fulfilled  by $\xi\in  L^2(0,T;\Spu_M)^\bot$. Then, letting 
$h\downarrow 0$ in \eqref{ad-Leb-points} yields $\xi(t) \in \Spu_M(t)^\bot$. 
Thus, the left-to-right implication holds true. 
Taking into account the representation of $ L^2(0,T;\Spu_M)^\bot$ and    
the Hahn-Banach theorem we find that 
\begin{equation*}
\begin{split}
L^2(0,T;\Spu_M)^*  \cong  L^2(0,T;\Spu)^*/L^2(0,T;\Spu_M)^\bot   
\cong   L^2(0,T;\Spu^*)/L^2(0,T;\Spu_M^\bot)
  \cong  L^2(0,T;\Spu_M^*)\,,
\end{split}
\end{equation*} 
which concludes the proof. 
\end{proof}
%
\subsection{Main result: Passage from adhesive to brittle}
\label{MainResult}
%
Following 
\cite{Mie-evolGamma}, we refer to the adhesive-to-brittle  convergence stated 
in  Thm.\   \ref{thm:main} below  as an 
\emph{evolutionary $\Gamma$-convergence result}. 
Let us point out in advance that, with this limit passage we obtain 
a \semi energetic 
solution to the brittle system, 
with the enhanced regularity property that $\ddot u\in H^2 (0,T;\Spu(0)^*)$. 
Furthermore,  the second of \eqref{propsuz}  below will allow us to 
test the momentum balance for the brittle system  \eqref{weak-momentum-brittle} by $\dot{u}$. 
This will be the key step 
for obtaining the
 energy-dissipation \emph{identity}. 
Like for the adhesive contact systems, we also obtain
 a uniqueness result for the displacements of the brittle system corresponding to   
a given semistable $z$.  In its proof, a pivotal role is played 
 by two  separate energy balances for the displacement
and for the internal variable, which we prove along with the energy-dissipation 
\emph{balance} on  an arbitrary interval $(s,t)\subset (0,T)$, for almost all $s<t\in (0,T)$. 
\begin{theorem}[Evolutionary $\Gamma$-convergence of the adhesive systems to the brittle limit]
\label{thm:main}
Assume \eqref{ass-domain}--\eqref{assdata}. 
For each $k\in\N$ let $(u_k,z_k)$ be a \semi energetic solution of the adhesive system 
$\ingrsysk$.  
Assume that  $(u_k(0),\dot u_k(0),z_k(0))=(u_0^k,u_1^k,z_0^k)$,  and that 
\begin{equation}                                               
 \label{initial-data-conv}            
   (u_0^k,u_1^k,z_0^k)\to (u_0,u_1,z_0) \quad \text{in } \Spv\times\Spw\times\Spz \quad 
\text{ and } \quad   \calE_k(0,u_0^k,z_0^k) \to \calE_\infty  (0,u_0,z_0),                  
\end{equation} 
such that
the pair $(u_0,z_0)$ complies with 
the stability condition \eqref{semistab-z} at $t=0$, with $\calE_\infty$ and $\calR_\infty$ given by  
\eqref{defRinfty}--\eqref{coefflim}. 
Then there exists a (not relabeled) subsequence $(u_k,z_k)_k$ 
and a pair $(u,z)$ with the following properties: 
\begin{compactenum}
\item {\bf immediate convergences: }
the following convergences hold true 
\begin{subequations}
\label{convs}
\begin{eqnarray}
\label{convsu}
\hspace*{-15mm}
&&u_k\rightharpoonup u\text{ in }   H^1(0,T;\Spu),  
\;\;\text{and}\;\;
u_k(t)\rightharpoonup u(t)\text{ in } \Spu  \text{ for all }t\in[0,T]\,,\\
\label{convsz}
\hspace*{-15mm}
&&
z_k(t)\to z(t)\text{ in }\Spz,\;\;
z_k(t)\to z(t)\text{ in }L^q(\GC),\,q\in[1,\infty),
\text{ for all }t\in[0,T]\,,\quad
\\
\label{convsz-2}
\hspace*{-15mm}
&&
z_k\weaksto z\text{ in }L^\infty(0,T;\SBV(\GC;\{0,1\})),\;\;
z_k(t)\weaksto z(t)\text{ in } \SBV(\GC;\{0,1\}) \cap L^\infty(\GC)
\text{ for   all }t\in[0,T] \,,\qquad\qquad\\
\label{convsddu}
\hspace*{-15mm}
&& \exists\,\lambda\in L^2(0,T;\Spu^*),\quad  \varrho  \ddot u_k + \partial_u \calJ_k 
\rightharpoonup\lambda\text{ in }L^2(0,T;\Spu^*)\,, 
\end{eqnarray}
\end{subequations} 
\item 
{\bf \semi energetic solution \& brittle constraint: }the limit pair $(u,z)$ is a \semi 
energetic solution  of the brittle 
system $\ingrsysinf$ 
in the sense of Def.\ \ref{def:energetic-sol}. It satisfies the initial condition 
\eqref{Cauchy} and the brittle constraint 
\begin{equation}
\label{propsuz}
\begin{aligned}
& 
\JUMP{u(t)}  |_{\supp z(t)}=0   \;\Surf\text{-a.e.\ in }\GC \text{ for every } t \in [0,T],   
\quad\text{and}
\\
& 
\JUMP{\dot{u}(t)}  |_{\supp z(t)}=0   \;\Surf\text{-a.e.\ in }\GC   
\text{ for  almost  all }t\in(0,T)\,,
\end{aligned}
\end{equation}    
Moreover, for $\lambda\in L^2(0,T;\Spu^*)$ obtained in \eqref{convsddu} we have 
\begin{equation}
\label{lambda-id}
\text{for a.a.\ }t\in(0,T):\quad\lambda(t)=\ddot u(t)\;\text{ in }\VZ(t)^*\,, 
\end{equation}
\item 
{\bf regularity of $u$: } in addition to \eqref{reg-u-britt;e}, the limit $u$ fulfills 
\begin{equation}
\label{better-H2}
u \in H^2(0,T; \VZ(0)^*)\,,
\end{equation}  
\item
{\bf weak temporal continuity: }
$\dot{u} \in   \mathrm{C}^0_{\mathrm{weak}}([0,T];\Spw)$, 
\item 
{\bf energy-dissipation balance: }
the pair $(u,z)$ complies with the energy-dissipation inequality 
on $(0,t)$ and on $(s,t)$, for every $t \in (0,T]$ and almost every $s \in (0,t)$,  and 
 as an \emph{identity} 
along  the interval $(s,t)\subset[0,T]$  for almost all $s,t \in (0,T) $ and for $s=0$,  
\begin{equation}
\label{eneqlim-identity} 
 \begin{aligned}
&\tfrac{1}{2}\|\dot u(t)\|_{\Spw}^2   
+ \int_s^t  2\calV(\dot u(\tau))   
\,\mathrm{d} \tau + \Var_{\calR_\infty}(z, [s,t])+
 \calE_\infty(t,u(t),z(t))\\ 
& \quad =\tfrac{1}{2} \|\dot u(s)\|_{\Spw}^2  
+  \calE_\infty(s,u(s),z(s)) + \int_s^t \partial_t\calE_\infty(\tau,u(\tau),z(\tau))\,\mathrm{d}\tau\,.
 \end{aligned}
\end{equation}  
In addition, the displacements, resp.\ the delamination variable, comply with the following 
separate balances of the bulk, resp.\ surface, energy terms 
along the interval $(s,t) \subset (0,T)$ for a.a.\ $s<t \in (0,T)$ and for $s=0$:
\begin{eqnarray}
\label{eneqlim-identity-bulk} 
\hspace*{-6mm}
&&\tfrac{1}{2}\|\dot u(t)\|_{\Spw}^2
+ \int_s^t  2\calV(\dot u(\tau)) 
\,\mathrm{d} \tau +
 \int_{\Omega\backslash\GC}\tfrac{1}{2}\mathbb{C}e(u(t)):e(u(t))\,\mathrm{d}x
-\langle\mathbf{f}(t),u(t)\rangle_\Spu
\\
\nonumber
\hspace*{-6mm}
&&\quad 
=\tfrac{1}{2} \|\dot u(s)\|_{\Spw}^2
+\int_{\Omega\backslash\GC}\tfrac{1}{2}\mathbb{C}e(u(s)):e(u(s))\,\mathrm{d}x
-\langle\mathbf{f}(s),u(s)\rangle_\Spu  
- \int_s^t\langle\dot{\mathbf{f}}(\tau),u(\tau)\rangle_\Spu \,\mathrm{d}\tau\,, 
\\
\hspace*{-6mm}
&&
\label{eneqlim-identity-surf} 
\mathrm{b}_\infty P(Z(t),\GC)+\!\int_{\GC}\hspace*{-3mm}a_\infty^0z(t)\,\mathrm{d}\Surf
+ \Var_{\calR_\infty}(z, [s,t])
=\mathrm{b}_\infty P(Z(s),\GC)+\!\int_{\GC}\hspace*{-3mm}a_\infty^0z(s)\,\mathrm{d}\Surf\,,\qquad 
\end{eqnarray}
\item
{\bf enhanced convergences: }
there hold the enhanced convergences for almost all $t\in(0,T)$:
\begin{subequations}
\label{enh-convs}
\begin{align}
&
\label{enh-convs-1}
\dot{u}_k(t) \to \dot{u}(t)  \text{ in } \Spw, 
\\
&
\int_0^t \dissu(\dot{u}_k(s)) \dd s \to \int_0^t \dissu(\dot{u}(s)) \dd s,
\\
&
\Var_{\calR_k} (z_k, [0,t]) \to \Var_{\calR_\infty} (z, [0,t]),
\\
& 
\label{enh-convs-4}
\calE_k (t,u_k(t),z_k(t)) \to \calE_\infty (t,u(t),z(t)),   
\end{align}
\end{subequations}
\item 
{\bf enhanced initial condition: } 
the Cauchy datum $u_1$ is even attained in the sense of difference quotients
\begin{equation}
\label{enhanced-Cauchy-dotu}
\lim_{h\downarrow 0} \frac1h \int_0^h \| \dot{u}(t) - u_1\|_{\Spw}^2 \dd t =0\,,
\end{equation} 
\item 
{\bf uniqueness of the displacements for given  semistable  $z\in L^\infty(0,T;\SBV(\GC;\{0,1\}))$: }
let $(u,z)$ and $(\tilde u,z)$ be \semi energetic solutions to the brittle system 
$\ingrsysinf$, satisfying the brittle 
momentum balance \eqref{weak-mom-brittle} with the same initial data $u_0$ and $u_1$.  
Then,  $\tilde u = u$.  
\end{compactenum}
\end{theorem}
The proof of Theorem \ref{thm:main} will be carried out in detail in Section \ref{s:4}. 
Instead we will now discuss our result and compare it with other existing results, 
in particular focusing on \cite{DMLar11EWED,DMLarToa,DMLazNar}. 
%
\subsubsection*{Discussion of our result: }
%
{\bf Momentum balance: }Integrating the $k$-momentum balance \eqref{weak-mom-intro} over $(0,T),$ 
using test function $\eta v\in L^2(0,T;\Spu)$ with $\eta\in C_0^\infty(0,T)$ and $v\in\Spu,$ 
convergences \eqref{convs} allow us to pass $k\to\infty$ in  \eqref{weak-mom-intro}
using weak-strong convergence arguments. By localization via the fundamental lemma of 
the Calculus of Variations, we obtain the limit equation
\begin{equation}
\label{brittle-alternative}
\hspace*{-2mm}
\langle\lambda(t),v\rangle_\Spu
+\int_{\Omega\backslash\GC}\!\!\!\!\big(\mathbb{C}e(u(t))+\mathbb{D}e(\dot u(t))\big):e(v)\,\mathrm{d}x
=\langle\mathbf{f}(t),u(t)\rangle_\Spu\;\text{for all }v\in\Spu,\text{ for a.a.\ }t\in(0,T)\,.
\end{equation}
But in general we cannot identify $\lambda(t)=  \varrho  \ddot u(t)+\zeta(t)$ with $\zeta(t)\in\partial_u\calJ_\infty(u(t),z(t)),$ cf.\ 
\eqref{defEinfty}. This is mainly due to the fact that $\ddot u_k$ and $\partial_u\calJ_k(u_k,z_k)$ 
only as a sum are uniformly bounded by a constant independent of $k,$ cf.\ \eqref{unifbdcombi}, 
whereas their separate bounds blow up with $k\to\infty$, 
cf.\ \eqref{kdepbd}.  
Therefore, relation \eqref{lambda-id} establishes a link between the balance \eqref{brittle-alternative} 
in $\Spu^*$ and the brittle momentum balance \eqref{weak-mom-brittle} which is restricted to 
the domain $\VZ(t)$ of $\calE_\infty(t,\cdot,z(t))$ that increases with time. 
Here, in these closed subspaces $\VZ(t)$ of $\Spu,$      
we indeed have $\lambda= \varrho  \ddot u(t)+\zeta(t)$ in $\VZ(t)^*$ 
with $\zeta(t)\in\partial_u\calJ_\infty(u(t),z(t)),$ 
since $\partial_u\calJ_\infty(u(t),z(t))\subset\VZ(t)^\perp$ by \eqref{DuJinfty}. 
Spaces akin to $\VZ(t)$ are also used in \cite{DMLar11EWED,DMLarToa}  to formulate the momentum 
balance.  
\par 
\noindent 
{\bf Energy balance: } It has to be pointed out that, in the brittle case, 
 we obtain the energy-dissipation balance \eqref{eneqlim-identity}  
along intervals $(s,t) \subset ( 0,T)$ 
with $s$ and $t$ Lebesgue points for $\dot{u}$, cf.\ the forthcoming 
Lemmata \ref{ChainruleLemma} and \ref{UEDEbrittle}. In particular, this balance splits into 
the bulk balance \eqref{eneqlim-identity-bulk} for the displacements and the surface balance 
\eqref{eneqlim-identity-surf} for the delamination variable. Such a pure bulk balance, 
where terms related to crack growth do not show up, is also obtained in \cite{DMLar11EWED}.   
But, in contrast to \eqref{eneqlim-identity}, \cite{DMLar11EWED,DMLarToa} observe 
their energy-dissipation balance to hold along \emph{all} subintervals $[s,t]\subset[0,T]$. 
This is strongly related  to (indeed,  implies, cf.\  the proof of \cite[Lemma 3.10]{DMLar11EWED})
  the fact that \cite{DMLar11EWED,DMLarToa} find solutions such that the 
map $\dot u:t\mapsto \dot u(t)$ is continuous from time into the space 
$\Spw$ (in the dynamic, damped case of \cite{DMLar11EWED}),  
$\VZ(t)^*$ (in the dynamic, undamped case of \cite{DMLarToa}). 
In our setting we  only manage to prove that 
$\dot{u} \in   \mathrm{C}^0_{\mathrm{weak}}([0,T];\Spw)$.    
\par
Observe that the enhanced energy-dissipation balance and regularity properties 
proved  in   \cite{DMLar11EWED}  do not stem from assuming suitable  
temporal regularity for the \emph{prescribed} crack evolution. 
 Given a (unique) solution $u$ of the momentum balance in $[0,T],$ 
their argument to obtain the enhanced energy-dissipation balance 
is based on choosing an arbitrary time $t_0\in(0,T),$ which is not a Lebesgue point, and to 
solve the momentum balance on $(t_0,T)$ with the new initial data 
$u_0=u(t_0),$ $u_1=\dot u(0)$. Glueing  the solution in $(0,t_0)$ with the new solution 
on $(t_0,T)$ and exploiting the previously proved  uniqueness of the displacement leads 
to the enhancements. 
But since our existence result for the brittle system arises 
from an adhesive contact approximation
relying on the \emph{well-preparedness of the initial data} \eqref{initial-data-conv}, 
we cannot choose an arbitrary time $t_0$ as a new initial time to solve momentum balance of 
the brittle system, without implicitly 
requiring the enhanced convergences \eqref{enh-convs} to hold at the arbitrary time $t_0\in(0,T)$. 
A further reason why we are not able to reproduce the arguments from \cite{DMLar11EWED}
 leading to the energy-dissipation balance at all 
times is the presence in our own balance of the surface energy terms due to delamination, which
may jump at countably many times.  In this context, let us mention that the 
energy balance obtained in \cite{DMLarToa}, akin to our \eqref{eneqlim-identity}, 
also features the dissipation due to crack growth. It is obtained along all subintervals 
of $[0,T]$ in a 2D setting with connected cracks that evolve continuously in time and piecewise 
even more smoothly.       
\par
\noindent 
{\bf Crack propagation criterion: }
Both in the adhesive and in the brittle models, 
crack propagation is governed by the semistability inequality  \eqref{semistab-z}. 
It expresses that the semistable delamination variable $z(t),$ among all possible competitors 
$\tilde z\in\Spz,$ is a minimizer of 
the functional $\calE(t,u(t),\cdot)+\calR(\cdot-z(t))$, 
with $u(t)$ kept fixed as the solution of the momentum balance (with $z(t)$) 
in $\VZ(t)^*$ at time $t$. 
\par 
In the fully rate-independent setting (i.e., no viscosity and inertia for the displacements), 
solutions to rate-independent systems for damage and delamination complying  
with the semistability, in place of the  \emph{global stability} condition
in the standard concept of \semi energetic solutions \cite{MieThe04RIHM,Miel05ERIS}, have been termed 
\emph{local solutions}  in, e.g., \cite{RouACVB13,RoThPa15SDLS}.  This higlights that 
 the semistability, as a minimality property, is  
local in the sense that the displacemements are not modified 
(it has however to be stressed that the concept from \cite{RouACVB13,RoThPa15SDLS} 
differs from the, weaker, 
notion of local solution  defined in \cite[Def.\ 4.5]{Miel08DEMF}, \cite[Sec.\ 1.8]{MieRouBOOK}).  
Let us also mention that, when taking the vanishing-viscosity \& inertia limit in 
the momentum balance, \semi energetic solutions in the sense of Def.\ \ref{def:energetic-sol} 
converge to  local solutions (in the sense of  \cite{RouACVB13,RoThPa15SDLS}) of the 
quasistatic limit system, as shown in the case of damage  in  \cite{LRTT}. 
\par
In contrast, models that govern crack propagation by Griffith's fracture criterion, cf.\ e.g., 
\cite{NicSan07DCP2,LaSaSe08ANTD,LalSan11WCIF,DMLazNar} in the dynamic setting, 
are rather based on \emph{global} minimality, 
and therefore correspond to \emph{(global) energetic} solutions in the fully  
rate-independent context.  The dynamic  evolution criterion, 
rephrased in our notation for easier comparison, 
is in fact a constrained \emph{global} minimization problem, namely
\begin{equation}
\label{Griffith}
\min_{\tilde z(t)\in\Spz}\big(\calE(t,v_{\tilde z}(t),\tilde z(t))
+\calK(\dot v_{\tilde z}(t))+\calR(\tilde z(t))\big)  
\end{equation}
where $v_{\tilde z}$ is the \emph{unique} solution of the momentum balance on $[0,t]$, 
corresponding to the (given) delamination variable $\tilde z$; 
 observe that,  both in the adhesive and in the brittle cases,  
we also have uniqueness of the displacements for given $z \in L^\infty (0,T;\SBV(\GC;\{0,1\}))$. 
The existence and uniqueness of a solution to \eqref{Griffith} is then proved, 
in \cite{DMLazNar}, relying on the very special 1D-geometry of the model. 
\par
Let us once more point out that \semi energetic solutions in the sense of 
Def.\ \ref{def:energetic-sol}, featuring semistability
as the evolution criterion for the internal variable, are obtained from 
 alternating (decoupled)  minimization on the time-discrete level,
cf.\ \cite{RosTho15CEx},
whereas solutions complying with the Griffith criterion \eqref{Griffith}, 
in analogy to the fully rate-independent setting,
are rather based on simultaneous minimization in the two variables.
However, it has been observed that 
energetic solutions to rate-independent systems, based on global, simultaneous minimization, 
tend to jump rather 
early in time.  This has motivated  the design of alternative solution concepts,  
resp.\ crack propagation criteria,  cf.\ e.g., \cite{KnMiZa08ILMC,  Lars10, RoThPa15SDLS} 
in the realm of crack propagation and delamination.  
In this spirit also the concept discussed in \cite{DMLarToa} has to be understood: 
Therein, the crack propagation criterion is a local-in-time reformulation of 
Griffith's fracture criterion, inspired by so-called $\epsilon$-stable solutions from \cite{Lars10}. 
It would be interesting to advance the study of \emph{alternative} solution notions 
to the brittle system, possibly by  combining  the adhesive-to-brittle 
limit with the vanishing-viscosity approximation, in the spirit of  \cite{MRS12,MRS13,MRS14}.  
\par 
\noindent 
{\bf The different scalings \eqref{coeffscale} and the resulting limit models 
\eqref{defRinfty}--\eqref{coefflim}: } In the brittle setting 
it has been noted (cf.\ \cite{RoThPa15SDLS}) 
 that the semistability condition governing crack propagation, reducing to 
 \begin{equation}
 \label{semistab-expl-brittle}
 \begin{aligned} & 
\int_{\GC} J_\infty \left( \JUMP{u(t)}, z(t) \right) \dd \Surf 
+ \mathrm{b}_\infty P(Z(t), \GC)\\ &  \leq \int_{\GC} J_\infty \left( \JUMP{u}, \tilde z \right) 
\dd \Surf +  \mathrm{b}_\infty P(\tilde{Z}, \GC) + (a_\infty^0 {+} a_\infty^1 ) 
\int_{\GC}  (z(t) {-} \tilde z) \dd \Surf
\end{aligned}
 \end{equation}  
 for all $\tilde z \in \SBV(\GC;\{0,1\}) $ and  for all  $t \in [0,T],$
does not feature any term of positive, finite value that depends on the displacements  
and thus forces $z$ to decrease,  as a function of time. In other words, 
crack growth seems to be rather induced by the perimeter regularization, than by the attempt to 
reduce the mechanical stresses. 
However, the solutions to the brittle systems obtained in Theorem \ref{thm:main} are selected 
by approximation with solutions of the adhesive contact models 
$\ingrsysk_k$. 
Since the semistability condition for finite $k\in\N$, i.e.
\begin{equation}
 \label{semistab-expl-adh}
 \begin{aligned} & 
 \int_{\GC} \tfrac k2 z_k(t) \left| \JUMP{u_k(t)} \right|^2 \dd \mathscr{H}^{d-1} 
+ \mathrm{b}_k P(Z_k(t), \GC) \\ & \leq  
\int_{\GC} \tfrac k2 \tilde{z} \left| \JUMP{u_k(t)} \right|^2 \dd \mathscr{H}^{d-1} 
+ \mathrm{b}_k P(\tilde{Z}, \GC) +(a_k^0 {+} a_k^1 ) \int_{\GC}  (z_k(t) {-} \tilde z) \dd \Surf 
 \end{aligned}
 \end{equation} 
  for all $\tilde z \in \SBV(\GC;\{0,1\}) $ and  for all  $t \in [0,T],$ 
 features the displacement-dependent adhesive contact 
term, which can drive crack propagation, the solutions to the brittle system 
obtained by Theorem \ref{thm:main} should  inherit this information.        
\par
On these grounds,  in \cite{RoThPa15SDLS}  (see also \cite[Sec.\ 7]{RosTho12ABDM})  the alternative scaling from \eqref{coeffscale-b}
has been proposed. While referring to the discussion in \cite{RoThPa15SDLS}, \cite[Sec.\ 7]{RosTho12ABDM} for all details, let us only mention here  that, when the coefficients $a_k^0,\, a_k^1,\, \mathrm{b}_k$ are scaled as in \eqref{coeffscale-b}, multiplying the semistability inequality
 \eqref{semistab-expl-adh}  by $k $ leads to 
\[
\begin{aligned}
& 
 \int_{\GC} \tfrac {k^2}2 z_k(t) \left| \JUMP{u_k(t)} \right|^2 \dd \mathscr{H}^{d-1} 
+ \mathrm{b} P(Z_k(t), \GC) \\ & \leq  
\int_{\GC} \tfrac {k^2}2 \tilde{z} \left| \JUMP{u_k(t)} \right|^2 \dd \mathscr{H}^{d-1} 
+ \mathrm{b} P(\tilde{Z}, \GC) + (a^0+a^1) \int_{\GC} (z_k(t) {-} \tilde z) \dd \Surf  
 \end{aligned}
 \]
  for all $\tilde z \in \SBV(\GC;\{0,1\}) $ and  for all  $t \in [0,T],$ 
which accounts, at least formally,  for the magnitude of the stresses. Indeed, from the contact surface boundary condition \eqref{cont-surf-intro} we read, taking into account that $z_k(t) \in \{0,1\}$, that
\[
 \int_{\GC} \tfrac {k^2}2 z_k(t) \left| \JUMP{u_k(t)} \right|^2 \dd 
\mathscr{H}^{d-1} =  \int_{Z_k(t) \cap \{ | \JUMP{u_k(t)}|>0\} } |(\bbD \dot{e}_k(t) 
+ \bbC e_k(t))|^2 \dd \Surf,
\]
provided that the solutions are sufficiently smooth as to ensure that the stress 
term on the  r.h.s.\ makes sense. Therefore, under the assumption of convergence 
and sufficient regularity of the solutions, and taking 
into account that $\Surf( \{| z_k(t)   \JUMP{u_k(t)}| >0 \}) \to 0 $  
as $k\to\infty$, one expects the \emph{rescaled} brittle model obtained from \eqref{coeffscale-b}  
to contain a term of the form
$
 \int_{Z(t) \cap\partial \{ | \JUMP{u(t)}|>0\} } |(\bbD \dot{e}(t) + \bbC e(t))|^2 \dd \Surf.
$
This conveys the information that, also in the brittle limit  a decrease of the 
semistable function $z$ is not only triggered by the perimeter regularization, but  by the mechanical stresses as well. 
\par
The alternative scaling in \eqref{coeffscale-b} also relates our brittle model to the one studied 
in \cite{DMLar11EWED}, despite one obvious, striking difference. In fact, 
while our energy-dissipation balance features the dissipation 
due to crack growth,  this is not the case in \cite{DMLar11EWED}, so that the model discussed there 
can be interpreted in the way that crack growth does not cost any dissipation.  
Indeed, in the energy balance in  \cite{DMLar11EWED} one may retrieve the presence 
of a dissipation potential, which  nonetheless 
 only seems to ensure the unidirectionality of crack growth, as it 
takes the value $0$ if $\Gamma(s)\subset\Gamma(t)$ for any $s<t$. Our brittle model 
mimics this scenario  when choosing the scaling \eqref{coeffscale-b} in the adhesive 
contact approximation. But, in contrast to \cite{DMLar11EWED}, solutions to our limit model 
carry the  additional information that crack growth is rate-independent and that the 
crack set has finite perimeter, 
which they  inherit from the approximating 
adhesive contact models,  thanks to the bounds  \eqref{unifbdzDiss} \& \eqref{unifbdz}. 
%
%
\section{Existence of \semi energetic solutions for adhesive contact ($k$ fixed)}
\label{s:3}
%
 In \cite{RosTho12ABDM},
the existence of \semi energetic solutions to an  adhesive contact system with perimeter regularization was proved in the case of a \emph{quasistatic} 
momentum balance, i.e.\ neglecting the 
inertial term. On the other hand, in 
\cite{RosRou10TARA} the fully dynamic case was considered, but 
the flow rule for the delamination parameter did not feature the perimeter regularization term we consider here.
\par
 That is why,
in this section we will briefly address  
the existence of \semi energetic solutions for the adhesive contact  evolutionary systems $\ingrsysk$,  with $k\in \N$ fixed, 
 by resorting to the abstract existence result for  damped inertial  systems proved 
in  \cite[Thm.\ 4]{RosTho15CEx}. In what follows,  for the 
reader's convenience we shall first  revisit the prerequisites on an abstract  
 damped inertial system $\ingrsys$ underlying the existence result in  \cite{RosTho15CEx}, 
and then verify that the systems $\ingrsysk$ do comply with them, thus deducing  
the existence of \semi energetic solutions 
for the adhesive contact models as a corollary. 
\par 
 Let $\ingrsys$ be a damped inertial system complying with  the basic conditions 
\eqref{basic-conds}. 
In line with the direct method of the calculus of variations and with tools from rate-independent 
and gradient systems,  \cite[Thm.\ 4]{RosTho15CEx} 
puts the  following  additional  requirements on the functionals 
$\calV:\Spu\to[0,\infty),$ 
$\calR:\Spz\to[0,\infty]$, and $\calE: [0,T]\times \Spu \times \Spz \to \R \cup \{\infty\}$: 
\begin{subequations}
\label{assE}
\begin{eqnarray}
\nonumber
\hspace*{-30mm}&&\text{\bf Boundedness from below \& Weak lower semicontinuity:}\\
\label{assEbdd}
\hspace*{-30mm}
&&\hspace*{15mm}\calE\text{ is bounded from below:}\;\exists\,C_0>0,\,\forall\,(t,u,z)\subset\mathrm{dom}(\calE):\;
\calE(t,u,z)\geq C_0\,;\\
\label{assElsc}
\hspace*{-30mm}
&&\hspace*{15mm}\text{for all }t\in[0,T],\;\calE(t,\cdot,\cdot)
\text{ is weakly sequentially lower semicontinuous on }\Spu\times\Spz\,,
\end{eqnarray}
\end{subequations} 
indeed, if $\calE$ is bounded from below, up to a shift we can assume 
that it is bounded by a positive constant. 
\begin{eqnarray}
\nonumber 
\hspace*{-10mm}
&&\hspace*{-8mm}\text{\bf Temporal regularity and power control:}\\
\label{assEpow}
\hspace*{-10mm}
&&\hspace*{-3mm}
\left.
\!\!\!
\begin{array}{l}         
\forall\,(u,z)\in \mathbf{D}_u\times\mathbf{D}_z,
\text{ the map } t\mapsto\calE(t,u,z)\text{ is differentiable with derivative }
\partial_t\calE(t,u,z)\text{ s.t.\ }\\
\exists\,C_1,C_2>0,\forall\,(t,u,z)\in\mathrm{dom}(\calE):\;
|\partial_t\calE(t,u,z)|\leq C_1(\calE(t,u,z)+C_2)\;\text{and fulfilling}\\
\text{for all sequences }t_n\to t,\,u_n\to u\text{ in }\Spu,\,z_n\to z\text{ in }\Spz\text{ with }
\textstyle{\sup_n}\calE(t_n,u_n,z_n)\leq C\text{  that  }\\
\limsup_{n\to\infty}\partial_t\calE(t_n,u_n,z_n)\leq\partial_t\calE(t,u,z)\,.
\end{array}\hspace*{-3mm}
\right\}
\end{eqnarray}
\begin{eqnarray}
\nonumber 
\hspace*{-10mm}
&&\hspace*{-10mm}\text{\bf Coercivity:}\\
\label{assDisscoercE}
\hspace*{-10mm}
&&
\hspace*{-6mm}
\left.
\begin{array}{l}
\text{there exist }\tau_o>0\text{ such that for all }(t,u_o,z_o)\in[0,T]\times\Spu\times\Spz\\
\text{the map }(u,z)\mapsto\calE(t,u,z)+\tau_o\calV\big(\tfrac{u-u_o}{\tau_o}\big)+\calR(z-z_o)
\text{ has sublevels bounded in }\Spv\times\Spx\,.
\end{array}\right\}
\end{eqnarray}
\begin{equation}
\label{assMRSC}
\begin{split}
&\hspace*{-2mm}\text{{\bf Mutual recovery sequence condition} ensuring the closedness of of stable sets: } \\
&\left.
\begin{array}{l}
\text{Let }(t_n,u_n,z_n)_n\subset\mathrm{dom}(\calE)\text{ for every }n\in\N
\text{ satisfy semistability 
condition \eqref{semistab-z},}\\
\text{let }
t_n\to t,(u_n,z_n)\rightharpoonup(u,z)\text{ in }\Spu\times\Spz
\text{ with }\sup_n\calE(t,u_n,z_n)\leq C\text{ for all }t\in[0,T].\\
\text{Then, for every }
\tilde z\in\Spz\text{ there exists }\tilde z_n\rightharpoonup\tilde z
\text{ in }\Spz\text{ such that }\\
\limsup_{n\to\infty}\big(\calE(t_n,u_n,\tilde z_n)+\calR(\tilde z_n-z_n)-\calE(t_n,u_n,z_n)\big)
\leq  \calE(t,u,\tilde z)+\calR(\tilde z-z)-\calE(t,u,z)\,.
\end{array}\right\}
\end{split}
\end{equation}
\par
As previouly mentioned, the existence  results in \cite{RosTho15CEx} allow for 
a  non-smooth and even non-convex (in lower order terms) dependence $u 
\mapsto \ene tuz$. 
However, since the mechanical energies $\calE_k (t,\cdot,z)$ from \eqref{defEk} 
are convex and G\^ateaux-differentiable, we will confine the discussion 
to energies with this property and denote by $\partial_u \calE(t,\cdot,z) $  
the G\^ateaux-differential of the convex functional $\ene t{\cdot}z$. 
Following 
\cite[Thm.\ 4]{RosTho15CEx}, we need to impose 
a suitable   condition on the differentials $\partial_u \calE$ 
in the spirit of Minty's trick: 
\begin{equation}
\label{assMinty}
\begin{split}
&\hspace*{-3mm}\text{\bf Continuity:}\\
&\text{For all sequences }(\mathrm{t}_n)_n,\mathrm{t}_n:[0,T]\to[0,T],
(u_n)_n\subset L^\infty(0,T;\Spv)\cap H^1(0,T;\Spu),\\
&(z_n)_n\subset L^\infty(0,T;\Spx)\cap\mathrm{BV}([0,T];\Spz),
 (\partial_u \ene{\mathrm{t}_n}{u_n}{z_n})_n  \subset L^2(0,T;\Spu^*)\text{ s.t.\ }\\
&\exists\,C>0,\,\forall\,n\in\N,\,\forall\,t\in[0,T]:\;\calE(t,u_n(t),z_n(t))\leq C\,\text{ and }\\
&\left\{
\begin{array}{l}
\mathrm{t}_n\to\mathrm{t}\text{ pointwise a.e.\ in }(0,T)\,,\\
u_n\overset{*}{\rightharpoonup}u\text{ in }L^\infty(0,T;\Spv)\cap H^1(0,T;\Spu)\,,\\
z_n\overset{*}{\rightharpoonup}z\text{ in }L^\infty(0,T;\Spx)\,,\;
z_n(t)
\overset{*}{\rightharpoonup}  z(t)\text{ in }\Spx\text{ for all }t\in[0,T]\,,\\
 \partial_u \ene{\mathrm{t}_n}{u_n}{z_n}  \rightharpoonup\xi\text{ in }L^2(0,T;\Spu^*)\,,\;
\limsup_{n\to\infty}\int_0^T
\langle   \partial_u \ene{\mathrm{t}_n}{u_n}{z_n},  u_n\rangle_\Spu\,\mathrm{d}t
\leq\int_0^T\langle\xi,u\rangle_\Spu\,\mathrm{d}t\,,
\end{array}
\right\}\\
&\text{then there holds }\xi(t)  =  \partial_u\calE(\mathrm{t}(t),u(t),z(t))
\text{ for a.a.\ }t\in(0,T)\,. 
\end{split}
\end{equation}
\par
 Finally, 
to find a bound on the inertial term,  a further requirement of 
 \cite[Thm.\ 4]{RosTho15CEx} 
is the following 
\begin{equation}
\label{assSubgr}
\begin{split}
&\hspace*{-20mm}\text{\bf Subgradient estimate:}\\
&\hspace*{10mm}\text{There exists constants }C_3,C_4,C_5>0\text{ and }\sigma\in[1,\infty)\text{ such that }\\
&\hspace*{10mm}\forall\,(t,u,z)\in\mathrm{dom}(\calE)\, : \quad  
 \|\partial_u\calE(t,u,z) \|_{\Spu^*}^\sigma  \leq C_3\calE(t,u,z)+C_4\|u\|_\Spu +C_5\,.
\end{split}
\end{equation}
\par 
We are now in the position to recall the  existence result from \cite{RosTho15CEx}. 
\begin{theorem}[{\cite[Thm.\ 4]{RosTho15CEx}}]
\label{thm:RT15}
 Let 
  $\ingrsys$   
  fulfill 
  \eqref{basic-conds} \& 
\eqref{assE}--\eqref{assSubgr}. 

Then for every $(u_0,u_1,z_0)\in\Spu  \times\Spw  \times\Spz$
 fulfilling the semistability \eqref{semistab-z} at $t=0$, i.e.
\begin{equation}
\label{init-semistab}
\ene 0{u_0}{z_0} \leq \ene{0}{u_0}{\tilde z} + \dissr(\tilde{z}{-}z_0)
\qquad \text{for all } \tilde{z} \in \Spz
\end{equation}
there exists a \semi energetic solution (in the sense of Definition \ref{def:energetic-sol}) to $\ingrsys$ 
satisfying the Cauchy condition    
$(u(0),\dot u(0),z(0))=(u_0,u_1,z_0)$. 
\end{theorem}
\par
 Our existence result for the adhesive contact system $\ingrsysk$, deduced from Thm.\ \ref{thm:RT15}, 
also guarantees the validity of the energy-dissipation inequality as an \emph{equality}, cf.\ \eqref{eneq}. This 
would follow from applying \cite[Prop.\ 3.5, Thm.\ 3.6]{RosTho15CEx}, but, still, we prefer to sketch the 
 proof of \eqref{eneq} for the sake of completeness and also for later 
reference in the proof of Thm.\ \ref{thm:main}. 
Nonetheless,  let us mention that  \cite[Prop.\ 3.5, Thm.\ 3.6]{RosTho15CEx} 
in fact ensure that \emph{any} \semi energetic solution to the adhesive contact system complies with 
the energy-dissipation balance. 
%
\subsubsection*{Proof of Theorem \ref{th:exist-mech}: }
%
{\bf Ad \eqref{basic-conds}:} In view of \eqref{defRk}, \eqref{defV}, and \eqref{assCD} 
conditions \eqref{diss-V} and  \eqref{diss-R}  on $\calR_k$ and $\calV$ are verified.  From the definition of $\calE_k$ 
we see that the functionals have  the  proper domain $\mathrm{dom}(\calE_k)=[0,T]\times\Spv\times\Spx$. 
\par
{\bf Ad \eqref{assE}: }
 To check
\eqref{assEbdd},  
we calculate in view of \eqref{defEk}, using Korn's and Young's inequality: 
\begin{equation}
\label{calc-ek-bded}
\begin{aligned}
\calE_k(t,u,z)
&\geq &&  \tfrac{C_\mathbb{C}^1}{2}\|e(u)\|_{L^2}^2
-\|\mathbf{f}(t)\|_{\Spu^*}\|u\|_\Spu + \int_{\GC} \tfrac k2 z \left| \JUMP{u}\right|^2 \dd \Surf 
\\
& && \qquad \qquad \qquad
+\mathrm{b}_k(P(Z,\GC)+\|z\|_{L^1(\GC)})-(a_k^0+\mathrm{b}_k)\Surf(\GC)
\\
&\geq && \tfrac{C_\mathbb{C}^1 C_K^2}{2}\|u\|_{\Spu}^2
-\tfrac{C_\mathbb{C}^1 C_K^2}{4}\|u\|_{\Spu}^2
-\tfrac{1}{C_\mathbb{C}^1 C_K^2}\|\mathbf{f}(t)\|_{\Spu^*}^2 
+ \int_{\GC} \tfrac k2 z \left| \JUMP{u}\right|^2 \dd \Surf  
\\
& && \qquad \qquad \qquad
+\mathrm{b}_k\|z\|_{\SBV(\GC)}-(a^0+\mathrm{b})\Surf(\GC)\\
&\geq && -c_* C_\mathbf{f}^2 + \int_{\GC} \tfrac k2 z \left| \JUMP{u}\right|^2 \dd \Surf  -(a^0+\mathrm{b})\Surf(\GC)\,,
\end{aligned}
\end{equation}
where we used that $\|\mathbf{f}(t)\|_{\Spu^*}\leq C_\mathbf{f}$,  
as well as that $\|z\|_{L^1(\GC)}\leq \Surf(\GC)$ due to $z\in \{0,1\}$, 
 and set  
$\tfrac{1}{C_\mathbb{C}^1 C_K^2}=c_*$. 
This proves \eqref{assEbdd}. For  $\calE_k(t,u,z)\leq E$  we  then  find that 
\begin{equation}
\label{coercEk}
\|u\|_\Spu^2\leq \tfrac{4}{C_\mathbb{C}^1 C_K^2}(E+c_*C_\mathbf{f}^2 +(a^0+\mathrm{b})\Surf(\GC))
\;\text{ and }\;
\|z\|_{\SBV(\GC)}  \leq  \mathrm{b}_k^{-1}(E+c_* C_\mathbf{f}^2 +(a^0+\mathrm{b})\Surf(\GC))\,.
\end{equation}
 The weak lower semicontinuity property \eqref{assElsc} can be straightforwardly checked. 
\par  
\textbf{Ad 
\eqref{assEpow}:}  
Observe 
that $\partial_t\calE(t,u,z)=-\langle\dot{\mathbf{f}}(t),u\rangle_\Spu$. In view of 
the regularity assumption \eqref{assRegf} we have $\dot{\mathbf{f}}(t_n)\to\dot{\mathbf{f}}(t)$ 
in $\Spu^*$ for $t_n\to t$ in $[0,T]$, 
which immediately gives the upper semicontinuity property of the powers. 
In view of \eqref{coercEk} and Young's inequality we find the 
following  power-control  estimate:  
\begin{eqnarray*}
| \partial_t\calE_k (t,u,z)|
&\leq& C_\mathbf{f}\|u\|_\Spu 
\leq \tfrac{1}{2}C_\mathbf{f}^2+\tfrac{1}{2}\|u\|_\Spu^2  \\
&\leq& \tfrac{1}{2}C_\mathbf{f}^2
+\tfrac{2}{C_\mathbb{C}^1C_K^2}  (\calE_k(t,u,z)+c_* C_\mathbf{f}^2 +(a^0+\mathrm{b})\Surf(\GC))\,.
\\ 
\end{eqnarray*}
\par  
{\bf Ad \eqref{assDisscoercE}: } The coercivity assumption on 
the sum of $\calE_k$ and $\calV$ directly follows from the coercivity of $\calV$ and the 
coercivity estimate \eqref{coercEk} deduced above for $\calE_k$. 
\par 
{\bf Ad \eqref{assMRSC}: }
We refer to  \cite[Sec.\ 5.2]{RosTho12ABDM}  for the construction of a mutual recovery sequence 
that respects the unidirectionality imposed 
by $\calR_k$ and the perimeter regularization. 
\par 
{\bf Ad \eqref{assMinty}: }
Recall that, 
for all $(t,u,z)\in\mathrm{dom}(\calE_k)$ the mapping $u\mapsto\calE_k(t,u,z)$ 
is G\^ateaux-differentiable, 
 with G\^ateaux-derivative  given by \eqref{GderEk}. 
Due to the quadratic nature of   $\calE_k(t,\cdot,z)$, it is easy to verify the continuity condition \eqref{assMinty}. 
\par 
{\bf Ad \eqref{assSubgr}: } 
Using \eqref{assRegf} and H\"older's inequality to estimate the terms in \eqref{GderEk} 
we thus obtain for every 
$(t,u,z)\in\mathrm{dom}(\calE_k)$ with $\calE_k(t,u,z)\leq E$ and for all $v\in \Spu$ 
\begin{equation}
\label{subgrest}
\begin{split}
&
|\langle\partial_u\calE_k(t,u,z),v\rangle_\Spu|
\\
&\leq 
(C_\mathbb{C}^2\|e(u)\|_{L^2}+C_\mathbf{f})\|v\|_\Spu
+k\big(\int_{\GC}z\big|\JUMP{u}\big|^2\,\mathrm{d}\Surf\big)^{1/2}
 \big(\int_{\GC}\big|\JUMP{v}\big|^2\,\mathrm{d}\Surf\big)^{1/2}\\
&\overset{\text{\small (1)}}{\leq}
\|v\|_\Spu\big(C_\mathbf{f}
+2\sqrt{k}\max\{\bar{C},4C_\mathbb{C}^2/C_\mathbb{C}^1\}
\big(1+\tfrac{C_\mathbb{C}^1}{4}\|e(u)\|_{L^2}^2+k\int_{\GC}z\big|\JUMP{u}\big|^2\dd \Surf\big)^{1/2}
\big)\\
&\overset{\text{\small (2)}}{\leq}
\|v\|_\Spu\big(C_\mathbf{f}
+ 4\sqrt{k}\max\{\bar{C},4C_\mathbb{C}^2/C_\mathbb{C}^1\}
\big(1+\calE_k(t,u,z) 
\big)\big)\,.
\end{split}
\end{equation}
 where (1) follows using that $k\geq1$ and the relation $\sqrt{a}+\sqrt{b}\leq 2\sqrt{a+b}$ for 
$a,b\geq0$; there, $\bar C$ is the constant associated with the continuous 
embedding $\Spu \subset L^2(\GC;\R^d)$. Estimate (2) is obtained using the 
fact that $(1+\ldots)^{1/2}\leq(1+\ldots)$ together with 
the fact that $\calE_k(t,u,z)$ bounds both 
$\tfrac{C_\mathbb{C}^1}{4}\|e(u)\|_{L^2}^2$ and $ \int_{\GC}\tfrac k2 z\big|\JUMP{u}\big|^2\dd \Surf$, 
cf.\ \eqref{calc-ek-bded}. 
\par 
{\bf Energy equality \eqref{eneq}: } 
For each $k\in\N,$ we observe that $\dot u_k\in L^2(0,T;\Spu)\cap H^1(0,T;\Spu^*)$ is an admissible test function 
in the $k$-momentum balance. Thus, applying this test and integrating the $k$-momentum balance 
over $[0,t]$ for any $t\in[0,T],$ yields, in view of \eqref{defV}, for the term resulting from the viscous damping that  
\begin{equation}
\label{visc-damping}
\int_0^t
 \int_{\Omega{\setminus}\GC} \mathbb{D} e(\dot u_k(s)) {\colon} e(\dot u_k(s)) \,  
\mathrm{d}s= \int_0^t2\calV(\dot u_k(s))\,\mathrm{d}s\,.
\end{equation}
For the external loading term we find by partial integration
\begin{equation*} 
\int_0^t\langle-\mathbf{f}(s),\dot u_k(s)\rangle_\Spu\,\mathrm{d}s=
\langle-\mathbf{f}(t), u_k(t)  \rangle_\Spu- \langle-\mathbf{f}(0), u_k(0)\rangle_\Spu
-\int_0^t\langle-\dot{\mathbf{f}}(s),u_k(s)\rangle_\Spu\,\mathrm{d}s\,.
\end{equation*}
Moreover, since $(\Spu,\Spw,\Spu^*)$ is a Gelfand triple, the inertial term satisfies 
the following chain rule:  
\begin{equation*}
\int_0^t\langle\varrho\ddot u_k(s),\dot u_k(s)\rangle_\Spu\,\mathrm{d}s
=  \tfrac{1}{2}\|\dot u_k(t)\|_\Spw^2-\tfrac{1}{2}\|\dot u_k(0)\|_\Spw^2\,.   
\end{equation*}
We also observe that the elastic bulk energy satisfies a similar chain rule, i.e., we have 
\begin{equation}
\label{chain-ruleC}
\int_0^t\int_{\Omega\backslash\GC}\hspace*{-2mm}\mathbb{C}e(u_k(s)):e(\dot u_k(s))
\,\mathrm{d}x\,\mathrm{d}s
= \int_{\Omega\backslash\GC}\hspace*{-2mm}\tfrac{1}{2}\mathbb{C}e(u_k(t)):e(u_k(t))\,\mathrm{d}x
-\int_{\Omega\backslash\GC}\hspace*{-2mm}\mathbb{C}e(u_k(0)):e(u_k(0))\,\mathrm{d}x\,.
\end{equation}
In order to treat the term related surface energy 
$\int_0^t\int_{\GC}kz_k\JUMP{u_k}\JUMP{\dot u_k}\,\mathrm{d}\Surf\,\mathrm{d}s$ we proceed 
as in \cite[(4.69)-(4.75)]{Roub10TRIP}, cf.\ also \cite[(8.15)]{RosRou10TARA}: 
Using a well-chosen partition of the 
time interval $[0,t]$ a Riemann sum argument is applied to 
the semistability condition to deduce a chain-rule  type inequality  of the following form 
\begin{equation*}
\begin{split}
&\int_0^t\int_{\GC}kz_k(s)\JUMP{u_k(s)}\JUMP{\dot u_k(s)}\,\mathrm{d}\Surf\,\mathrm{d}s\\
&\leq \int_{\GC}kz_k(t)\big|\JUMP{u_k(t)}\big|^2\,\mathrm{d}\Surf
-\int_{\GC}kz_k(0)\big|\JUMP{u_k(0)}\big|^2\,\mathrm{d}\Surf+\Var_{\calR_k}(z_k,[0,t])\,.
\end{split}
\end{equation*}
 In view of the last estimate, putting 
all the above terms together in the $k$-momentum balance, 
results in the energy-dissipation  inequality opposite to \eqref{enineq} 
\begin{equation*}
 \begin{aligned}
\tfrac{1}{2}\|\dot u_k(t)\|_{\Spw}^2  & 
+ \int_0^t  2\calV(\dot u_k(s)) 
\,\mathrm{d} s + \Var_{\calR_k}(z_k, [0,t])+
 \calE_k(t,u_k(t),z_k(t))\\ 
&  \geq \tfrac{1}{2} \|\dot u_k(0)\|_{\Spw}^2
+  \calE_k(0,u_k(0),z_k(0)) + \int_0^t \partial_t\calE_k(s,u_k(s),z_k(s))\,\mathrm{d}s \,.
 \end{aligned}
 \end{equation*}
Thus, in combination with \eqref{enineq}, we have an equality that holds for all $t\in[0,T]$. 
Subtracting the energy equality given on $[0,s]$ from the one given on $[0,t],$ where $s<t,$ 
results in \eqref{eneq}. 
\par 
{\bf Uniqueness of the displacements for given $z\in L^\infty(0,T;\SBV(\GC;\{0,1\}))$:}  
Suppose that the pairs $(u,z)$ and $(\tilde u,z)$ both satisfy   
the momentum  balance \eqref{weak-mom-intro} and the same Cauchy conditions.   Then, 
it is immediate to check that 
$w:=u-\tilde u$ fulfills   \eqref{weak-mom-intro} for $\mathbf{f}=0$, with 
$w(0)=\dot w(0)=0$.  We  now choose the test function $v=\dot w$  and, 
exploiting the chain rule for each of the terms in \eqref{weak-mom-intro}, 
as well as  the positive definiteness  \eqref{assCD} of 
the viscosity and elasticity tensors $\bbD$ and $\bbC$ and Korn's inequality,
we obtain $\text{for a.a.\ }t\in(0,T)$
\begin{equation}
\begin{aligned}
\tfrac{1}{2}\|\dot w(t)\|_{\Spw}^2
+ C_{\bbD}^1C_K^2   \int_0^t  \| \dot w \|_{\Spu}^2 \dd s 
 +
\frac{ C_{\bbC}^1C_K^2}2 \|w(t) \|_{\Spu}^2 
 & \leq  k  \int_0^t   \int_{\GC} z \left| \JUMP{w}\right|   \left| \JUMP{\dot w}\right|  \dd \mathscr{H}^{d-1} \dd s\\ &  \leq  \frac{C_{\bbD}^1C_K^2}2    \int_0^t  \| \dot w \|_{\Spu}^2 \dd s
 +  C   \int_0^t  \| w \|_{\Spu}^2 \dd s\,,
 \end{aligned}
\end{equation}
where the second estimate follows from the arguments previously used for \eqref{subgrest}, 
and  Young's inequality.  We now  absorb the first term on the r.h.s.\ into the l.h.s., 
and use the Gronwall Lemma to deal with the second one. We thus conclude that $w\equiv0$ on $[0,t]$ for every $t\in (0,T]$, whence the desired uniqueness.  
\par 
{\bf Uniform bounds \eqref{unifbdsk}: }
Following the arguments of \cite[Prop.\ 6.3]{Miel08DEMF} using a Gronwall estimate 
and the boundedness of the given data, 
the power control condition \eqref{assEpow}  
yields that 
\begin{equation*}
\sup_{t\in[0,T]}\big( \tfrac{1}{2} \|\dot u_k(t) \|_{\Spw}^2+ 
\calE_k(t,u_k(t),z_k(t))+\int_0^t\calV(\dot u_k(s))\,\mathrm{d}s+\Var_{\calR_k}(z_k,[0,t])
\big) \leq C\,.
\end{equation*}
Again invoking \eqref{assEpow} thus gives estimates \eqref{unifbdE} \& \eqref{unifbdDiss}. 
From this we deduce \eqref{unifbdu} using \eqref{coercEk}. 
\par 
For \eqref{unifbdzDiss} we argue that \eqref{unifbdE} \& \eqref{unifbdDiss} imply for 
each $k\in\N$ and all $t\in[0,T]$ that 
$0\leq z_k(t)\leq1$ and that $z_k$ is monotonically decreasing in time. Hence, 
for every $t \in [0,T] $ 
\begin{equation*}
\| z_k \|_{\BV ([0,t]; L^1(\GC))}  
= \int_0^T \int_{\GC} (z_k(t) - z_k(0))  \,\mathrm{d}x\,\mathrm{d}t
\leq   \int_0^T\int_{\GC}(1-z_k(0))\,\mathrm{d}x\,\mathrm{d}t \leq C\,. 
\end{equation*} 
Since the coefficients $\mathrm{b}_k$ may tend to zero as 
$k\to\infty$ (cf.\  the scaling \eqref{coeffscale-b}),  we shall not deduce 
the uniform-in-time  estimate \eqref{unifbdz} for the perimeter terms
$P(Z_k(t),\GC)$ from the energy estimate \eqref{unifbdE} (taking into account that $\calE$ estimates the perimeter, cf.\
 \eqref{calc-ek-bded}). Instead, 
 we will resort to the  
 $k$-semistability inequality, implying 
  the following estimate 
\begin{equation*}
\mathrm{b}_kP(Z_k(t),\GC)\leq  \mathrm{b}_kP(\widetilde Z_k(t),\GC)+\calR_k(\tilde z_k-z_k)
+a_k^0\int_{\GC}\left(  z_k(t)-\tilde z_k\right) \,\mathrm{d}\Surf
\end{equation*}
for any finite perimeter set $\widetilde Z_k\subset Z_k(t),$ in the sense that 
its characteristic function $\tilde z_k$ satisfies $\tilde z_k\leq z_k(t)$ $\Surf$-a.e.\ in $\GC$. 
In view of the allowed scalings of the coefficients, cf.\ \eqref{coeffscale}, we may therefore 
cancel out the $k$-dependence of the coefficients by  multiplying  by 
a suitable power of $k$ and thus find
\begin{equation*}
\mathrm{b}P(Z_k(t),\GC)\leq  \mathrm{b}P(\widetilde Z_k(t),\GC)+\calR_1(\tilde z_k-z_k(t))
+a^0\int_{\GC} \left(z_k(t)-\tilde z_k\right)\,\mathrm{d}\Surf\,. 
\end{equation*}
Choosing the particular competitor $\widetilde Z_k=\emptyset,\tilde z_k=0$ a.e.\ in $\GC$ 
yields the uniform bound \eqref{unifbdz}.     
 Since  the delamination variables $z_k $ take values in $\{0,1\}$, 
the second bound in \eqref{unifbdz} is  immediate. 
\par 
 The uniform bound \eqref{unifbdcombi} on $ \varrho \ddot u_k+\partial_u\calJ_k(\cdot,u_k,z_k)$ follows 
by comparison in the $k$-momentum balance using that, for every $v \in \Spu$, 
\begin{equation*}
\begin{aligned}
 \Big| \pairing{}{\Spu}{ \varrho \ddot u_k+\partial_u\calJ_k(\cdot,u_k,z_k)}{v} \Big|  &  = 
\Big|\int_{\Omega\backslash\GC}(\mathbb{C}e(u_k)+\mathbb{D}e(\dot u_k)):e(v)\,\mathrm{d}x
-\langle\mathbf{f},v\rangle_\Spu\Big|
\\ & \leq \|v\|_\Spu\big(C_\mathbb{C}^2\|u_k\|_\Spu+C_\mathbb{D}^2\|\dot u_k\|_\Spu 
+ C_{\mathbf{f}} \big)\,, \end{aligned}
\end{equation*}
and the right-hand side is uniformly bounded by estimates \eqref{unifbdu} and \eqref{assdata}. 
\par      
{\bf $k$-dependent bounds \eqref{kdepbd}: }
The bound  \eqref{kdepsubgr} on the term $\partial_u \calJ_k (u_k,z_k)$ 
follows from the very same calculations developed in   \eqref{subgrest}.  
Then, \eqref{kdepddu} follows combining \eqref{kdepsubgr}  with the 
previously proved \eqref{unifbdcombi}. 
\hfill $\square$
%
\section{Passage from adhesive to brittle: Proof of Main Theorem \ref{thm:main}}
\label{s:4}
%
The proof of main Theorem \ref{thm:main} is carried out along the statements given in 
 Items 1-8. 
In particular, our roadmap for this section is the following:  
\begin{compactitem}
\item[{\bf 1.\ immediate convergences \eqref{convs}:}] 
they are proved in  Lemma \ref{statement1} below.
\item[{\bf Section \ref{FineProps}:}] We discuss the limit passage 
in the {\bf semistability inequality}. 
Subsequently, we explain fine properties of the semistable adhesive and 
brittle delamination variables, 
which are induced by the perimeter regularization in combination with the unidirectional, 
$1$-homogeneous dissipation potential. The most important feature is the so-called \emph{support-convergence} 
of the supports of the semistable adhesive variables $z_k$ to the support of the brittle limit $z$, 
cf.\ Prop.\ \ref{SuppConv} below. 
This property will be  at the basis of  the construction of recovery spaces 
 to pass to the limit in the momentum balance 
\emph{with inertia} of the adhesive systems and, in particular, 
to deduce sufficient compactness results for the inertial terms, cf.\ the forthcoming 
Prop.\ \ref{Propsrecsp}
and Lemma \ref{Limmombal}. 
\item[{\bf Section \ref{MoscoBC}:}] The definition of the  above mentioned 
recovery spaces relies on the construction of recovery sequences for the 
test functions of the momentum balance  of the target brittle system. 
We detail this construction in Prop. \ref{corconst}. From the latter we  deduce 
a \textsc{Mosco}-convergence statement for  the functionals $\calJ_k$ 
to the functional $\calJ_\infty$, which enforces the brittle constraint. 
Hence,  in Lemma \ref{BC} we infer the validity of 
the  {\bf 2.\ brittle constraint \eqref{propsuz}}  for the limit pair $(u,z)$.  
\item[{\bf Section \ref{RecSp}:}] By suitably adapting the construction ideas of 
the recovery sequence,  
we design a  suitable sequence of  \emph{recovery spaces} $\spyn$, cf.\   \eqref{recsp} below, 
and prove their density in the spaces $L^2(s,T;\Spu_z(s))$, with $s\in (0,T]$ arbitrary, 
in Prop.\ \ref{Propsrecsp}.
We also introduce the  spaces
 $\spyns$  tailored 
in such a way as to ensure compactness for the sequence 
$(u_k, \dot{u}_k, \ddot{u}_k)_k$,  hence, in particular for the  
inertial terms  $(\ddot{u}_k)_k$, by comparison arguments 
in the momentum balance. 
The art in the construction  of these recovery spaces thus lies in the fact that 
their definition bypasses
the blow-up of the adhesive contact terms in the momentum  balance, 
without losing the information on the 
crack set of the brittle limit (i.e., the support of the brittle delamination variable).       
\item[{\bf Section \ref{LimitMomBEnDiss}:}]Using the recovery spaces we carry 
out the compactness argument for the 
inertial terms and pass to the {\bf limit in the momentum balance}, cf.\ Lemma \ref{Limmombal}. 
We thus conclude that $(u,z)$ is a {\bf \semi energetic solution of the brittle system}.   
In the line of these arguments we will also obtain the {\bf 3.\ additional regularity \eqref{better-H2}} as well as  
the {\bf 4.\ weak temporal continuity of $\dot u$}, cf.\ Cor.\ \ref{WeakConti}.     
\item[{\bf Section \ref{LEnbal}:}] We verify that the limit pair satisfies the 
{\bf 5.\ energy-dissipation balance}.  
In a first step, the energy-dissipation inequality \eqref{enineq} 
will be deduced via lower semicontinuity arguments. 
To find the reverse inequality requires to prove a chain rule for the 
inertial terms, cf.\ Lemma \ref{ChainruleLemma}. 
The remaining statements   of Theorem \ref{thm:main}, i.e.\ 
the {\bf 6.\ enhanced convergences \eqref{enh-convs}}, 
 the {\bf 7.\ enhanced initial condition \eqref{enhanced-Cauchy-dotu}}, and the 
{\bf 8.\ uniqueness of the displacements 
for given $z$}, are  deduced   
from the energy balance in Lemmata  \ref{EnhancedConvs}, \ref{EnhancedCauchy}\,\&\,\ref{Unique}, 
respectively. 
\end{compactitem}
\medskip

\par
We now resume to verify statement {\bf 1.} of Theorem \ref{thm:main}, i.e., the 
convergence of (a subsequence of) the \semi energetic solutions of the 
adhesive contact systems in the sense of \eqref{convs} to a limit pair $(u,z)$. 
In \underline{all of the forthcoming   results} we will tacitly require that 
\[
\text{the assumptions of Theorem \ref{thm:main} hold true.}
\]
We start by proving convergences  \eqref{convs}, relying on 
compactness arguments based on the bounds \eqref{unifbdsk}. 
\begin{lemma}[Statement 1.\ of Theorem \ref{thm:main}: convergences \eqref{convs}] 
\label{statement1}
There is a subsequence $(u_k,z_k)_k$ of the adhesive contact systems 
${\ingrsysk}_k$ and a limit $(u,z)$ such that {\bf convergences \eqref{convs}} hold true.
\end{lemma}
\begin{proof}
The uniform bound \eqref{unifbdDiss} 
allows us to find a subsequence $(u_k)_k$ and a limit $u$ such that 
$u_k\rightharpoonup u$ in   $H^1(0,T;\Spu)$.
Since $H^1(0,T;\Spu) \Subset \mathrm{C}^0_{\mathrm{weak}}([0,T];\Spu)$ 
(the latter being the space of weakly continuous functions with values in $\Spu$) 
by  Aubin-Lions type arguments, cf.\ \cite{simon86},    we also
  conclude the second of \eqref{convsu}.  Hence $u(0)=u_0$ in view of \eqref{initial-data-conv}.  
\par 
In view of the bound for $(z_k)_k$ in $\BV([0,T]; L^1(\GC))$, 
and taking into account the continuous embedding of $L^1(\GC)$ 
into the space $\mathrm{M}(\GC)$ of Radon measures on $\GC$, we may apply
 a suitable version of Helly's selection principle for functions with values 
in the dual of a separable Banach space, cf.\ 
 \cite[Lemma 7.2]{DMDSMo06QEPL},  and 
find a (not relabeled) subsequence $(z_k)_k$ and a limit function $z$ such that 
$z_k(t)\rightharpoonup z(t)$ in $\mathrm{M}(\GC) $ for every $t\in[0,T]$. 
Taking into account that the functions $z_k $ are uniformly bounded in $L^\infty (0,T;\SBV(\GC;\{0,1\}))$, a fortiori we conclude the pointwise convergence of $z_k(t)$, for all $t\in [0,T]$,  w.r.t.\ the weak$^*$ topology of $\SBV(\GC;\{0,1\}) \cap L^\infty (\GC)$, i.e. the second   of 
\eqref{convsz-2}.  Taking into account that $ \SBV(\GC;\{0,1\})   \Subset L^1(\GC)$, we have that $z_k(t) \to z(t)$ strongly in $\Spz$ and, ultimately, by the bound in $L^\infty (\GC)$ we infer strong convergence in $L^q(\GC)$ for all $1\leq q<\infty$. This gives \eqref{convsz}. 
Finally, the Aubin-Lions type compactness result from \cite[Cor.\ 7.9]{Roub05NPDE}  
combined with  the Banach Alaouglu Bourbaki theorem also ensures that, up 
to a further subsequence, there exists $\tilde z  \in L^\infty((0,T)\times \GC)$ 
such that $z_k \to \tilde z$ weakly$^*$ in $L^\infty((0,T)\times \GC)$ 
and strongly in $L^q((0,T)\times \GC)$  for all $1\leq q<\infty$. 
Therefore, $\tilde z = z$ a.e.\ in $(0,T)$. This gives the first of \eqref{convsz-2}. 
\par
Finally, the existence of an element $\lambda\in L^2(0,T;\Spu^*)$ and convergence \eqref{convsddu} 
follow from  \eqref{unifbdcombi}.
\end{proof}
%
\subsection{Limit passage in the semistability condition \& fine properties of semistable sets for 
perimeter-regularized models with unidirectionality}
\label{FineProps}
%
The limit  passage in the semistability condition  as $k\to\infty$ results from the 
construction of a mutual recovery sequence in correspondence with the sequence $(u_k, z_k)_k$, 
converging to the candidate \semi energetic solution $(u,z)$  of the brittle model as 
specified in  \eqref{convs}. 
Namely, we  show that for every $\tilde z \in \Spx$ there exists a 
sequence $(\tilde {z}_k)_k \subset \Spx$ such that for every $t\in [0,T]$ 
\[
\limsup_{k\to\infty} \left( \calE_k(t,u_k(t), \tilde{z}_k) 
{-}  \calE_k(t,u_k(t), z_k(t)) {+} \calR_k(\tilde{z}_k {-} z_k(t)) \right) 
\leq  \calE_\infty(t,u(t), \tilde{z}) {-}  \calE_\infty(t,u(t), z(t)) 
{+} \calR_\infty(\tilde{z} {-} z(t)),
\]
so that the positivity of the l.h.s.\ in the above inequality, 
granted by the semistability inequality for the adhesive system, entails 
the positivity of the r.h.s., hence the semistability condition in the brittle limit. 
We refer to   \cite[Sec.\ 5.2]{RosTho12ABDM} for all the details of the construction; 
regarding scaling \eqref{coeffscale-b} we also point 
to \cite[Sec.\ 7]{RosTho12ABDM} 
and to \cite{RoThPa15SDLS}. In this way, we conclude 
\begin{lemma}[Limit passage in the semistability condition] 
The  limit pair $(u,z)$ extracted by convergences \eqref{convs} satisfies the semistability inequality \eqref{semistab-z} 
for the brittle system $\ingrsysinf$.    
\end{lemma}
We now discuss consequences of the above semistability result, which are true 
\underline{for all $k\in\N\cup\{\infty\}$}, 
since they rely 
on the unidirectionality of the delamination process encoded in the 
1-homogeneous dissipation potentials
$\calR_k$ \eqref{defRk}, resp.\ $\calR_\infty$  \eqref{defRinfty}. 
In fact, it was already observed in \cite[Sec.\ 6, (6.5)]{RosTho12ABDM}, 
that the unidirectionality allows us to extend the semistability 
inequality for $k\in\N\cup\{\infty\}$ 
(cf.\ \eqref{semistab-expl-brittle}--\eqref{semistab-expl-adh}), 
to a more general inequality that compares the perimeter of 
(semi-)stable sets  
and their competitors with 
their volume difference, cf.\ \eqref{SSZ} below.      
\begin{lemma}[Consequence of semistability]
\label{Conssemi}
Let $t\in[0,T]$ be fixed and, for $k\in\N \cup \{\infty\}$ 
  let 
$z_k(t)$ be semistable for $\calE_k(t,u_k(t),\cdot)$ in the sense of 
\eqref{semistab-z}.  
Then the finite-perimeter set 
 $Z_k(t) $  with characteristic function   $z_k(t)$  
satisfies the following inequality for all $\widetilde Z\subset   Z_k(t) $: 
\begin{equation}
\label{SSZ}
\mathrm{b}_k {P}(Z_k(t),\GC)
\le \mathrm{b}_k {P}(\widetilde Z,\GC)
+  (a_k^0\!+\! a_k^1) \calL^{d-1}(Z_k\backslash\widetilde Z)\,. 
\end{equation}
\end{lemma}
It was deduced in \cite[Thm.\,6.3]{RosTho12ABDM} that finite-perimeter sets 
satisfying \eqref{SSZ} have an additional regularity property, 
introduced by Campanato as the Property $\frak{a},$ 
cf.\ e.g.\ \cite{Camp63PHAC,Camp64PUFS,Giaq83MICV,Grie02LEBN} 
and called \emph{lower density estimate} in e.g.\ \cite{FonFra95RBVQ,AmFuPa05FBVF}: 
\begin{proposition}[Lower density estimate for semistable sets]\label{Propa}
 Keep $t\in[0,T]$ fixed and assume that the 
finite-perimeter set   $Z_k(t)\subset\GC$  satisfies \eqref{SSZ}, 
with $k\in\N \cup \{\infty\}$.  Then, for all $k\in\N \cup \{\infty\},$ $Z_k(t)$ 
complies with the following \emph{lower density estimate}: 
there are constants $R,\frak{a}(\GC)>0$ 
depending solely on $\GC\subset\R^{d-1}$, on $d$, and on 
the parameters  $a_k^0,a_k^1, \mathrm{b}_k>0,$  such that    
\begin{align}
\label{propa}
\forall\,y\in\supp z_k(t) \ \forall\,\rho_\star>0:\quad 
\calL^{d-1}(Z_k(t)\cap B_{\rho_\star}(y))\ge\begin{cases}
\frak{a}(\GC)\rho_\star^{d-1}&\text{if }\rho_\star< R,\\
\frak{a}(\GC)R^{d-1}&\text{if }\rho_\star\geq R.
\end{cases}
\end{align}
Here, $B_{\rho_\star}(y)$ denotes the open ball of radius $\rho_\star$ with center 
in $y$ and the support of the $\mathrm{SBV}$-function $z_k(t)$ is defined as in \eqref{defsupp}.  
\end{proposition}

Let us point out that sets satisfying the lower density estimate \eqref{propa}, 
are sometimes also called $(d{-}1)$-thick, see e.g.\ 
\cite{Lehr08PHIU,EgHaRe15HIFV}. The proof of Proposition \ref{Propa} is 
carried out by contradiction to \eqref{SSZ}. The lower bound 
$\frak{a}(\GC)\rho_\star^{d-1},$ which holds uniformly for all radii 
$\rho_\star$ and at every point of   $\supp z_k(t)$,   is obtained with the aid of a 
uniform relative isoperimetric inequality  proved  in 
\cite[Thm.\ 3.2]{Tho13UPSI}. In turn, for the proof of the latter result 
it is crucial that $\GC$  is  convex. 
\par
Now, let us highlight a simple consequence, yet crucial  for  our arguments, 
  of Proposition \ref{Propa} in Lemma \ref{key-lemma} below.  
It involves the \emph{essential closure} of the  semistable sets $Z_k$, 
$k \in \N \cup\{\infty\}$. We recall that the essential closure of a 
measurable set $E\subset \GC$ is defined as follows 
(cf.\ e.g. \cite[p.\ 21]{Pfeff01DI}): 
\begin{equation}
\label{essential-closure}
\mathrm{cl}_*E:=\Big\{
x\in\R^{d-1}\, : \ \limsup_{r\to0}\frac{\calH^{d-1}(E\cap B_r(x))}{\calH^{d-1}(B_r(x))}>0
\Big\} \,.
\end{equation}
Let us point out that the set $\mathrm{cl}_*E$ is not necessarily (topologically)  closed. 
However, the following key property holds 
(cf.\ \cite[Cor.\ 1.5.3]{Pfeff01DI})
\begin{equation}
\label{essentially-equal}
\calH^{d-1}((E\backslash\mathrm{cl}_*E)\cup(\mathrm{cl}_*E\backslash E))=0\,.
\end{equation}
\begin{lemma}
\label{key-lemma}
 Keep $t\in[0,T]$ fixed and, for  $k\in\N \cup \{\infty\}$  
  let 
$z_k(t)$ be semistable for $\calE_k(t,u_k(t),\cdot)$ in the sense of 
\eqref{semistab-z},  
with associated  finite-perimeter set 
 $Z_k(t)$.  Then,
\begin{equation}
\label{support-closure}
\supp z_k(t) \subset \mathrm{cl}_*Z_k(t) 
\end{equation}
and, therefore, 
\begin{equation}
\label{essentially-eq-Zk}
\Surf( \supp z_k(t) {\setminus} Z_k(t)) =0.
\end{equation}
\end{lemma}
Observe that, since $ \mathrm{cl}_*Z_k(t)$ need not be  (topologically)  closed,   
\eqref{support-closure} does not follow from the definition \eqref{defsupp} 
of $\supp z_k(t)$, which guarantees $\supp z_k(t) \subset A$ for every  set such 
that $\Surf (Z_k(t) {\setminus} A)=0$, provided that $A$ is closed. 
Indeed, \eqref{support-closure} is due to the lower density estimate \eqref{propa} in the 
case $\rho_\star\in(0,R)$, yielding
\[
\forall\,y\in\supp z_k(t):\quad 
\lim_{\rho_\star\to0}\frac{ \calH^{d-1}(Z_k(t)\cap B_{\rho_\star}(y))}{\calH^{d-1}(B_{\rho_\star}(y))}
\geq \frac{\frak{a}(\GC)}{\calH^{d-1}(B_1(0))} >0, 
\]
since $\calH^{d-1}(B_{\rho_\star}(y))/\calH^{d-1}(B_1(0)) = \rho_\star^{d-1}$. 
Then, \eqref{essentially-eq-Zk}  directly ensues from   \eqref{essentially-equal} 
combined with \eqref{support-closure}. 
\par
The second, key consequence of the  lower density estimate \eqref{propa} is  a
\emph{support convergence} result, proved in  \cite{RosTho12ABDM} 
and recalled in Prop.\ \ref{SuppConv} below, which  
  further strengthens the convergence of the
delamination variables $z_k$ for the adhesive contact models. 
In fact, it  states one part for Hausdorff convergence of  the supports of the
 $(z_k)_k$
the support of $z$, namely
\begin{proposition}[Support convergence {\cite[Thm.\ 6.1]{RosTho12ABDM}}]
\label{SuppConv}
Let $t\in[0,T]$ be fixed.  
For all $k\!\in\!\N\cup\{\infty\}$ assume that the finite-perimeter sets 
$Z_k(t)\subset\GC$ satisfy \eqref{SSZ} and that the associated characteristic 
functions $z_k(t)\overset{*}{\rightharpoonup} z(t)$ in $\mathrm{SBV}(\GC,\{0,1\})$  
for some $z\in L^\infty (0,T; \mathrm{SBV}(\GC,\{0,1\}))$. 
 For all $k\!\in\!\N$ set
\begin{equation}
\label{defrhok}
 \rho(k,t):  =\inf\big\{\rho>0:\,\supp z_k(t)\subset\supp z(t)+B_\rho(0)\big\}. 
\end{equation}
Then 
\begin{equation}
\label{suppconv}
 \supp z_k(t)\subset \supp z(t)+B_{\rho(k,t)}(0)
\quad\text{and}\quad
\rho(k,t)\to0\text{ as }k\to\infty.  
\end{equation} 
In particular, if $\supp z(t)=\emptyset,$ then also $\supp z_k(t)=\emptyset$ for all 
$k\geq k_0$ from a particular index $k_0\in\N$ on.
\end{proposition}
The counterpart to \eqref{suppconv}, namely 
$\supp z(t) \subset \supp z_k(t) + B_{\tilde{\rho}(k,t)}(0)$ 
with $\tilde{\rho}(k,t) \to 0$ as $k\to\infty$, can be obtained 
directly from the pointwise  strong $L^1(\GC)$-convergence of the sequence $(z_k)$, cf.\ 
\eqref{convsz}, 
 see \cite[Cor.\ 6.8]{RosTho12ABDM} for the proof. 
\subsection{Recovery test functions for the momentum balance and proof of the brittle constraint}
\label{MoscoBC}
The limit passage in the   momentum balance for the adhesive system \eqref{weak-mom-intro}
as $k\to\infty$ 
requires, for each test function  $v\in\Spu_z(t),$ with $t\in (0,T)$ fixed, 
the construction of a \emph{recovery sequence} $(v_k)_k$ fulfilling  
the following (minimal) convergence properties as $k\to\infty$:
 $\int_{\GC} k z_k(t) \JUMP{u_k(t)}  \JUMP{v_k} \dd \Surf \to 0$   
and   $v_k\to v$ in 
$\Spu=H^1(\Omega\backslash\GC;\R^d)$. Based on the knowledge of support convergence 
from Prop.\ \ref{SuppConv} these features 
can be guaranteed by a  construction of recovery sequences for 
the test functions  of \eqref{weak-mom-intro}, developed 
in \cite[Prop.\ 8, Cor.\ 2]{MiRoTh10DDNE}. 
Since this construction will also the starting point for proving that the union 
of the recovery spaces $\spyn$
is dense in $L^2(s,T;\VZ(s)^*)$ for every $s\in (0,T)$, 
we shall illustrate it in detail it in the ensuing Proposition \ref{corconst}.  
We will state  the latter result
for a  fixed $z\in L^\infty(\GC)$ and we will later apply it to 
$z(t)$, $t\in [0,T]$ fixed, with $z\in L^\infty(0,T;\SBV(\GC))$ 
a limiting curve for the sequence $(z_k)_k$
of the adhesive contact delamination variables.
\par
The construction of the recovery test functions (for the adhesive momentum balance)  
is based on the fact that 
any function $v\in\VZ = \{ v \in \Spu\, : \ \JUMP{v}=0 \text{ on } \supp z \} $   can be 
written in terms of its symmetric $v_\mathrm{sym}$ and its antisymmetric part 
$v_\mathrm{anti}$
 with respect to the plane $x_1=0$.    Rewriting any $x\in\Omega$ as $x=(x_1,y)$ for 
$y=(x_2,\ldots,x_d)\in\R^{d-2},$ this is 
\begin{equation}
\label{antisym}
\begin{alignedat}{2}
v(x_1,y)&:=v_\mathrm{sym}(x_1,y)+v_\mathrm{anti}(x_1,y)&&\in  \VZ, \;
\text{ with }\\ 
v_\mathrm{sym}(x_1,y)&:=\tfrac{1}{2}(v(x_1,y)+v(-x_1,y))&&\in H^1(\Omega;\R^d)\;\text{ and }\\
v_\mathrm{anti}(x_1,y)&:=\tfrac{1}{2}(v(x_1,y)-v(-x_1,y))&&
\in H^1((\Omega\backslash\GC)\cup  \supp z;\R^d),
\end{alignedat}
\end{equation}
where we assume here and in what follows that the domain $\Omega$ is oriented 
in a coordinate system such that the origin is contained in $\GC$ and 
the normal $\mathrm{n}$ to $\GC$ points in 
$x_1$-direction,  cf.\ Figure \ref{Fig:setting} on p.\ \pageref{Fig:setting}. 
 The following result gives
 the definition of the   recovery sequence and its convergence properties. 
\begin{proposition}[Recovery sequence for the test functions, 
\cite{MiRoTh10DDNE,RosTho12ABDM,RoThPa15SDLS}]
\label{corconst}
Consider  $z\in L^\infty(\GC)$ and let  
 $M:= \supp z$
 be a $(d{-}1)$-thick subset of $\GC$. 
  Let 
  \begin{equation}
  \label{def-dM}
   d_{M}(x):=\min_{\hat x\in M}|x-\hat x| 
  \quad \text{for all
$x\in\overline{\Omega}_\pm$.}
\end{equation}
Let  $v\in H^1(\Omega_-\cup M\cup\Omega_+;\R^d)$,  such that  $v=0$ on
$\GD$ in the trace sense. 
With $\xi_\rho^{M}(x):=\min\{\frac{1}{\rho}(d_{M}(x)-\rho)^+,1\}$ set
\begin{align}
\label{testconst} 
\begin{aligned}
r(\rho, M, v):=v_M^\rho(x_1,y):  & =v_{\mathrm{sym}}(x_1,y)+
\xi_\rho^{M}(x_1,y)\,v_{\mathrm{anti}}(x_1,y)
\\
&
  \quad \text{for all } \rho>0 \text{ and } v \in   
 H^1(\Omega_-\cup M\cup\Omega_+;\R^d),
 \end{aligned}
\end{align}
with $v_\mathrm{sym}$ and $v_\mathrm{anti}$ as in \eqref{antisym}. 
Then, the following statements hold:
\begin{itemize}
\item [(i)]
$v^\rho\to v$ strongly in   $
H^1(\Omega_-\cup\Omega_+;\R^d)$ 
as $\rho\to0,$
\item [(ii)]
 $v\in 
H^1(\Omega_-\cup M\cup\Omega_+,\R^d)$
$\Rightarrow$ $v^\rho\in 
H^1(\Omega_-\cup(
M{+}B_{\rho}(0))\cup\Omega_+;\R^d)$. 
\end{itemize}
\end{proposition} 
Later on, we will apply Prop.\ \ref{corconst} to $z=z(t)$, $t \in [0,T]$ 
fixed and  $z\in L^\infty(0,T;\SBV(\GC))$ a limiting curve for the sequence $(z_k)_k$,
with  $ \rho= \rho(k,t)$ the sequence of radii for which support convergence holds. 
We thus obtain for every test function $v \in \Spu_z(t)$  a 
sequence $(v_k = v^{\rho(k,t)})_k$ which converges  to $v$ 
even strongly  in $\Spu=H^1(\Omega\backslash\GC;\R^d)$ and which in fact has $\JUMP{v_k}=0$ 
on $\supp z(t) + B_{\rho(k,t)}(0) \supset \supp z_k(t)$. 
\par
 While referring to \cite[Prop.\ 8, Cor.\ 2]{MiRoTh10DDNE}, 
\cite[Prop.\ 5.4]{RosTho12ABDM}, and \cite[Sec.\ 4.1]{RoThPa15SDLS} 
for more details and the proof of Prop.\ \ref{corconst}, 
let us briefly hint at the main underlying tools:
Observe  that   $v_\mathrm{anti}=0$ on $\supp z$. 
Hence, in view of \eqref{ass-domain1}, 
$v_\mathrm{anti}|_{\Omega_\pm}=:v_\mathrm{anti}^\pm\in H^1(\Omega_\pm;\R^d)$ 
satisfies homogeneous Dirichlet conditions on the closed set 
$M= \supp z \subset\GC,$ i.e.\ 
$v\in H^1_M(\Omega_\pm;\R^d)$. This observation is essential, because 
it enables us to apply a Hardy's inequality, stating the existence of a constant 
$C_M>0$ such that for all $v\in H^1_M(\Omega_\pm;\R^d)$: 
\begin{equation}
\label{Hardy}
\big\|v/d_M\big\|_{L^2(\Omega_\pm,\R^d)}\leq 
C_M\big\|\nabla v\big\|_{L^2(\Omega_\pm,\R^{d\times d})}\,,
\end{equation} 
 with $d_M$ from \eqref{def-dM}.  
Such type of Hardy's inequality is the crucial tool allowing us to verify the strong 
$H^1(\Omega_\pm;\R^d)$-convergence of the recovery sequence under construction. 
\par
It has to be stressed that, to our knowledge,  so far the above Hardy's 
inequality for closed sets $M$ of arbitrarily low regularity has been proved 
only in $L^p$-spaces with $p>d,$  see \cite[p.\ 190]{Lewi88UFS}. 
This is essentially the reason
 why, 
 in the works 
\cite{MiRoTh10DDNE,RosTho12ABDM},
Proposition \ref{corconst}
was proved in 
$W^{1,p}$ with $p>d$,  only,  in the setting of 
rate-independent 
delamination coupled to a viscous evolution of the bulk for materials with constitutive relations 
of $p$-growth with $p>d$.  
\par
Only recently, Hardy's inequality \eqref{Hardy}   has been obtained  
in \cite{EgHaRe15HIFV} under much weaker integrability assumptions on 
the displacements, with only slightly strengthened regularity 
assumptions on the closed set $M$. 
More precisely, the additional regularity imposed on $M$ 
in \cite[Thm.\ 3.1]{EgHaRe15HIFV}
for Hardy's inequality to hold, is the lower density estimate \eqref{propa}; 
exactly the fine regularity property deduced in Proposition \ref{Propa} 
for finite-perimeter sets being semistable in the sense of \eqref{SSZ}. Thus, due to these 
recent results, \cite[Thm.\ 3.1]{EgHaRe15HIFV} in combination 
with \cite[Thm.\ 6.1]{RosTho12ABDM}, we are now able to perform the limit passage 
from adhesive to brittle without an additional $W^{1,p}$-regularization, where $p>d,$ 
for the displacements (and the assumption $p>d$ becomes unnecessary also in \cite{RosTho12ABDM}). 
The adhesive-to-brittle limit  has been addressed in the recent  \cite{RoThPa15SDLS}, 
in the context of \emph{local solutions}
to
\emph{fully rate-independent} delamination, with material laws in the (static) 
momentum balance featuring the general growth exponent $p\in(1,\infty)$. 
\begin{remark}
\label{rmk:subtle}
\upshape  
Observe though, that for $p>d$ it is $\JUMP{v}\in \mathrm{C}^0(\barGC,\R^d)$ for 
any $v\in W^{1,p}(\Omega\backslash\GC,\R^d)$. 
Thus, if $z|\JUMP{v}|=0$ a.e.\ on $\GC$  
for a given function $z\in L^\infty(\GC),$ then in particular 
$\JUMP{v}\equiv0$ on $\supp z$.  
This conclusion is no longer valid for $p\leq d$ and therefore the above property is 
directly incorporated in the definition of $\Spu_z$ in \eqref{spu-z}. This is 
essential, because we will exploit the support convergence \eqref{suppconv} 
for the construction of the recovery sequence and, for this, the usage 
the \textit{closed} set $\supp z$ is important. 
\par
This fact  also motivates the definition of  the functional $\calJ_\infty = \calJ_\infty (u,z)$ 
from \eqref{brittle-constr-funct} in terms of the constraint 
$\JUMP{u}=0$ a.e.\ on $\supp z$, which, in this context, is stronger than just 
requiring $z \JUMP{u}=0$. 
\end{remark}
A first consequence  of Prop.\ \ref{corconst}, joint with  Lemma \ref{key-lemma}, 
is the \textsc{Mosco}-convergence of the functionals $\calJ_k(\cdot, z_k)$ from 
\eqref{JKfunc} to $\calJ_\infty(\cdot,z)$, with $(z_k)_k$ converging to $z$ 
weakly$^*$ in $\SBV(\GC;\{0,1\})$ and \underline{$z_k,\, z$ satisfying 
inequality \eqref{SSZ}}. In fact, \eqref{SSZ} guarantees the 
lower density estimate  \eqref{propa}, which in turn ensures the support 
convergence \eqref{suppconv} and thus Prop.\ \ref{corconst}, at the core of the 
proof of the $\limsup$-inequality. Estimate   \eqref{propa} also yields the 
validity of Lemma \ref{key-lemma}, at the basis of the proof of  the $\liminf$-inequality. 
 We will only detail the proof of the latter estimate, which  
  in turn will allow us to deduce  the argument for showing the brittle constraint. 
\begin{lemma}
\label{l:Mosco}
Let  $(z_k)_k\subset\SBV(\GC;\{0,1\})$ fulfill
$z_k\overset{*}{\rightharpoonup}z$ in $\SBV(\GC;\{0,1\})$, and suppose that 
$z_k$ for every $k\in \N$, and 
 $z$ (and the associated finite-perimeter sets $Z_k,\, Z$) comply with   \eqref{SSZ}. Then,
the functionals $\calJ_k(\cdot, z_k):\Spu\to\R$ Mosco-converge 
to the functional $\calJ_\infty(\cdot,z):\Spu\to\R\cup\{\infty\}$ w.r.t.\ the topology of 
$H^1(\Omega{\setminus}\GC;\R^d)$,  i.e.,  there holds 
\begin{subequations}
\label{Mosco-precise}
\begin{compactenum}
\item[-- \textbf{$\liminf$-inequality:}] for every $u  \in H^1(\Omega{\setminus}\GC;\R^d)$
 and $(u_k)_k \subset H^1(\Omega{\setminus}\GC;\R^d)$ there holds
\begin{equation}
\label{mosco-liminf} u_k \weakto u \ \text{ weakly in
$H^1(\Omega{\setminus}\GC;\R^d)$} \  \Rightarrow \ \liminf_{k\to
\infty} \calJ_k(u_k,z_k) \geq \calJ_\infty (u,{z}); 
\end{equation}
\item[-- \textbf{$\limsup$-inequality:}]  for every
$v \in H^1(\Omega{\setminus}\GC;\R^d)$ there is a sequence $(v_k)_k\subset H^1(\Omega{\setminus}\GC;\R^d)$ such that
\begin{equation}
\label{mosco-limsup} v_k \to v \ \text{ strongly in
$H^1(\Omega{\setminus}\GC;\R^d)$} \  \text{and} \
\limsup_{k\to\infty} \calJ_k(v_k,z_k) \leq \calJ_\infty(v,z). 
\end{equation}
\end{compactenum}
\end{subequations} 
\end{lemma} 
\begin{proof}
In order to prove \eqref{mosco-liminf}, we may confine the discussion  
to the case $\liminf_{k\to\infty} \calJ_k(u_k,z_k) <\infty$. 
Therefore, up to a subsequence we have $\sup_{k\in \N}  \calJ_k(u_k,z_k) \leq C$ and there holds
\[
\int_{\GC} \tfrac12 z \left|\JUMP{u}\right|^2 \dd \Surf \leq \liminf_{k\to\infty} \int_{\GC} \tfrac12 z_k \left|\JUMP{u_k}\right|^2 \dd \Surf
\leq \lim_{k\to\infty} \frac{C}{k}=0,
\]
where the first inequality follows from combining the 
strong convergence $z_k\to z$ in $L^1(\GC)$  
(due to $\SBV(\GC;\{0,1\}) \Subset L^1(\GC)$), 
with the weak convergence $\JUMP{u_k} \weakto \JUMP{u}$ in $L^2(\GC;\R^d)$ 
(due to $u_k \weakto u $ in $H^1(\Omega{\setminus}\GC;\R^d)$ 
via trace theorems and Sobolev embeddings). Hence 
\begin{equation}
z \JUMP{u} = 0 \text{ a.e.\ on } \GC \ \Leftrightarrow \  \JUMP{u} = 0  
\text{ a.e.\ on }  Z  \ \Leftrightarrow \  \JUMP{u} = 0  \text{ a.e.\ on }  \supp z 
\end{equation}
the last implication due to the fact that $\supp z$ and $Z$ coincide, 
up to a $\Surf$-negligible set, thanks to \eqref{essentially-equal} 
and \eqref{essentially-eq-Zk}.  Then,
$\calJ_\infty (u,{z}) =0 \leq \liminf_{k\to\infty} \calJ_k(u_k,z_k) $. 
\par
The proof of the $\limsup$-inequality follows from adapting 
that  for  \cite[Prop.\ 5.4]{RosTho12ABDM}. 
\end{proof} 
\par 
We are now in the position to conclude the brittle constraint for the limit pair $(u,z)$ as a consequence of the lower $\Gamma$-limit 
\eqref{mosco-liminf}. For the time-derivative, i.e., the pair $(\dot u,z),$ 
the brittle constraint will be obtained arguing on difference quotients. 
\begin{lemma}[Brittle constraint \eqref{propsuz}]
\label{BC}
The limit pair $(u,z)$ obtained by convergences \eqref{convs} 
satisfies the brittle constraint \eqref{propsuz}.  
\end{lemma}
\begin{proof}
Let $t\in [0,T]$ be fixed. We may apply Lemma \ref{l:Mosco} to the sequence $(z_k(t))_k$ and to $z(t)$, which satisfy the semistability condition for the energies $\calE_k(t,u_k(t), \cdot)$ and $\calE_\infty (t,u(t), \cdot)$,  respectively, and thus inequality \eqref{SSZ}. Therefore, for every $t\in [0,T]$ the energies $\calJ_k(\cdot, z_k(t))$ Mosco-converge to $\calJ_\infty(\cdot, z(t))$ in $\Spu$.  Since $u_k(t) \weakto u(t)$ in $\Spu$ by \eqref{convsu}, the $\liminf$-inequality \eqref{mosco-liminf} ensures that 
$\calJ_\infty (u(t),z(t)) \leq \liminf_{k\to\infty}  \calJ_k(u_k(t), z_k(t)) \leq C$, 
where we have used that 
$ \calJ_k(u_k(t), z_k(t)) \leq \calE_k(t,u_k(t), z_k(t)) + \tilde{C} \leq C$ for  
constants $C$, $\tilde C$ uniform w.r.t.\ $t\in [0,T]$ and $k\in \N$ (cf.\ \eqref{unifbdE}). 
Hence $\calJ_\infty (u(t),z(t)) =0$ for every $t\in[0,T]$, whence the brittle 
constraint \eqref{propsuz} for $u$. 
\par 
In order to deduce that the brittle constraint is also satisfied by the time-derivative 
$\dot u$ given that $\big|\JUMP{u(t)}\big|^2\big|_{\supp z(t)}=0,$ 
we argue with the aid of difference quotients. 
In particular, it follows from the definition 
of the Bochner space $W^{1,p}(0,T;\mathbf{B}),$ with $\mathbf{B}$ a reflexive Banach space, 
cf.\ e.g.\ \cite[Def.\ A.1, p.\ 140]{Brez73}, that  for every $v\in W^{1,p}(0,T;\mathbf{B})$ there holds 
$\frac{v(t)-v(t-h)}{h}\to \dot v(t)$ strongly in $\mathbf{B}$ at every Lebesgue point $t$ of 
$\dot v$. Namely, for a.a.\ $t\in(0,T)$
\begin{equation}
\label{goodconv}
\lim_{0<h\to0}\Big\|\tfrac{v(t)-v(t-h)}{h}-\dot v(t)\Big\|_\mathbf{B}
= \lim_{0<h\to0}\Big\|\tfrac{1}{h}\int_{t-h}^t\dot v(s)\mathrm{d}s-\dot v(t)\Big\|_\mathbf{B}=0\,.
\end{equation}
For $v=\JUMP{u}\in H^1(0,T;L^2(\GC))$ such that $v(t,x)=0$ for a.a.\ $x\in\supp z(t),$  
for a.a.\ $t\in(0,T)$ we also have that 
$v(t-h,x)=0$ for a.a.\ $x\in\supp z(t),$ since 
 $\supp z(t) \subset \supp z(t-h)$.   The latter is due to the fact that $z(t)\leq z(t-h)$ a.e.\ in $\GC$ 
by the unidirectionality of $\calR_\infty$. Thus, in view of \eqref{goodconv}, denoting with 
$\calX_{\supp z(t)}$ the  characteristic  function of $\supp z(t),$ we obtain for a.a.\ $t\in(0,T)$ that 
\begin{eqnarray*}
\int_{\GC}\!\!\!\calX_{\supp z(t)}|\dot v(t)|^2\,\mathrm{d}\Surf
=
\int_{\GC}\!\!\!\calX_{\supp z(t)}
\big|\tfrac{v(t)-v(t-h)}{\bar h}-\dot v(t)\big|^2\,\mathrm{d}\Surf
\leq 
\int_{\GC}\!\!
\big|\tfrac{v(t)-v(t-h)}{\bar h}-\dot v(t)\big|^2\,\mathrm{d}\Surf
\to0\,. 
\end{eqnarray*}
Since the integrand on the left-hand side of the above inequality is positive, we conclude that, indeed, 
$\dot v(t)=\JUMP{\dot u(t)}=0$ a.e.\ on $\supp z(t)$ for a.a.\ $t\in(0,T)$.  
\end{proof}
%
\subsection{Recovery spaces for the momentum balance and their properties}
\label{RecSp}
%
Let us briefly resume the discussion, sketched in the Introduction, on 
the  difficulties attached to the limit passage as $k\to\infty$
in the momentum balance for the 
adhesive contact systems. We
recall its weak formulation, i.e. 
\begin{equation}
\label{momentum-k}
\langle\varrho \ddot{u}_k(t), v\rangle_\Spu
+ \int_{\Omega\backslash\GC} 
\left( \mathbb{D}e(\dot{u}_k(t)) \colon e(v) {+} \mathbb{C}e(u_k(t)):e(v) \right) \,\mathrm{d}x
+\int_{\GC} kz_k(t)\JUMP{u_k(t)}\JUMP{v}\,\mathrm{d}\Surf  = \langle\mathbf{f}(t),v\rangle_\Spu\,.
\end{equation} 
for every  $v\in \Spu$ and for a.a. $t\in (0,T)$. 
In the limit passage as $k\to\infty$, one has to face two problems: 
\begin{compactenum}
\item[(1)] the blow-up of the bounds on the adhesive contact term 
$ k z_k \JUMP{u_k}$  tested against $v\in\Spu$ as $k\to\infty$;
\item[(2)]
the consequent blow-up of the bounds (by comparison) on the inertial terms $(\ddot{u}_k)_k$.
\end{compactenum}
In Section \ref{MoscoBC} we  have illustrated the construction of the 
recovery sequence for the test functions of the brittle momentum balance. 
 In \cite{RosTho12ABDM}, such a  construction allowed us to overcome problem $(1)$ 
 in the \emph{quasistatic} (viscous) setting, 
where inertial terms in the momentum balance were neglected.  
\par
In what follows, we will exploit a \emph{refinement of this method} 
in order to tackle  problem $(2),$
by costructing a sequence of 
\emph{recovery spaces} for the space $\Spu_z(t)$ (cf.\  \eqref{spu-z}) 
of the  test functions for the momentum  balance  in the brittle limit. 
The definition of these recovery spaces and the proof of their properties relies on the support convergence $\supp z_k(t)\subset \supp z(t)+B_{\rho(k,t)}(0)$ for every $t\in[0,T]$, with $\rho(k,t)\to0$ as $k\to\infty$, 
of the semistable solutions $z_k$ to the adhesive systems
$\ingrsysk$ (cf.\ Prop.\ \ref{SuppConv})
 This  convergence is intended along the 
 very same sequence of indices $k$ 
such that convergences \eqref{convsz} hold. 
In particular, the extracted sequence  $(\rho(k,\cdot))_k$ of radii  
is independent of $t\in[0,T]$.  
Moreover, due to the temporally monotonically decreasing nature of (semistable) $z_k$ we also have 
\begin{equation}
\label{monoton-key-later}
\forall\,k\in\N\cup\{\infty\},\,\forall\,t>s\in[0,T]:\;\supp z_k(t)\subset\supp z_k(s)\,.
\end{equation}
But note that there is in general no monotonicity relation between $\rho(k,t)$ 
and $\rho(k,s),$ because 
$\supp z$ and $\supp z_k$ need not decrease with the same speed.    
\par
We now choose a \emph{nonincreasing} sequence $(\eps_n)_n$ with $\eps_n \downarrow  0$. 
Then, thanks to \eqref{monoton-key-later} for any $k\in\N\cup\{\infty\}$, we also have the 
following relation for every $s\in [0,T)$ and $t\in [s,T]$: 
\begin{equation}
\label{monotone-inclusion}
\text{If } \supp z_k(s)\subset z(s)+\overline{B_{\eps_n}(0)},
\text{ then also }\supp z_k(t)\subset z(s)+\overline{B_{\eps_n}(0)}\,. 
\end{equation}
This relation will be of great use later on, when deducing sufficient compactness results for 
the adhesive inertial terms on the intervals $[s,T]$. That is why,  
for the above chosen sequence 
$(\eps_n)_n$ with $\eps_n \downarrow  0,$ we now introduce 
the following \emph{recovery spaces} for all $n\in\N$  and 
for $t\in [0,T]$ and $s\in [0,T)$: 
\begin{equation}
\label{recsp}
\begin{split}
\Spu_z(\eps_n,t)&:=\{v\in H^1\big((\Omega\backslash\GC)
\cup(\supp z(t)\!+\!\overline{B_{\eps_n}(0)});\R^d\big)\, : \ v=0\text{ on }\GD\}
\\
&\, =  \{v\in \Spu\, : \ \JUMP{v}=0 \text{ a.e.\ on  } \GC\cap(\supp z(t)+\overline{B_{\eps_n}(0)})\},
\\
\spyn &:= L^2 (s,T;  \Spu_z(\eps_n,s) ), 
\\
 \spyns &:=
\Big\{(v_1,v_2,v_3)\in(\spyn)^*\times(\spyn)^*\times(\spyn)^*\, 
: \,\\
&\hspace*{0.7cm}
\text{ for }i=1,2:\,
\textstyle{\int_{ s+h}^T}\big\langle v_i(t)-v_i(t-h)-\textstyle{\int_{t-h}^t}   
 v_{i+1}(s)\,
\mathrm{d}s,\phi(t)
\big\rangle_\Spu\,\mathrm{d}t=0
\\
&\hspace*{0.7cm}
\qquad \qquad \qquad \text{ for all }
\phi\in\spyn,\,  h \in (0,T-s)  \},
\Big\}\,.                                                         
\end{split}
\end{equation}
Observe that the definition of the spaces $\spyns$ encompasses the information that 
$v_{i+1}$ is the time-derivative  of the function $v_i$ in $(\spyn)^*,$ for $i=1,2$. 
Indeed, choosing test functions $\phi=\eta\varphi$ with 
$\eta\in \mathrm{C}_0^\infty(s,T)$  such that  $\supp \eta \subset (s+h,T)$,
and 
 $\varphi\in\Spv_z(\eps_n,s),$ the 
fundamental lemma of the Calculus of Variations yields that 
$\big\langle v_i(t)-v_i(t-h)-\textstyle{\int_{t-h}^t}  
 v_{i+1}(s)\,
\mathrm{d}s, 
 \varphi  
\big\rangle_\Spu=0$ for a.a.\ $t\in(s+h,T)$ and for  $h\in(0,T-s)$. Hence, 
$v_i(t)=v_i(t-h)-\textstyle{\int_{t-h}^t}  
 v_{i+1}(s)\,
\mathrm{d}s$ in $\Spu_z(\eps_n,s)^*,$ which corresponds to the notion of the time-derivative 
in Bochner-spaces, 
cf.\ \cite[p.\ 140, Def.\ A.1]{Brez73}.  
\par 
In fact, the needed compactness of the adhesive inertial terms $(\ddot u_k)_k$ will be deduced in 
the spaces $\spyns,$ first for all $s\in[0,T)$ and for $n\in\N$ fixed. Next, we will 
prove the existence of a limit for the inertial terms $(\ddot{u}_k)_k$  in the spaces 
$L^2(s,T; \Spu_z(s)^*)$  for all $s\in[0,T)$. 
Our argument for this will be based on the properties of the recovery spaces 
summarized in Prop.\ \ref{Propsrecsp}. 
Most crucial is the  density result stated in Item 3..  
Ultimately, it will allow us to show that the limit inertial term $\ddot u(t)$  is 
 an element of \ $\Spu_z(t)^*$ for almost all  $t\in (0,T)$.  
This density result can be concluded from the  support convergence \eqref{suppconv} 
from Prop.\ \ref{SuppConv}, while 
the other properties follow from  Prop.\ \ref{PropsSpuM}.          
\begin{proposition}
\label{Propsrecsp}
For all $s,t\in[0,T],$ let the spaces   
$\Spu_z(t),\Spu_z(\eps_n,t), 
\spyn$, and $\spyns$ be 
as in  \eqref{spu-z} and \eqref{recsp}.   Then, 
\begin{compactenum}
\item  
for every $n\in\N$ and every $t\in[0,T]$ 
the space $\Spu_z(\eps_n,t)$ is a closed subspace of $\Spu_z(t)$ endowed with the 
norm $\|\cdot\|_\Spu$. For every $n\in\N$ and all $s\leq t$ 
the space $\Spu_z(\eps_n,s)$ is a closed subspace of $\Spu_z(\eps_n,t)$. 
Moreover,  since the sequence $(\eps_n)_n$ is monotonically decreasing, there holds 
\begin{equation}
\label{monot-vepsn}
\VZ(\eps_n,t) \subset \VZ(\eps_{n+1},t)  \quad \text{ for every $t\in [0,T]$ and $n\in \N$; }
\end{equation}
\item 
 for every $s\in (0,T],$ the space $\spyn$ 
is a closed subspace of 
 $L^2(s,T,\VZ(s))$. Hence, 
$\spyn$  endowed with the norm $\big(\int_0^T\|\cdot\|_\Spu^2\,\mathrm{d}t\big)^{1/2}$ is  a reflexive Banach space, and so is 
 $(\spyn)^* \cong L^2 (s,T; \VZ(\eps_n, s)^*)$; 
\item   
for every $s\in[0,T],$ the union $\cup_{n\in\N}\spyn$ is dense in $L^2(s,T;\VZ(s))$; 
\item  for every $s\in[0,T]$ and every  $n\in\N,$ the space $\spyns$  endowed with the norm 
$\|\cdot\|_{(\mathbf{Y}_n^{s.})^*}\times\|\cdot\|_{(\mathbf{Y}_n^{s.}))^*}\times\|\cdot\|_{(\mathbf{Y}_n^{s.})^*}$ 
is a  reflexive Banach space. 
\end{compactenum}
\end{proposition}
\begin{proof}
{\bf Ad 1.: } 
For each $s<t\in[0,T]$ it holds  $\supp z(t)\subset\supp z(s)\subset\GC$  and, thus, 
$\Spu(\eps_n,s)\subset\Spu(\eps_n,t)\subset\Spu$. Similarly 
$\supp z(t)\subset \supp z(t)+ \overline{B_{\eps_n}(0)} $   and, hence, $\Spu(\eps_n,t)\subset\Spu_z(t)$. 
It is also a standard matter to verify that $\VZ(\eps_n,t) \subset \VZ(\eps_{n+1},t) $, 
since $B_{\eps_{n+1}}(0) \subset B_{\eps_n}(0)$ for
$(\eps_n)_n$ decreasing. 
The closedness then follows by Prop.\ \ref{PropsSpuM} for each of the spaces. 
\par 
{\bf Ad 2.: } It can be straightforwardly verified that $\spyn$ is a subspace 
 of $L^2(s,T;\Spu_z(s))$; its  closedness follows from the very same argument as in the proof of Prop.\ \ref{PropsSpuM}. 
The representation formula for $(\spyn)^*$ is a standard fact in the 
theory of Bochner spaces,  cf., e.g., \cite{DieUhl77}. 
\par
{\bf Ad 3.: } In order to verify that  $\cup_{n\in\N}\spyn$ is dense in $L^2(s,T;\Spu_z(s))$,  we fix 
a function 
$v\in L^2(s,T;\Spu_z(s))$ and prove the existence of a sequence $(v_n)_n\subset \cup_{n\in\N}\spyn$ satisfying 
$v_n\to v$ in   $L^2(s,T;\Spu)$.   
For this, we first pick from the equivalence class $v$ a selection 
$\bar v$  
that is defined for every  $t\in[s,T]$. 
 For instance, we may choose   
\begin{equation}
\label{repv}
\bar v(t):=\left\{
\begin{array}{ll}
 \lim_{h\to0} \frac{1}{h}\int_{t-h}^t  v(r)  \,\mathrm{d}r&\text{if }t\text{ is a Lebesgue point  of $v$,}\\
0&\text{otw.,} 
\end{array}
\right. 
\end{equation} 
recalling the definition of  Lebesgue points:
\begin{equation*}
 t\in(s,T] \text{ is a Lebesgue point of $ v$, if }\;
\lim_{h\to0}\big\| v(t)-\tfrac{1}{h}\int_{t-h}^t  v(r)\,\mathrm{d}r\big\|_\Spu=0\,.
\end{equation*}   
Using this representative $\bar v$ and the recovery operator $r$ from \eqref{testconst},  for every $n\in\N$ we set  
\begin{equation}
\label{recfct}
v_n(t): = r(\eps_n, \supp z(s),\bar v(t)) \qquad \text{for all }  t\in[0,T]. 
\end{equation}
 By construction of the recovery operator, cf.\ \eqref{testconst}, 
we have for $\bar{v}(t) \in \Spu_z(s)$ that $v_n(t) \in \Spu(\eps_n, s)$ for all $t\in [s,T]$. Moreover, 
 $\|v_n(t)-\bar v(t)\|_{H^1(\Omega\backslash\GC,\R^d)}\to0$ as $n\to\infty$ 
and, in addition, the following estimate holds true: 
\begin{eqnarray*}
\|v_n(t)-\bar v(t)\|_{H^1(\Omega\backslash\GC,\R^d)}
\leq \|\bar v_\mathrm{anti}(t)(\xi_{\eps_n}^{\supp z(s)}-1)\|_{H^1(\Omega\backslash\GC,\R^d)}
\leq \|\bar v_\mathrm{anti}(t)\|_{H^1(\Omega\backslash\GC,\R^d)}\,. 
\end{eqnarray*}
The dominated convergence theorem thus allows us to conclude that $v_n \to \bar v $ 
in $L^2(s,T; H^1(\Omega\backslash\GC,\R^d))$, 
which concludes the proof. 
\par 
 {\bf Ad 4.: } 
We now show that $\spyns$ endowed with the norm 
$\|\cdot\|_{(\spyn)^*}
\times\|\cdot\|_{(\spyn)^*}\times\|\cdot\|_{(\spyn)^*}$ 
is a reflexive Banach space. For this, we argue that $\spyns$ 
is a closed subspace of the reflexive Banach space 
$(\spyn)^*\times(\spyn)^*\times(\spyn)^*$: Consider a sequence 
$(v^k_1,v^k_2,v^k_3)_k\subset\spyns$ 
such that $(v^k_1,v^k_2,v^k_3)\to(v_1,v_2,v_3)$ in 
$(\spyn)^*\times(\spyn)^*\times(\spyn)^*$,  which means 
for each $i\in \{1,2,3\}$   that 
\begin{equation}
\label{conv-doubt}
\sup_{\phi\in\spyn,\|\phi\|_\Spu=1} \left|
\int_s^T\langle v_i^k(t)-v_i(t),\phi(t)\rangle_\Spu \dd t  \right| 
\to0 \qquad \text{ as $k\to\infty$.}
\end{equation}
 This allows  us to pass  to the limit in the terms
$\int_{s+h}^T\langle v_i^k(t), \phi(t) \rangle_\Spu \dd t $ and
$\int_{s+h}^T\langle v_i^k(t-h), \phi(t) \rangle_\Spu \dd t $ for $i=1,2$ and all $\phi \in \spyn$. 
Moreover, for the integral term involving $v_{i+1}$ we observe that  for all 
$\phi \in \spyn,$ for almost all $t\in (s,T),$ we have that 
\begin{equation}
\label{doubt}
\int_{t-h}^t 
\langle v_{i+1}^k(\tau)-v_{i+1}(\tau), \phi(t)  \rangle_\Spu\,\mathrm{d}\tau\to0\,.
\end{equation} 
Indeed, for $\phi \in \spyn,$ we have by definition $\phi(t)\in \Spu(\eps_n,s)$ for a.a.\ $t\in(s,T)$.  
For all $\tau\in[s,T]$ we may thus set $\varphi(\tau):=\phi(t)$ and understand $\varphi\in\spyns$ 
as a function 
constant in time. 
Using $\varphi$ as a particular choice in \eqref{conv-doubt}, we conclude \eqref{doubt}. 
Since also 
\begin{equation*}
\begin{split}
\Big|\int_{t-h}^t
 \langle v_{i+1}^k(\tau)-v_{i+1}(\tau),\phi(t)\rangle_\Spu\,\mathrm{d}\tau \Big|
\leq \sup_{ \phi  \in \spyn} \int_s^T
\langle v_{i+1}^k(\tau)-v_{i+1}(\tau),\phi(t)\rangle_\Spu\,\mathrm{d}\tau 
\leq C\,,
\end{split}
\end{equation*}
the dominated convergence theorem implies that 
\begin{equation*}
 \int_{s+h}^T 
\int_{t-h}^t \langle v_{i+1}^k(\tau)-v_{i+1}(\tau),\phi(t)\rangle_\Spu\,\mathrm{d}\tau 
\,\mathrm{d}t\to0\,. 
\end{equation*} 
From this we ultimately conclude that $(v_1,v_2,v_3)\in\spyns$. 
\end{proof}
%
\subsection{Compactness for the inertial terms \& limit in the momentum balance}
\label{LimitMomBEnDiss}
%
With the first result of this section,
 Lemma \ref{Limmombal},  we pass from adhesive to brittle in the momentum balance.  
In fact, this limit passage  
will go hand in hand with establishing sufficient compactness 
for the inertial terms. These arguments rely on the recovery spaces $\spyn$ and $\spyns$ 
introduced in \eqref{recsp}, 
which are just small enough to prevent the blow-up of the functional derivatives of the 
adhesive contact term, 
but still large enough to carry the information on the support of the limit 
delamination variable. That is why, compactness and limit passage cannot be separated. 
More precisely,   
we shall  prove one by one  Items 1.-3.\ of Lemma \ref{Limmombal} below.     
\begin{lemma}[Compactness for the inertial terms \& limit passage in the momentum balance]
\label{Limmombal}
The following statements hold true: 
\begin{compactenum}
\item[{\bf 1.}] {\bf Compactness \& brittle momentum balance in $\VZ(s)^*$ 
for every $s\in[0,T)$ fixed: }\\
For every $s\in [0,T)$ there exists a (not relabeled) subsequence of $(u_k)_k$, 
possibly depending on $s$, and a function 
$\mu^s \in L^2(s,T; \VZ(s)^*)$, such that 
\begin{subequations}
\begin{align}
&
\label{enh-conv-triple}
(u_k,\dot u_k,\ddot u_k)\rightharpoonup (u,\dot u,\mu^s)\text{ in }
\spyns  \text{ as } k \to \infty \quad \text{for every } n \in\N, \text{ and }
\\
& 
\label{weak-distributional}
\begin{aligned}
&
\langle
\int_{t-h}^t \mu_s(\tau)\,\mathrm{d}\tau,  \mathsf{v} \rangle_\Spu=  \langle
\dot u(t)-\dot u(t-h), \mathsf{v} \rangle_\Spu=0
\quad\text{for all } \mathsf{v} \in \Spu_z(s),
\\
& \qquad \qquad 
\qquad  \text{ for a.a.\ }t\in(s+h,T), \text{ and  for all } h \in (0, T-s),
 \end{aligned}
\end{align}
whence
  \begin{equation}
  \label{added-derivative}
   \pairing{}{\Spu_z(s)}{\mu_s(t)}{\mathsf{v}} 
= \lim_{h\downarrow 0}  \pairing{}{\Spu_z(s)}{\frac{\dot{u}(t)-\dot{u}(t-h)}{h}}{\mathsf{v}}   
\qquad \text{for all } \mathsf{v} \in \Spu_z(s) \quad \foraa\, t \in (s,T).
   \end{equation}
Furthermore, the momentum balance holds with test functions in $\VZ(s)$, i.e., for a.a.\ $t \in (s,T)$:
\begin{equation}
\label{fundamental-lemma}
\langle\varrho \mu_s(t),\mathsf{v}\rangle_\Spu+\int_{\Omega\backslash\GC}
\big(\mathbb{C}e(u(t))+\mathbb{D}e(\dot u(t))\big):e(\mathsf{v})\,\mathrm{d}x = \langle \mathbf{f}(t), \mathsf{v} \rangle_\Spu
\quad\text{ for every } \mathsf{v}\in \VZ(s)\,.
\end{equation}
\end{subequations}
\item[{\bf 2.}] {\bf Compactness independent of $s\in[0,T)$: }\\
Let   $D\subset(0,T]$ be a dense and countable subset. There exist a (not relabeled) subsequence of 
$(u_k, \dot{u}_k, \ddot{u}_k)_k$ and  a function
\begin{subequations}
\label{diag2}
\begin{eqnarray}
\label{mu-zero}
&&\mu \in L^2(0,T; \VZ(0)^*) \cap \bigcap_{s\in D} L^2(s,T; \VZ(s)^*)  
\quad\text{ such that} \\
\label{via-diagonal}
&&(u_k, \dot{u}_k, \ddot{u}_k) \weakto (u,\dot u, \mu) \; \text{ in } \spyns \text{ for all } s 
\in D \cup\{0\} \text{ and every } n \in \N\,,\text{ and s.t.\ }\qquad\qquad\\
&&(u,\dot u,\mu)
\text{ comply with \eqref{weak-distributional}--\eqref{fundamental-lemma} for all }s\in D \cup\{0\}\,.
\end{eqnarray}
\end{subequations}
\item[{\bf 3.}] {\bf Brittle momentum balance in $\VZ(t)^*$ for a.a.\ $t\in(0,T)$: }\\
The function $\mu$ satisfies for almost all $t\in (0,T)$ 
\begin{subequations}
\label{finalbrittle}
\begin{eqnarray}
\label{target-subtle}
&&\mu(t) \in \VZ(t)^* 
\quad\text{and}\\
\label{entitles}
&&    \tfrac{\dot{u}(t+h)-\dot{u}(t)}{h} \weakto {\mu}(t) \text{ in } \VZ(t)^* \quad \text{as } h \downarrow 0,
  \end{eqnarray} 
hence $\mu(t) = \ddot{u}(t)$ in the sense of  \eqref{weak-ddot-u}. Moreover, 
the momentum balance \eqref{weak-momentum-brittle} holds, 
and $\text{for a.a.\ }t\in(h,T) \text{ for every $h\in (0,T)$}$ it is                         
  \begin{equation}
  \label{we-use-it-later}
  \langle\dot u(t)-\dot u(t-h), v \rangle_\Spu 
= 
\langle \int_{t-h}^t\ddot{u}(\tau)\,\mathrm{d}\tau,  \mathsf{v} \rangle_\Spu
\quad\text{for all } v  \in \Spu_z(t)\,.
\end{equation} 
\end{subequations}
 Moreover, $\lambda\in  L^2(0,T;\Spu^*)$ extracted in \eqref{convsddu}, cf.\ also 
\eqref{brittle-alternative}, satisfies relation \eqref{lambda-id}, i.e., 
\begin{equation}
\label{lambda-id2}
\lambda(t)=\ddot u(t)\quad\text{ in }\VZ(t)^*\;\text{ for a.a.\ }t\in(0,T)\,. 
\end{equation}
\end{compactenum}
\end{lemma}
\noindent
An inspection of the proof of Lemma \ref{Limmombal} will reveal that, 
in fact, \eqref{weak-distributional}\,\&\,\eqref{we-use-it-later} also hold 
with the forward differences $\dot{u}(t+h) - \dot{u}(t)$. 
\par
Before carrying out the  details, we briefly summarize the 
{\bf main ideas of the proof}: 
\begin{compactitem}
\item[{\bf Ad 1.}]{\bf Compactness \& limit balance in $\VZ(s)^*$ for every $s\in[0,T)$ fixed: }
Using the previously introduced recovery spaces $\spyn$ and $\spyns,$ 
for arbitrary fixed $s\in[0,T)$ (which is the starting point of the time-intervals $(s,T)$ 
taken into account in  $\spyn$ and $\spyns$), and $n\in\N,$ we extract a convergent ($s$-dependent) 
subsequence $(u_k,\dot u_k,\ddot u_k)_k\subset \spyns$ 
and a limit triple $(u,\dot u,\mu_n^s)\in\spyns$. Thanks to the definition \eqref{recsp} 
of $\spyns$ we are entitled to say that 
$\mu_n^s=\ddot u$ in $(\spyn)^*$. This allows us to pass to the limit in the momentum balance 
integrated over $(s,T)$ and to obtain a limit balance in $(\spyn)^*$. 
By a diagonal sequence argument over $n\in\N$ we can even extract a 
subsequence and a limit converging for all $n\in\N$ to find 
\eqref{enh-conv-triple} \& \eqref{weak-distributional}. 
Due to the density result Prop.\ \ref{Propsrecsp}, Item 3., we can then pass $n\to\infty$ to find \eqref{added-derivative} 
and the limit momentum to 
hold for a.a.\ $t\in(s,T)$ with test functions $v\in\VZ(s),$ i.e., \eqref{fundamental-lemma}. 
\item[{\bf Ad 2.}]{\bf Compactness independent of $s\in[0,T)$: } 
The subsequences 
and their limit extracted by a diagonal procedure over $n\in\N$ in Item 1, depend on $s\in[0,T)$. 
By arguing on the countable dense set $D\subset[0,T]$ 
we can essentially repeat the demonstration of Item 1 in a further diagonal 
procedure over the elements of 
$D\cup\{0\}$ to conclude statements \eqref{diag2}.    
\item[{\bf Ad 3.}]{\bf Brittle momentum balance in $\VZ(t)^*$ for a.a.\ $t\in(0,T)$: }In order to 
show that $\mu(t)\in\VZ(t)^*$ and to extend the brittle 
momentum balance to hold in $\VZ(t)^*$ we adapt the arguments of \cite[Lemma 2.2]{DMLar11EWED}. 
The basis for this is a further density result \cite[Lemma 2.3]{DMLar11EWED}, 
which in our setting guarantees  that 
\begin{equation}
\label{from-D-Lars}
\begin{split}
&\text{\emph{for the monotonically increasing sequence of closed linear subspaces 
$(\VZ(t))_{t\in[0,T]}$ of 
the}}\\
&\text{\emph{separable Hilbert space $\Spu$ there exists an at most countable set 
$S\subset[0,T]$ such that:}}\\
&\hspace*{4cm}
\VZ(t)=\overline{\cup_{s<t}\VZ(s)}\quad\text{for all }t\in[0,T]\backslash S\,. 
\end{split}
\end{equation}
In this way, we can approximate a test function $\phi\in\VZ(t)$ for any $t\in(0,T)$ out of a set of zero Lebesgue measure 
by a sequence $(\phi_m)_m\in\cup_m\VZ(s_m)$ with $s_m\nearrow t$ and $(s_m)_m\subset D$. 
Hence, \eqref{diag2} holds along $(s_m)_m$ and by approximation we may ultimately infer statements \eqref{finalbrittle}.  
 Finally, relation \eqref{lambda-id2} ensues by direct comparison of the brittle momentum balance 
\eqref{weak-mom-brittle} with \eqref{brittle-alternative}. 
\end{compactitem}
%
\subsubsection*{Proof of Lemma \ref{Limmombal}:}
%
{\bf Ad 1. Compactness \& brittle momentum balance in $\VZ(s)^*$ for every $s\in[0,T]$ fixed: }
Observe that, by the  very definition \eqref{recsp} of $\spyn$, 
  for every $v\in \spyn$ and almost all $t\in (s,T)$ there holds $\JUMP{v(t)} =0$ on $\supp(z(s)) + \overline{B_{\eps_n}(0)}$, and thus 
on $\supp(z(t)) + \overline{B_{\eps_n}(0)}$ 
since $\supp(z(t))  \subset  \supp(z(s))$ (cf.\ \eqref{monotonicity-Vz-later-used}).  Moreover, in dependence of $s\in[0,T]$ and 
$n\in\N$ we find, thanks to support convergence \eqref{suppconv} an index $k(s,n)$, such that for all $k\geq k(s,n)$ it is 
$\supp z_k(s) \subset \supp(z(s)) + B_{\eps_n}(0)$ and thus, by \eqref{monotone-inclusion}, also 
$\supp z_k(t) \subset \supp(z(s)) + B_{\eps_n}(0)$.   
All in all, we conclude that 
\begin{equation}
\label{null-duality-pairing}
\begin{split}
&\forall\,s\in[0,T),\forall\,n\in\N\;\exists\,k(s,n)\;\forall\,k\geq k(s,n):\;\\
&\hspace*{4cm}
\langle\partial_u\calJ_k(z_k(t),u_k(t)),v(t)\rangle_{\Spu}=0 \quad  
\text{for all }v\in \spyn,\,\foraa\, t \in (s,T).  
\end{split}
\end{equation}
Therefore, by comparison 
in the $k$-momentum balance, we can deduce the following uniform bounds 
for the inertial terms,  which are independent of $k\in\N$: 
\begin{equation} 
 \exists\,C>0\;\forall\,s\in[0,T) \forall\,n\in\N\;\exists\,k(s,n)\;\forall\,k\geq k(s,n):\;
\quad 
\|\ddot u_k\|_{(\spyn)^*}  
\leq C\,.
\end{equation} 
In addition, from the uniform bound \eqref{unifbdu}, 
i.e.\ $\|u_k\|_{H^1(0,T;\Spu)}\leq C,$ we also deduce that  
$\|u_k\|_{(\spyn)^*}\leq C$ as well as $\|\dot u_k\|_{(\spyn)^*}\leq C$, 
since $L^2(s,T;\Spu) \subset (\spyn)^*$ continuously.  
 In particular, we observe for all 
$n\in\N$ and all $v\in\spyn$ that
\begin{eqnarray*}
\int_h^T\Big|\Big\langle\tfrac{u_k(t)-u_k(t-h)}{h}-\dot u_k(t),v(t)\Big\rangle_\Spu\Big|\,\mathrm{d}t
&\leq&\int_h^T\Big\|\tfrac{u_k(t)-u_k(t-h)}{h}-\dot u_k(t)\Big\|_{\Spu^*}\|v(t)\|_\Spu\,\mathrm{d}t\\ 
&\leq & C\int_h^T\Big\|\tfrac{u_k(t)-u_k(t-h)}{h}-\dot u_k(t)\Big\|_{\Spu}\|v(t)\|_\Spu\,\mathrm{d}t\\
&\leq&\|v\|_{L^2(0,T;\Spu)}\left( \int_{h}^T\Big\|\tfrac{u_k(t)-u_k(t-h)}{h}-\dot u_k(t)\Big\|_\Spu^2
\,\mathrm{d}t \right)^{1/2} \to0
\end{eqnarray*}
as $h\to0$ since $u_k\in H^1(0,T;\Spu)$. Hence, $\dot u_k$ indeed is the partial time 
derivative of $u_k$ also in 
the space $(\spyn)^*$. In the same way, 
 taking into account that $u_k\in H^2(0,T;\Spu^*)$ and that 
$L^2(s,T;\Spu^*) \subset (\spyn)^*$ continuously, 
we argue that $\ddot u_k$ 
is the partial time derivative of 
$\dot u_k$ in $(\spyn)^*$. 
Due to these observations we deduce the following uniform bounds
\begin{equation}
\label{multibd}
 \exists\,C>0\;\forall\,s\in[0,T] \ \forall\,n\in\N\;\exists\,k(s,n)\;\forall\,k\geq k(s,n):\; 
 \quad \|u_k\|_{H^1(0,T;\Spu)}+
 \|u_k\|_{\spyns}\leq C\,, 
\end{equation}
where $\spyns$ is defined in \eqref{recsp}. 
\par
Again, keep $s\in[0,T]$ fixed. 
By \eqref{multibd} and by the reflexivity of the spaces granted by Prop.\ \ref{Propsrecsp}, 
for every $n\in\N$ we can extract a convergent subsequence. 
But to be more precise, here we extract this subsequence by a diagonal procedure: 
Starting with $n=1,$ from the corresponding bound \eqref{multibd}, we   
find a (not relabeled, $s$-dependent) subsequence such that  
\begin{equation}
\label{convddu}
u_k\rightharpoonup u^1\text{ in }H^1(0,T;\Spu),\;
(u_k,\dot u_k,\ddot u_k)\rightharpoonup (u^1,\dot u^1,\mu^1_s)\text{ in }
\widetilde{\mathbf{Y}}_1^{s}\,. 
\end{equation}
Observe that,  in view of convergence \eqref{convsu}, 
taking into account the continuous embedding $L^2(s,T;\Spu) \subset (\spyn)^*$, 
  we can identify the 
limits $u^1$  and $\dot u^1$,  
i.e.\ 
we have $u^1=u|_{(s,T)}$ and $\dot{u}^1=\dot{u}|_{(s,T)}$. 
Now, for $n=2$ the above subsequence 
satisfies 
the corresponding bound \eqref{multibd}, 
so that we can extract a further (not relabeled, $s$-dependent) subsequence satisfying 
\[
(u_k,\dot u_k,\ddot u_k)\rightharpoonup (u,\dot u,\mu^2_s)\text{ in }
\widetilde{\mathbf{Y}}_2^{s}\,. 
\]
Due to the monotonicity property $\VZ(\eps_1, s) \subset \VZ(\eps_2, s)$ 
 for $\eps_2<\eps_1$, 
it holds $L^2(s,T; \VZ(\eps_1, s) ) \subset L^2(s,T; \VZ(\eps_2, s))$. Hence 
we find that  the restriction of the element  $\mu^2_s \in L^2(s,T; \VZ(\eps_2, s))^*$ to 
$\mathbf{Y}_{1}^s= L^2(s,T; \VZ(\eps_1, s))$ coincides with $\mu^1_s$. 
Proceeding this way, we obtain a (not relabeled, $s$-dependent) sequence $(u_k)_k$ and 
a sequence of limits $(u,\dot u,\mu^n_s)_n$ such that, for every $n\in\N$: 
\begin{equation}
\label{convddun}
\begin{gathered}
u_k\rightharpoonup u\text{ in }H^1(0,T;\Spu),\;
(u_k,\dot u_k,\ddot u_k)\rightharpoonup (u,\dot u,\mu^n_s)\text{ in }
\widetilde{\mathbf{Y}}_n^{s}  \text{ as } k \to \infty  \text{ and }
\\
\mu^n_s|_{\mathbf{Y}_{n-1}^s} = \mu^{n-1}_s  \quad \text{for every } n \in \N\,.
  \end{gathered}
\end{equation}
For each $n\in\N,$ due to the weak convergence of the sequence in 
$\spyns,$ we 
also have that 
\begin{equation}
\label{propmun}
\int_{s+h}^T\langle \dot u(t)-\dot u(t-h)
-\int_{t-h}^t 
  \mu^n_s(\tau)\,\mathrm{d}\tau, \phi(t) \rangle_\Spu\,\mathrm{d}t=0 
\text{ for all }\phi\in 
\spyn
\text{ for all }h\in(0,T-s)\,. 
\end{equation}
Let us now 
extend the functions $(\mu_s^n)_n$ to an element $\mu_s \in L^2(s,T;\Spu_z(s)^*) $ by setting
\begin{equation}
\label{muext}
\langle \mu_s   ,\phi\rangle_{L^2(s,T;\Spu_z(s))}:=\left\{
\begin{array}{ll}
\langle  \mu_s^n,  \phi\rangle_{\spyn}&\text{if }\phi\in \spyn \text{ for some } n \in \N,\\
0&\text{if }\phi\in L^2(s,T;\Spu_z(s))\backslash \cup_{n\in \N} \spyn.
\end{array}
\right. 
\end{equation}
Observe that $\mu_s$ is well-defined. Indeed, suppose that $\phi \in \mathbf{Y}_{n_1}^s \cap  \mathbf{Y}_{n_2}^s $ 
for some $n_1<n_2 \in \N$. By  the monotonicity property \eqref{monot-vepsn}, 
there holds $\mathbf{Y}_{n_1}^s \subset  \mathbf{Y}_{n_2}^s$ and, thanks to \eqref{convddun}, 
$\langle  \mu_s^{n_2},  \phi\rangle_{\mathbf{Y}_{n_2}^s} = \langle  \mu_s^{n_1},  \phi\rangle_{\mathbf{Y}_{n_1}^s} $. 
Observe that 
$\|\mu_s\|_{L^2(s,T;\Spu_z(s)^*)} \leq \sup_{n\in \N}\|\mu^n_s\|_{(\spyn)^*}\leq C$, 
and \eqref{enh-conv-triple} follows from   \eqref{convddun}.  
\par
 Using the density of $\cup_{n} \spyn$ in $L^2(s,T;\Spu_z(s))$, we now show \eqref{weak-distributional}. 
For this, let $\phi\in L^2(s,T;\Spu_z(s))$. 
Using the construction of the recovery sequence 
\eqref{recfct} we find $\phi_n(\tau)= r(\eps_n,  \supp(z(s)),\phi)\in\mathbf{Y}_n^{s}$ 
with the property $\phi_n\to\phi$ strongly in $L^2(s,T;\Spu_z(s))$. 
Using weak-strong convergence arguments and the dominated convergence theorem we deduce that 
\begin{eqnarray*}
\int_{s+h}^T\langle\int_{t-h}^t  \mu_s^n(\tau)\mathrm{d}\tau,\phi_n(t)\rangle_\Spu\mathrm{d}t
\to \int_{s+h}^T\langle\int_{t-h}^t\mu_s(\tau)\mathrm{d}\tau,\phi(t)\rangle_\Spu\mathrm{d}t\,,
\end{eqnarray*}
essentially arguing along the lines of the proof of  Item 4 in  Prop.\ \ref{Propsrecsp}.   
Hence, from \eqref{propmun}, we deduce that 
\begin{equation}
\label{propmu-limit}
\int_{s+h}^T\langle \dot u(t)-\dot u(t-h)
-\int_{t-h}^t {\mu}_s(\tau)\,\mathrm{d}\tau, \phi(t) \rangle_\Spu\,\mathrm{d}t=0
\text{ for all }\phi\in  L^2(s,T;\Spu_z(s))\text{ and for all }h\in(0,T-s)\,. 
\end{equation} 
Choosing $\phi\in L^2(s,T;\Spu_z(s))$ such that $\phi(t,x)=\eta(t)\mathsf{v}(x)$ with 
$\eta\in C_0^\infty(s,T)$ and $\mathsf{v}  \in\Spu_z(s)$ the fundamental lemma of the 
Calculus of Variations yields  \eqref{weak-distributional}.  
From this, taking the limit as $h\to 0$,  we get  \eqref{added-derivative},     cf.\ \eqref{weak-ddot-u}.
   \par
In order to prove \eqref{fundamental-lemma}, 
  we shall pass to the limit as $k\to\infty$
in the $k$-momentum balance, with test functions 
$ v\in\mathbf{Y}_n^{s}$, for  $n \in \N$ \emph{fixed}. To this aim, 
 we test \eqref{momentum-k}   by $v\in \spyn$ and integrate over $(s,T)$. 
 Convergences  \eqref{enh-conv-triple} then allow us to pass to the limit 
$k\to\infty$  for $n\in \N$ fixed:  
\begin{eqnarray}
&&\int_s^T \big(\langle\varrho\ddot u_k,  v  \rangle_\Spu+\int_{\Omega\backslash\GC}
\big(\mathbb{C}e(u_k)+\mathbb{D}e(\dot u_k)\big):e(v) \,\mathrm{d}x \big)\,\mathrm{d}t=\int_s^T\langle  \mathbf{f}, v  \rangle_\Spu  \dd t
\nonumber
 \\
&&\hspace*{2cm}\downarrow
\nonumber
\\
&&
\label{added-13Apr16}
\int_s^T \big(\langle\varrho  \mu_s ,v\rangle_\Spu+\int_{\Omega\backslash\GC}
\big(\mathbb{C}e(u)+\mathbb{D}e(\dot u)\big):e(v)\,\mathrm{d}x\big)\,\mathrm{d}t
= \int_s^T\langle  \mathbf{f}, v  \rangle_\Spu  \dd t \quad \text{for all } 
v\in\spyn\,.\qquad
\end{eqnarray} 
\par
We now obtain the (integrated) brittle  momentum balance, first with test functions in 
$L^2(s,T;\Spu_z(s))$. Indeed, let 
  $v\in L^2(s,T;\VZ(s))$ be fixed, and let $(v_n)_n \subset \spyn$ be the corresponding 
recovery sequence  given in \eqref{recfct}. 
Taking into account that  $v_n\to v$ in $L^2(0,T;H^1(\Omega\backslash\GC,\R^d))$ as $n\to\infty$, 
we pass to the limit with $n$ in \eqref{added-13Apr16}  and finally obtain 
\[
\int_s^T\big(\langle\varrho \mu_s,v\rangle_\Spu+\int_{\Omega\backslash\GC}
\big(\mathbb{C}e(u)+\mathbb{D}e(\dot u)\big):e(v)\,\mathrm{d}x\big)\,\mathrm{d}t=\int_s^T \langle  \mathbf{f},v\rangle_\Spu \dd t 
\quad\text{ for every }v\in L^2(s,T;\VZ(s))\,.
\]
Again, choosing test functions $v$ of the form $v(t,x) = \eta(t) \mathsf{v}(x)$ with 
$\eta\in C_0^\infty(s,T)$ and $ \mathsf{v}  \in\Spu_z(s)$
we obtain \eqref{fundamental-lemma}.  This concludes the proof of Lemma \ref{Limmombal}, Item 1.
\par
\noindent
{\bf Ad 2.\ Compactness independent of $s\in[0,T]$: }
For the countable, dense set $D\subset[0,T],$ let us order the elements of $D$ 
in an increasing sequence $(s_n)_n$ with $s_n<s_{n+1}$ for all $n\in\N$. 
First of all, we apply the previously proven Item 1 of Lemma \ref{Limmombal} 
for $s=0$ and find a not relabeled 
subsequence $(u_k,\dot{u}_k, \ddot{u}_k)_k$,  and 
$\mu:= \mu_0 \in L^2(0,T; \VZ(0)^*)$ such that \eqref{enh-conv-triple}--\eqref{fundamental-lemma}  
hold for $\mu$ and $s=0$. We now apply it for $s=s_1$, and find a further subsequence, 
and $\mu_{s_1}\in  L^2(s_1,T; \VZ(s_1)^*)$, fulfilling \eqref{enh-conv-triple}--\eqref{fundamental-lemma}. 
Observe that, since $s_1>0$, there holds $\VZ(0) \subset \VZ(s_1)$. Therefore, from \eqref{fundamental-lemma}  
at $s=0$ and at $s=s_1$ we read that 
\[
\begin{aligned}
 \pairing{}{\VZ(s_1)}{ \varrho  \mu_{s_1}(t)}{\mathsf{v}}   = \langle \mathbf{f}(t), \mathsf{v} \rangle_\Spu  
& - \int_{\Omega\backslash\GC}
\big(\mathbb{C}e(u(t))+\mathbb{D}e(\dot u(t))\big):e(\mathsf{v})\,\mathrm{d}x
  =  \pairing{}{\VZ(0)}{\mu_0(t)}{\mathsf{v}}
 \\
 & 
\quad\text{ for every } \mathsf{v}\in \VZ(0) \quad \foraa\, t \in (s_1,T)\,,
\end{aligned}
\]
whence $\mu_{s_1} (t)= \mu_0(t)=\mu(t)$ in $\VZ(0)^*$ for almost all $t\in (s_1,T)$. 
With a diagonal procedure we conclude the statement of Lemma \ref{Limmombal}, Item 2. 
\par
\noindent
{\bf Ad 3.\ Brittle momentum balance in $\VZ(t)$ for a.a.\ $t\in(0,T)$: }
From Item 1 of Lemma \ref{Limmombal} it follows that   
for every $s\in D \cup \{0\}$ there exists a set $N_s$ with zero Lebesgue measure, such that for all $t\in [s,T]\setminus N_s$  there holds
\begin{equation}
\label{subtle-argum-1}
\mu(t) \in \VZ(s)^*,\,\tfrac{u(t+h)-u(t)}{h} \weakto \mu(t) \text{ in } 
\VZ(s)^* \text{ as $h\downarrow 0$}, \  \  \mu(t) \text{ fulfills \eqref{fundamental-lemma}  
in $\VZ(s)^*$.}
\end{equation}
In order to show 
\eqref{target-subtle}, we shall now adapt 
an argument from the proof of \cite[Lemma 2.2]{DMLar11EWED}. Indeed, set 
 $N:= \cup_{s\in D\cup\{0\}}N_s$, with $N_s$ the negligible set out of which \eqref{subtle-argum-1} holds. 
 Then, $N$ is also negligible and for every $t\in [0,T]\setminus N$  properties \eqref{subtle-argum-1}
hold at every $s\in D\cup\{0\}$ with $s<t$. 
Now, 
to the monotonically increasing family of  closed sets $(\VZ(t))_{t\in [0,T]}$
we apply \eqref{from-D-Lars}. 
Hence, let us fix $t\in [0,T]\setminus (S {\cup} N)$, and let us pick an increasing sequence $(s_m)_m \subset D$ with $s_m \uparrow t$.  Due to
\eqref{from-D-Lars}, 
for every $\phi \in \VZ(t)$ there exists 
 a sequence $(\phi_m)_m$, with $\phi_m \in \VZ(s_m)$ for every $m\in \N$, such that $\phi_m \to \phi$ in $\VZ(t)$. 
 Observe that, in particular,  $\mu(t)$ fulfills \eqref{fundamental-lemma}  
with the test functions $\phi_m$ for all $m\in \N$. Therefore, 
 \begin{equation}
 \label{subtle-3}
 \begin{aligned}
 \exists\, \lim_{m\to\infty}\pairing{}{\VZ(s_m)}{\varrho \mu(t)}{\phi_m}  & =  
\lim_{m\to\infty} \Big(  \langle \mathbf{f}(t), \phi_m \rangle_\Spu   - \int_{\Omega\backslash\GC}
\big(\mathbb{C}e(u(t))+\mathbb{D}e(\dot u(t))\big):e(\phi_m)\,\mathrm{d}x\Big)
\\
& 
 = \langle \mathbf{f}(t), \phi \rangle_\Spu   - \int_{\Omega\backslash\GC}
\big(\mathbb{C}e(u(t))+\mathbb{D}e(\dot u(t))\big):e(\phi)\,\mathrm{d}x\,.
 \end{aligned}
 \end{equation}
 Since $\phi \in \VZ(t)$ is arbitrary, the right-hand side of \eqref{subtle-3} defines an element in $\VZ(t)^*$, which, for the time being,  we denote by  $ \tilde{\mu}(t)$. Observe that, in fact
 \begin{equation}\label{temporary-def}
\pairing{}{\VZ(t)}{\varrho \tilde{\mu}(t)}{\phi} =  \lim_{m\to\infty}\pairing{}{\VZ(s_m)}{ \varrho \ddot{u} (t)}{\phi_m}   
\quad \text{for \emph{every}  $(\phi_m)_m \subset \cup_{s<t} \VZ(s)$ with $\phi_m\to\phi$ in $\VZ(t)$,} 
 \end{equation}
 whence
 \[ \pairing{}{\VZ(t)}{\varrho\tilde{\mu}(t) }{\phi}   =  
 \pairing{}{\VZ(s)}{ \varrho \mu(t)}{\phi}  \qquad \text{for all } \phi\in \VZ(s) \text{ and all } s<t,
 \]
 choosing the constant sequence $\phi_m \equiv \phi$ in \eqref{temporary-def}. 
  Repeating the very same argument as in the proof of 
  \cite[Lemma 2.2]{DMLar11EWED}, we may in fact check that, for $t\in [0,T]\setminus (S {\cup} N)$ fixed, 
  \eqref{entitles} holds, 
 which ultimately entitles us to 
denote  $\mu(t)$  by $\ddot{u}(t)$, cf.\ \eqref{weak-ddot-u}. 
   Clearly, \eqref{subtle-3} yields  that $\ddot{u}(t)$ satisfies the brittle momentum balance, with test functions in
$\VZ(t)$, for all $t\in [0,T]\setminus (S {\cup} N)$. This gives
 \eqref{weak-momentum-brittle}.  Furthermore, again using \eqref{temporary-def} we may extend 
 \eqref{weak-distributional}
 to test functions $\mathsf{v}$  in $\Spu_z(t)$, namely we conclude \eqref{we-use-it-later}    for every $h\in (0,T)$. 
   \par
A comparison argument in \eqref{weak-momentum-brittle}, taking into account that $u\in H^1(0,T;\Spu)$ and that 
$\mathbf{f} \in L^2(0,T;\Spu^*)$, 
shows that  the map 
$t \mapsto 
\sup_{v \in \VZ(t)} |\pairing{}{\VZ(t)}{\ddot{u} (t)}{v}| \doteq \| \mu(t)\|_{\VZ(t)^*}$ is in $L^2(0,T)$, whence  
$
\ddot{u}\in L^2(0,T;\VZ^*). 
$
Thus, $u\in H^2_{\#}(0,T;\VZ^*),$ cf.\ \eqref{Sobolev-VZ}. 
\par 
 To find relation \eqref{lambda-id2} at time $t\in(0,T)$ out of a negligible set, 
we test both \eqref{weak-mom-brittle} and \eqref{brittle-alternative} with arbitrary $v\in\VZ(t)$. 
Subtracting the two equations from each other yields $\langle \varrho  \ddot u(t)-\lambda(t),v\rangle_{\VZ(t)}=0$.  
This concludes the proof of Lemma \ref{Limmombal}, Item 3.
\hfill$\square$
\vspace*{1mm} 
\par
By exploiting the validity of the brittle momentum balance \eqref{weak-mom-brittle} and relation \eqref{entitles} 
we now deduce the  {\bf 3.\ additional regularity \eqref{better-H2}}  of the limit $\ddot u$.      
\begin{lemma}[Regularity \eqref{better-H2} of the limit $\ddot u$] 
\label{Lbetter-H2}
There holds 
\begin{equation}
\label{diff-quotients-strong}
\tfrac{\dot{u}(\cdot+h) - \dot{u}(\cdot)}{h} \to \ddot{u} \text{ strongly in } L^2(0,T;\VZ(0)^*),
\end{equation}
{\em therefore $u \in H^2(0,T; \VZ(0)^*)$.}
\end{lemma}
\begin{proof}
In order to show \eqref{diff-quotients-strong}, we will check that 
for every sequence $(h_n)_n$ with $h_n \to 0$
\begin{equation}
\label{Cauchy-sequence}
\left(\tfrac{\dot{u}(\cdot+h_n) - \dot{u}(\cdot)}{h_n} \right)_{n\in\N} \text{ is a Cauchy sequence in } L^2(0,T;\VZ(0)^*). 
\end{equation}
To this aim, we observe that for almost all $t\in (0,T)$ and all $v\in \VZ(0)$
\[
\begin{aligned}
\pairing{}{\VZ(0)}{\tfrac{\dot{u}(\cdot+h_n) - \dot{u}(\cdot)}{h_n}}{v}   
& \stackrel{(1)}{=} \tfrac1{h_n} \pairing{}{\VZ(0)}{\int_t^{t+h_n}\ddot{u}(\tau) \dd \tau}{v} 
\\
& 
\begin{aligned}
 \stackrel{(2)}{=}  \tfrac1{\varrho h_n} \Big(  & 
\langle \int_t^{t+h_n} \mathbf{f}(\tau), v \rangle_\Spu    
-   \int_t^{t+h_n}\int_{\Omega\backslash\GC}
\mathbb{C}e(u(\tau)) :e(v)\,\mathrm{d}x \dd \tau \\
 &  - \int_{\Omega\backslash\GC}
  \mathbb{D}e(u(t+h_n))  :e(v)\,\mathrm{d}x +  \int_{\Omega\backslash\GC}
  \mathbb{D}e(u(t))  :e(v)\,\mathrm{d}x \Big),
  \end{aligned}
  \end{aligned}
\]
where (1) follows from \eqref{we-use-it-later} with test functions $v$ in $\VZ(0)\subset \VZ(t)$, 
and (2) ensues from the momentum balance \eqref{weak-momentum-brittle}. 
Taking into account that $\mathbf{f} \in \mathrm{C}^1([0,T];\Spu^*) $ and that $u \in H^1(0,T;\Spu)$, 
from the above identity we conclude \eqref{Cauchy-sequence}. 
Therefore, there exists $w\in L^2(0,T;\VZ(0)^*)$ such that 
$\frac{\dot{u}(\cdot+h_n) - \dot{u}(\cdot)}{h_n} \to w$ in $L^2(0,T;\VZ(0)^*)$, whence, 
up to a subsequence, 
$\frac{\dot{u}(t+h_n) - \dot{u}(t)}{h_n} \to w(t)$ in $\VZ(0)^*$ for almost all $t\in (0,T)$. 
Taking into account 
\eqref{entitles}, we ultimately conclude that $w(t) = \ddot {u}(t)$ in $\VZ(0)^*$, 
and \eqref{diff-quotients-strong} ensues. 
\end{proof}
Since regularity \eqref{better-H2}, i.e., $u\in H^2(0,T;\VZ(0)^*),$ holds true 
and since the spaces $\VZ(0)\subset\Spw\subset\VZ(0)^*$ form a Gelfand triple, 
we in fact have that 
$\dot u \in L^\infty (0,T;\Spw) \cap \mathrm{C}^0([0,T];\VZ(0)^*)$.  
Therefore, we are now in the position to deduce  
the  {\bf 4.\ weak temporal continuity of $\dot u$} as a corollary of \eqref{better-H2}.    
\begin{corollary}[Weak continuity of $\dot u$, Theorem \ref{thm:main}, Item 4]
\label{WeakConti}
We have 
\begin{eqnarray}
\label{propu1}
&&\dot u(t)\in\Spw
  \text{ and } \|\dot{u}(t)\|_\Spw\leq C \quad \text{for every } t \in [0,T]\,, 
\\
\label{propu2}
&&t\mapsto\dot u(t)\text{ is weakly continuous from }[0,T]\text{ to }\Spw.
\end{eqnarray}
\end{corollary}
\begin{proof}
Given $t\in[0,T],$ 
using that $\dot u \in L^\infty (0,T;\Spw)$ we find that 
there exists a sequence $(t_n)_n\subset[0,T]$ 
with $t_n\to t,$ such that $\|\dot u(t_n)\|_\Spw\leq C$. Since $\dot u(t_n)\to\dot u(t)$ 
 in $\Spu_z(0)^*$ we conclude from the continuous embedding 
$\Spw\hookrightarrow\Spu_z(0)^*$ that also $\dot u(t)\in\Spw$ with 
$\|\dot u(t)\|_\Spw\leq C$ and $\dot u(t_n)\rightharpoonup \dot u(t)$ in $\Spw$. 
This proves \eqref{propu1}. The same argument with arbitrary $t,(t_n)_n\subset[0,T]$ 
such that $t_n\to t$ yields \eqref{propu2}.
\end{proof}
Observe that, by \eqref{initial-data-conv} and \eqref{convsu}, the limit displacement $u$ 
satisfies the initial condition $u(0)=u_0$ in $\Spu$.  
But it remains to verify that $\dot u(0)=u_1$ in $\Spw$. For this, we will prove the pointwise-in-time 
weak $\Spw$-convergence of $\dot u_k(t)$ to $\dot u(t),$ cf.\ \eqref{ptw-weak-dotu-convergence} below. 
In fact, in view of convergences \eqref{convs}, this convergence result is also the missing 
piece allowing us in Sec.\ \ref{LEnbal} to pass to 
the limit in the energy balance as an inequality via lower semicontinuity arguments.    
\begin{lemma}[Pointwise-in-time weak $L^2$-convergence \& initial condition $\dot u(0)=u_1$]
\label{Linit}
Along the same sequence 
as in Lemma \ref{Limmombal}, Item 1, it holds
\begin{equation}
\label{ptw-weak-dotu-convergence}
\dot{u}_k(t) \weakto \dot{u}(t) \quad \text{in } \Spw \quad \text{for every } t \in [0,T],
\end{equation}
{therefore $\dot{u}(0) = u_1$ thanks to \eqref{initial-data-conv}.}
\end{lemma}
  \begin{proof}
  It follows from convergence \eqref{via-diagonal}, for $s=0$, and from the previously obtained estimates,  that 
  the sequence $(\dot{u}_k)_k$ is bounded in $L^2(0,T;\Spu) \cap L^\infty (0,T;\Spw) \cap H^1 (0,T;\VZ(\eps_n,0)^*)$ for every $n\in \N$. 
  Since for each $n$ the space $\VZ(\eps_n,0)$ is densely and compactly embedded in $\Spw$, 
we have that $\Spw\subset \VZ(\eps_n,0)^*$  densely, and compactly. 
By a Aubin-Lions compactness argument (cf.\ e.g.\ \cite[Cor.\ 5, p.\ 86]{simon86}),
  we conclude 
  \begin{equation}
  \label{Aubin-Lions}
  \dot{u}_k \to \dot{u} \quad \text{in } L^p(0,T;\Spw) \cap \mathrm{C}^0 ([0,T];\VZ(\eps_n,0)^*)
  \end{equation}
  for some fixed $n\in \N$. Ultimately, we infer convergence \eqref{ptw-weak-dotu-convergence}: 
indeed, for every  $t\in [0,T]$, every subsequence of the sequence $(\dot{u}_k(t))_k$, 
bounded in $\Spw$, admits a  further subsequence weakly converging in $\Spw$ to some limit $v_t$.  
In view of \eqref{Aubin-Lions}, we  have that $v_t = \dot{u}(t)$: since the limit does not depend 
on the extracted subsequence,  convergence \eqref{ptw-weak-dotu-convergence} holds. 
\end{proof}
%
\subsection{Limit passage in the energy balance \& Proof of Thm.\ \ref{thm:main}, Items 5.-8.}
\label{LEnbal}
%
We  deduce the energy balance \eqref{eneqlim-identity} for the brittle limit system.  
First, in Lemma \ref{LEDEbrittle}, we will obtain the inequality $\leq$, \emph{at all $t\in [0,T]$}, by suitable 
lower semicontinuity arguments, cf.\  \eqref{eneqlim-inequality-leq} below. 
\begin{lemma}[Upper energy-dissipation estimate via lower semicontinuity]
\label{LEDEbrittle}    
We have                  
\begin{equation}
\label{eneqlim-inequality-leq} 
 \begin{aligned}
\tfrac{1}{2}\|\dot u(t)\|_{\Spw}^2   & 
+ \int_0^t  2\calV(\dot u(s)) 
\,\mathrm{d} s + \Var_{\calR_\infty}(z, [0,t])+
 \calE_\infty(t,u(t),z(t)) \
 \\
& \leq \tfrac{1}{2} \|\dot u(0)\|_{\Spw}^2
+  \calE_\infty(0,u(0),z(0)) + \int_0^t \partial_t\calE_\infty(s,u(s),z(s))\,\mathrm{d}s
\quad \text{ for all } t \in [0,T]\,.
 \end{aligned}
\end{equation} 
\end{lemma}
\begin{proof}
Inequality \eqref{eneqlim-inequality-leq}  follows by passing to the limit as $k\to\infty$
in the energy-dissipation inequality for the adhesive system. On the left-hand side, we  exploit convergences \eqref{convs}, 
which give $\int_0^t  \calV(\dot u(s)) \dd s \leq \liminf_{k\to\infty} \int_0^t  \calV(\dot{u}_k(s)) \dd s$ and 
$\Var_{\calR_\infty}(z, [0,t]) 
\leq \liminf_{k\to\infty} \Var_{\calR_k}(z_k, [0,t])$. Furthermore, 
 the pointwise convergences for $u$ and $z$ in 
\eqref{convs} give $ \calE_\infty(t,u(t),z(t))  \leq \liminf_{k\to\infty}  \calE_k(t,u_k(t),z_k(t)) $ via \eqref{mosco-liminf}, and the limit passage 
in the term $\frac{\varrho}{2}\|\dot{u}_k(t)\|_{\Spw}^2$ is guaranteed by \eqref{ptw-weak-dotu-convergence}. On the right-hand side, we use
the convergence for the initial data \eqref{initial-data-conv} and again \eqref{convs}, 
which allows us to pass to the limit in $\int_0^t \partial_t\calE_k(s,u_k(s),z_k(s))\,\mathrm{d}s$. 
\end{proof}
\par                                                          
The energy-dissipation inequality opposite to \eqref{eneqlim-inequality-leq} will be proved in 
Lemma \ref{UEDEbrittle} ahead. For this we will have to test the brittle 
momentum balance \eqref{weak-mom-brittle} by $\dot u$. Note that this is admissible since also $\dot u$ satisfies the 
brittle constraint \eqref{propsuz}. Also observe that the quadratic bulk term and the external loading term comply 
with a chain rule. The missing piece is thus a chain-rule inequality  involving the 
kinetic term $\varrho \ddot{u}$ and the Gelfand triple $(\Spu_z(t), \Spw, \Spu_z(t)^*)$, established now in 
Lemma \ref{ChainruleLemma} (cf.\ \eqref{gelfand-1}  and \eqref{gelfand-2}  ahead). 
For the proof of Lemma \ref{ChainruleLemma}, Item 1, we adapt the arguments from 
\cite[Lemma 3.5]{DMLar11EWED}.
Lemma \ref{ChainruleLemma}, Item 2, can then be concluded  following the lines of 
\cite[Lemma 3.6]{DMLar11EWED}, exploiting the weak continuity of $\dot u$ proved in 
Lemma \ref{WeakConti}. 
\begin{lemma}[Chain rule for the inertial term]
\label{ChainruleLemma}
Let $u\in H^2(0,T;\VZ(0)^*)$ comply with the regularity 
properties \eqref{reg-u-britt;e} \& \eqref{propu2}, 
and with the brittle momentum balance for given 
$z\in \mathrm{BV}(0,T;L^1(\GC))\cap \mathrm{B}([0,T];\SBV(\GC;\{0,1\})),$ 
semistable as in \eqref{semistab-z} for all $t\in[0,T]$. Then,  the following statements hold true: 
\begin{compactenum}
\item
for all $ s,t \in (0,T]$ such that $s$ and $t$ are Lebesgue points for $\| \dot{u}(\cdot)\|_{\Spw}^2$ there holds:  
\begin{equation}
\label{gelfand-1}
\tfrac12 \| \dot{u}(t) \|_{\Spw}^2 - \tfrac12 \| \dot{u}(s) \|_{\Spw}^2  
=\int_s^t \pairing{}{\Spu_z(\tau)}{\varrho\ddot{u}(\tau)}{\dot{u}(\tau)} \dd \tau \,,
\end{equation}
\item 
$\dot u$ fulfills  the integral chain-rule inequality 
\begin{equation}
\label{gelfand-2}
\tfrac12 \| \dot{u}(t) \|_{\Spw}^2 - \tfrac12 \| u_1 \|_{\Spw}^2  
\geq\int_0^t \pairing{}{\Spu_z(\tau)}{{\varrho}\ddot{u}(\tau)}{\dot{u}(\tau)} \dd \tau 
\end{equation} 
holds true for every   Lebesgue point $t\in (0,T]$ of $\| \dot{u}(\cdot)\|_{\Spw}^2$. 
\end{compactenum}
\end{lemma}
\begin{proof}
{\bf Ad 1.:} 
In order to prove \eqref{gelfand-1} we adapt the argument from the 
proof of \cite[Lemma 3.5]{DMLar11EWED}. 
A straightforward calculation shows that 
\begin{equation}
\label{straightforward}
\begin{aligned}
\tfrac1h \left(\|\dot{u}(t) \|_{\Spw}^2 - \|\dot{u}(t-h) \|_{\Spw}^2 \right)   =
\pairing{}{\Spu_z(t)}{{\varrho}\tfrac{\dot{u}(t) - \dot{u}(t-h)}{h}}{\dot{u}(t) + \dot{u}(t-h)}
\end{aligned}
\end{equation}
for all $h\in (0,T)$ and for almost all $t\in (h,T)$. 
Integrating \eqref{straightforward}
on the interval $(s,t)$ yields for every $h\in (0,s)$
\begin{equation}
\label{here-we-take-the-lim}
\tfrac1{2h} \int_{t-h}^t \|\dot{u}(\tau) \|_{\Spw}^2 \dd \tau  -\tfrac1{2h} \int_{s-h}^s \|\dot{u}(\tau) \|_{\Spw}^2 \dd \tau 
=\tfrac12 \int_s^t  
\pairing{}{\Spu_z(\tau)}{{\varrho}\tfrac{\dot{u}(\tau) - \dot{u}(\tau-h)}{h}}{\dot{u}(\tau) + \dot{u}(\tau-h)} \dd \tau\,.
\end{equation}
Now, since $\dot u \in L^2(0,T;\Spu)$, there holds
\begin{equation}
\label{hwtl-1}
\dot{u}(\tau) + \dot{u}(\tau-h) \to 2 \dot{u}(\tau) 
\quad \text{in } \Spu\quad\text{for a.a.\ }\tau\in(0,T)
\end{equation}
along a sequence $h\downarrow 0$. Combining this with the weak convergence
$\frac{\dot{u}(\tau) - \dot{u}(\tau-h)}{h}\weakto \ddot{u}(\tau)$ in $\Spu_z(\tau)^*$ for almost all $\tau\in (0,T)$,
we infer the pointwise convergence
\begin{equation}
\label{ptwconv}
 a_h(\tau):=
\pairing{}{\Spu_z(\tau)}{{\varrho}\tfrac{\dot{u}(\tau) - \dot{u}(\tau-h)}{h}}{\dot{u}(\tau) 
+ \dot{u}(\tau-h)}\to  2  \pairing{}{\VZ(\tau)}{\varrho\ddot{u}(\tau)}{\dot{u}(\tau)}   \qquad \foraa\, \tau\in (0,T). 
\end{equation}
  In order to pass to the limit in the integral on the right-hand side of \eqref{here-we-take-the-lim}, 
we shall apply a variant of the dominated convergence theorem, cf.\ e.g.\ 
\cite[Thm.\ 5.3, p.\ 261]{Els02}. For this, we further observe that 
\begin{equation*}
\begin{split}
|a_h(\tau)|&\leq \|\varrho \tfrac{\dot{u}(\tau) - \dot{u}(\tau-h)}{h}  \|_{\Spu_z(\tau)^*}
\Big(\| \dot{u}(\tau) \|_{\Spu_z(\tau)}+\| \dot{u}(\tau-h) \|_{\Spu_z(\tau)}\Big)
\\
&=\|  \varrho   \tfrac{\dot{u}(\tau) - \dot{u}(\tau-h)}{h}  \|_{\Spu_z(\tau)^*}
\Big(\| \dot{u}(\tau) \|_{\Spu}+\| \dot{u}(\tau-h) \|_{\Spu}\Big)
=:M_h(\tau)\,.
\end{split}
\end{equation*}
We now introduce the short-hands 
$l_h(\tau):=\|\varrho \tfrac{\dot{u}(\tau) - \dot{u}(\tau-h)}{h}  \|_{\Spu_z(\tau)^*}$ and 
$m_h(\tau):=\big(\| \dot{u}(\tau) \|_{\Spu}+\| \dot{u}(\tau-h) \|_{\Spu}\big),$ so that 
$M_h(\tau)=l_h(\tau)m_h(\tau)$. 
Since $\dot u\in L^2(0,T;\Spu)$ we have 
\begin{equation}
\label{convmh}
m_h\to2\|\dot u\|_\Spu\quad\text{in }L^2(0,T)\,, 
\end{equation}
while, thanks to \eqref{we-use-it-later}, we find for $l_h$ that 
\[
  \begin{aligned}
  l_h(\tau)= 
\| \tfrac1h \int_{\tau-h}^{\tau} g(s) \dd s \|_{\Spu_z(\tau)^*}
&= \| \tfrac1h \int_{\tau-h}^{\tau} g(s) \dd s \|_{\Spu^*}
  \text{ with }    g : (0,T) \to \Spu^*
 \text{ given by }  
 \\
 & 
 \pairing{}{\Spu}{g(s)}{v} = 
\langle \mathbf{f}(s), v \rangle_\Spu - \int_{\Omega{\setminus}\GC}\left( \mathbb{D} e(\dot{u}(s))\colon e(v) + \mathbb{C} e(u(s))\colon e(v) \right) \dd x\,.
\end{aligned}
\]
Now, since $g\in L^2(0,T;\Spu^*)$,  
the sequence of functions $\tau \mapsto \tfrac1h \int_{\tau-h}^{\tau} g(s) \dd s$ 
converges to $g$ in $L^2(0,T; \Spu^*)$. Therefore, 
$(l_h)_h$  converges to $\|g(\cdot) \|_{\Spu^*}$ in $L^2(0,T)$. 
Together with \eqref{convmh} we infer that $(M_h)_h\subset L^1(0,T)$ and 
$M_h\to2\|g(\cdot)\|_{\Spu^*}\|\dot u(\cdot)\|_\Spu$ in $L^1(0,T)$. Now, 
the dominated convergence theorem, c.f.\ e.g.\ 
\cite[Thm.\ 5.3, p.\ 261]{Els02}, allows us to conclude that $\int_s^t a_h(\tau)\,\mathrm{d}\tau
\to\int_s^t2\langle \varrho  \ddot u(\tau),\dot u(\tau)\rangle_{\Spu_z(\tau)}\,\mathrm{d}\tau,$ and thus 
 the convergence of  the right-hand side of \eqref{here-we-take-the-lim}. 
Additionally, the integrals on the left-hand side of \eqref{here-we-take-the-lim}  
converge to the left-hand side of \eqref{gelfand-1} since $s$ and $t$ are 
Lebesgue points for $\| \dot{u}(\cdot)\|_{\Spw}^2$. Thus, we conclude \eqref{gelfand-1}.                                        
\par 
{\bf Ad 2.: }
The proof can be adapted from 
\cite[Lemma 3.6]{DMLar11EWED}. More precisely, we choose a sequence of Lebesgue 
points $(t_j)_{j\in\N}$ of 
$\|\dot u(\cdot)\|^2_\Spw$
such that $t_j\searrow0$. 
By Lemma \ref{WeakConti}, $\dot u:[0,T]\to\Spw$ is weakly  continuous, i.e., we have 
$\dot u(t_j)\rightharpoonup \dot{u}(0)$ in $\Spw$.  
Moreover, for all Lebesgue points $t_j$ and $t$ the chain rule \eqref{gelfand-1} holds true. 
Hence, we find: 
\begin{eqnarray*}
&&\int_0^t\langle\varrho\ddot u(\tau),\dot u(\tau)\rangle_{\VZ(\tau)}\,\mathrm{d}\tau
=\liminf_{j\to\infty} 
\int_{t_j}^t\langle\varrho\ddot u(\tau),\dot u(\tau)\rangle_{\VZ(\tau)}\,\mathrm{d}\tau
\leq \limsup_{j\to\infty}\tfrac{1}{2}
\big(\|\dot u(t)\|_\Spw^2-\|\dot u(t_j)\|_\Spw^2\big)\\
&\leq& \tfrac{1}{2}\|\dot u(t)\|_\Spw^2
-\liminf_{j\to\infty}\tfrac{1}{2}\|\dot u(t_j)\|_\Spw^2
\leq \tfrac{1}{2}\|\dot u(t)\|_\Spw^2-\tfrac{1}{2}\|\dot u(0)\|_\Spw^2
=\tfrac{1}{2}\|\dot u(t)\|_\Spw^2-\tfrac{1}{2}\| u_1 \|_\Spw^2\,.
\end{eqnarray*}
Above, the last inequality follows from the lower semicontinuity of $\|\cdot\|^2_\Spw$ 
with respect to weak 
convergence in $L^2(\Omega)$ and the last equality is due to the initial condition 
$\dot u(0)=u_1$ verified in Lemma  \ref{Linit}.  
\end{proof}
Thanks to the previously established chain rule \eqref{gelfand-2} we are now in the position to prove 
the energy-dissipation inequality opposite to \eqref{eneqlim-inequality-leq}, 
i.e.\ \eqref{eneqlim} below. 
As already mentioned, for this we will test the brittle 
momentum balance \eqref{weak-mom-brittle} by $\dot u$, apply  chain rules 
separately to each of the energy terms, 
and combine the obtained relation with the brittle semistability condition,  
arguing as in the proof of the energy-dissipation balance \eqref{eneq} for the adhesive system. 
\begin{lemma}[Lower energy-dissipation estimate, 
balance \eqref{eneqlim-identity}]
\label{UEDEbrittle}
The limit pair $(u,z)$ extracted by convergences 
\eqref{convs} satisfies the following lower energy-dissipation estimate 
\begin{equation}
\label{eneqlim} 
 \begin{aligned}
\tfrac{1}{2}\|\dot u(t)\|_{\Spw}^2  & 
+ \int_s^t  2\calV(\dot u(\tau)) 
\,\mathrm{d} \tau + \Var_{\calR_\infty}(z, [s,t])+
 \calE_\infty(t,u(t),z(t))\\ 
&   \geq \tfrac{1}{2} \|\dot u(s)\|_{\Spw}^2
+  \calE_\infty(s,u(s),z(s)) + \int_s^t \partial_t\calE_\infty(\tau,u(\tau),z(\tau))\,\mathrm{d}\tau\,
\\
& 
 \qquad \text{for all } s, t \in [0,T]\backslash \mathsf{L}  \text{ with } s<t  
\text{ and for } s=0, 
 \end{aligned}
\end{equation} 
 where  $\mathsf{L}$ denotes the set  of Lebesgue points of  $\|\dot u(\cdot)\|_{\Spw}^2$. 
Hence, the energy-dissipation balance \eqref{eneqlim-identity} holds true 
 as well as the bulk energy-dissipation balance \eqref{eneqlim-identity-bulk} 
and the surface energy-dissipation balance \eqref{eneqlim-identity-surf}.  
\end{lemma}
\begin{proof}
  We test the momentum balance \eqref{weak-momentum-brittle} of the brittle limit system by $\dot u,$ 
which is admissible according 
to \eqref{propsuz}. We argue  as in the proof of the $k$-energy balance \eqref{eneq}, i.e.\
using integration by parts on the loading term, and   exploiting  the
analogues of \eqref{visc-damping} and \eqref{chain-ruleC} (since $u\in H^1(0,T;\Spu)$), 
for the  viscous  and the  bulk energy  terms, respectively.
For the   
 inertial term,  we use \eqref{gelfand-1}   and \eqref{gelfand-2}. 
Thus 
 we find for almost all $s,t\in (0,T)$ with $s<t$  and for $s=0$ 
\begin{equation}
\label{momuint}
\begin{split}
&\tfrac{1}{2}\|\dot{u}(t)\|_\Spw^2-\tfrac{1}{2}\|\dot{u}(s)\|_\Spw^2
+\langle-\mathbf{f}(t),u(t)\rangle_\Spu-\langle-\mathbf{f}(s),u(s)\rangle_\Spu
-\int_s^t\langle-\dot{\mathbf{f}}(\tau), u(\tau)  \rangle_\Spu\,\mathrm{d}\tau\\
&\qquad+\int_{\Omega\backslash\GC}\tfrac{1}{2}\big(
\mathbb{C}e(u(t)):e(u(t))-\mathbb{C}e(u(s)):e(u(s))\big)\,\mathrm{d}x
+\int_s^t\calV(e(\dot u(\tau)))\,\mathrm{d}\tau \geq 0
\,.   
\end{split}
\end{equation} 
\par
We further note that the semistability inequality for the brittle limit at time  $t_0=s$, 
tested with $\tilde z=z(t)$ for arbitrary $t\in[0,T]$ 
reduces to 
\begin{equation}
\label{semiz}
\calR_\infty(z(t)-z(s)) +  \mathrm{b}_\infty P(Z(t),\GC)
-a_\infty^0\int_{\GC}z(t)\,\mathrm{d}\Surf  - \mathrm{b}_\infty P(Z(s),\GC) 
+ a_\infty^0\int_{\GC}z(s)\,\mathrm{d}\Surf \geq 0\,.
\end{equation} 
Summing up \eqref{momuint} and \eqref{semiz} results 
in   \eqref{eneqlim} valid in Lebesgue points $s,t$ of 
$\|u\|_\Spw^2$. 
\par 
Then, finally, the energy-dissipation balance \eqref{eneqlim-identity} 
follows from combining \eqref{eneqlim-inequality-leq} with \eqref{eneqlim}. 
\par 
Observe that   \eqref{eneqlim-identity}  rewrites as  $A+B=0$, where $A$ and $B$ stand for 
the left-hand sides of inequalities \eqref{momuint} \& \eqref{semiz},
which in turn state $A\geq 0$ and $B\geq 0$. Therefore, as a by-product we 
have that $A=0=B$.  
In particular, from $A=0$ it is immediate to conclude that 
the chain-rule inequality  \eqref{gelfand-2}  also
holds as an \emph{equality} at  the Lebesgue points of $\|\dot u(\cdot)\|_{\Spw}^2$.  
\end{proof}
%
Thanks to the previously proved energy-dissipation balance \eqref{eneqlim-identity} for the brittle system, 
also exploiting the assumed convergence of the initial data \eqref{initial-data-conv} and 
the immediate convergences \eqref{convs}, we can now deduce the enhanced convergences \eqref{enh-convs}. 
\begin{lemma}[Enhanced convergences \eqref{enh-convs}]
\label{EnhancedConvs}
The enhanced convergences \eqref{enh-convs} hold true. Therefore, $(u,z)$ 
comply with the upper energy-dissipation estimate \eqref{eneqlim-inequality-leq} 
on the interval $(s,t)$, for every $t\in(0,T]$ and almost all $s\in (0,t)$. 
\end{lemma}
\begin{proof} 
 The previously proved convergences  as well as  \eqref{initial-data-conv} yield 
\begin{equation*}
\begin{aligned}
&
\tfrac{1}{2}\|\dot u(t)\|_{\Spw}^2 
+ \int_0^t  2\calV(\dot u(s))   
\,\mathrm{d} s + \Var_{\calR_\infty}(z, [0,t])+
 \calE_\infty(t,u(t),z(t))\\
 & 
\leq \liminf_{k\to\infty}\tfrac{1}{2}\|\dot{u}_k(t)\|_{\Spw}^2 
+  \liminf_{k\to\infty} \int_0^t  2\calV(\dot{u}_k(s))   
\,\mathrm{d} s + \liminf_{k\to\infty} \Var_{\calR_k}(z_k, [0,t])+
\liminf_{k\to\infty} \calE_k(t,u_k(t),z_k(t)) 
 \\
 & 
 \leq \limsup_{k\to\infty} \left( \tfrac{1}{2}\|\dot{u}_k(t)\|_{\Spw}^2  
+ \int_0^t  2\calV(\dot{u}_k(s))   
\,\mathrm{d} s +  \Var_{\calR_k}(z_k, [0,t])+
 \calE_k(t,u_k(t),z_k(t)) 
 \right)
\end{aligned}
\end{equation*}
\begin{equation*}
\begin{aligned}
 &  = \lim_{k\to\infty}  \tfrac{1}{2} \|u_1^k\|_{\Spw}^2  + \lim_{k\to\infty}  \calE_k (0,u_0,z_0)
 + \lim_{k\to\infty} 
 \int_0^t \partial_t\calE_k(s,u_k(s),z_k(s))\,\mathrm{d}s
 \\
&  =\tfrac{1}{2} \|u_1\|_{\Spw}^2
+  \calE_\infty(0,u(0),z(0)) + \int_0^t \partial_t\calE_\infty(s,u(s),z(s))\,\mathrm{d}s
\\
& = \tfrac{1}{2}\|\dot u(t)\|_{\Spw}^2   
+ \int_0^t  2\calV(\dot u(s))   
\,\mathrm{d} s + \Var_{\calR_\infty}(z, [0,t])+
 \calE_\infty(t,u(t),z(t))
\end{aligned}
\end{equation*}
where the last equality follows from  the energy equality \eqref{eneqlim-identity}, 
at almost all $t\in (0,T)$. 
Hence, 
all inequalities turn out to hold as equalities, and convergences \eqref{enh-convs} ensue 
from  a standard argument.
\par
To obtain \eqref{eneqlim-inequality-leq}, as in the proof of 
Lemma \ref{LEDEbrittle} we  pass to the limit as $k\to\infty$  in the upper 
energy-dissipation inequality for the adhesive system: we can now do so on the interval $(s,t)$ 
for every $t\in (0,T]$ and for  $s\in (0,t)$, out of a negligible set, 
such that convergences \eqref{enh-convs-1} \& \eqref{enh-convs-4} hold as $s$. 
This concludes the proof. 
\end{proof}
Next, we deduce the enhanced validity of the initial condition stated in 
Theorem \ref{thm:main}, Item 7.
\begin{lemma}[Enhanced initial condition \eqref{enhanced-Cauchy-dotu}] 
\label{EnhancedCauchy}
Let $u\in \mathrm{C}^0_{\mathrm{weak}}([0,T];\Spw) \cap H^2(0,T;\VZ(0)^*) $ comply with the regularity 
properties \eqref{reg-u-britt;e}, 
 with the brittle momentum balance for given 
$z\in  \mathrm{B}(0,T;\SBV(\GC;\{0,1\})) $ $  \cap \mathrm{BV}(0,T;L^1(\GC)),$ 
semistable as in \eqref{semistab-z} for all $t\in[0,T],$ and 
with the bulk energy balance \eqref{eneqlim-identity-bulk} for the brittle system. 
Then, \eqref{enhanced-Cauchy-dotu} holds true, and in particular, along any sequence 
of Lebesgue points $(t_j)_j$ of $\|\dot u(\cdot)\|_\Spw$ with $t_k\to0$ it holds 
$\dot u(t_j)\to u_1$ strongly in $\Spw$. 
\end{lemma}
\begin{proof}
We adapt the arguments of \cite[p.\ 10]{DMLar11EWED}: 
Thanks to \eqref{propu2} we have $\dot u(t_j)\rightharpoonup u_1$ in $\Spw$. Thus, in order to 
verify that $\dot u(t_j)\to u_1$ strongly in $\Spw,$ it is sufficient to show that 
\begin{equation}
\label{convtoshow}
\limsup_{j\to\infty}\|\dot u(t_j)\|_\Spw^2\leq \|u_1\|_\Spw^2\,. 
\end{equation}
From the bulk energy-dissipation balance 
\eqref{eneqlim-identity-bulk} we deduce
\begin{eqnarray*}
\tfrac{1}{2}\|\dot u(t_j)\|_\Spw^2&\leq&\tfrac{1}{2} \|  u_1\|_\Spw^2+
\int_{\Omega\backslash\GC}\tfrac{1}{2}\mathbb{C}e(u_0):e(u_0)
-\tfrac{1}{2}\mathbb{C}e(u(t_j)):e(u(t_j))\,\mathrm{d}x\\
&&\quad
+\langle\mathbf{f}(t_j),u(t_j)\rangle_\Spu -\langle\mathbf{f}(0),u_0\rangle_\Spu 
- \int_0^{t_j}\pairing{}{\Spu}{\dot{\mathbf{f}}(\tau)}{u(\tau)} \,\mathrm{d}\tau  
\\
&&\to\tfrac{1}{2}\| u_1\|_\Spw^2\quad\text{as }t_j\to0\,.
\end{eqnarray*}
Here, the convergence of the terms on the right-hand side is 
due to the regularity property $u\in H^1(0,T;\Spu)$, which ensures that $u(t_j)\to u_0$ 
strongly in $H^1(\Omega\backslash\GC;\R^d)$, and to assumption \eqref{assdata} on $\mathbf{f}$. 
Hence, \eqref{convtoshow}, 
and thus the enhanced initial condition \eqref{enhanced-Cauchy-dotu}, ensue.  
\end{proof}
\par 
Thanks to the above proved enhanced validity of the initial condition we are now 
in the position to conclude the uniqueness result stated in 
Theorem \ref{thm:main},  Item 8. 
\begin{lemma}[Uniqueness of the displacements  for a given $z \in L^\infty (0,T;\mathrm{SBV}(\GC;\{0,1\}))$]
\label{Unique} 
The uniqueness of the displacements 
holds true in the sense of Theorem \ref{thm:main},  Item 8.    
\end{lemma}
\begin{proof}
Suppose that $(u,z)$ and $(\tilde u,z)$ both are \semi energetic solutions to the 
 evolutionary  brittle system 
$\ingrsysinf$ and that they both satisfy the brittle 
momentum balance \eqref{weak-mom-brittle} with the same initial data $u_0$ and $u_1$. 
Then,    
$w:=u-\tilde u$  fulfills   \eqref{weak-mom-brittle} for almost all $t\in (0,T)$,  with $\mathbf{f}=0$ and 
 $w(0)=\dot w(0)=0$.  For $t\in (0,T)$ fixed we test  \eqref{weak-mom-brittle}   $v(t)=\dot w(t),$  
which is admissible since it satisfies 
$\JUMP{\dot w(t)}=0$ a.e.\ on $\supp z(t)$. 
To treat the quadratic bulk terms and  the external loading term resulting from this 
test we have suitable 
chain rules at our disposal, cf.\ also \eqref{momuint}. 
It remains to verify a chain rule for the inertial term  
$\langle\ddot w(t),\dot w(t)\rangle_\Spw$.
\par For this, we use the information that both 
$u$ and $v$ satisfy the bulk energy-dissipation balance \eqref{eneqlim-identity-bulk}, 
and hence, the enhanced initial condition in the sense of the above Lemma \ref{EnhancedCauchy}. 
Thus, picking a sequence $(t_j)_j,$ which are Lebesgue points for both functions 
$\|\dot u(\cdot)\|_\Spw$ 
and $\|\dot v(\cdot)\|_\Spw,$ and which satisfies $t_j\to0$ as $j\to\infty,$ we conclude 
by Lemma \ref{EnhancedCauchy} that also  
$\dot w(t_j)\to(u_1-v_1)=0$ strongly in $\Spw$. 
Moreover, observe that the chain rule equality \eqref{gelfand-1} holds true 
also for $w$ in all Lebesgue points $s,t$ of $\|\dot{w}(\cdot)\|_\Spw^2,$ 
since $w$ solves the momentum balance.  
Thus, choosing $s=t_j$ in \eqref{gelfand-1} 
and letting $j\to\infty,$ the previously deduced strong convergence 
$\dot w(t_j)\to(u_1-v_1)=0$ in $\Spw$ 
now yields the chain rule with initial datum for $w,$ namely 
$\tfrac{1}{2}\|\dot w(t)\|_\Spw^2-\tfrac{1}{2}\|\dot w(0)\|_\Spw^2
=\int_0^t\langle\varrho\ddot w(\tau),\dot w(\tau)\rangle_{\VZ(\tau)}\,\mathrm{d}\tau$.     
Hence, by exploiting the chain rule for each of the terms in \eqref{eq-u},  
we readily obtain for almost all $t\in (0,T)$ 
\begin{equation*}
\tfrac{1}{2}\|\dot w(t)\|_{\Spw}^2  
+ \int_0^t  2\calV(\dot w(s))  
\,\mathrm{d} s +
 \int_{\Omega\backslash\GC}\tfrac{1}{2}\mathbb{C}(e(w(t)):e(w(t))\,\mathrm{d}x
=0\,.
\end{equation*}
This implies that each of the positive terms on the left-hand side has to be zero separately, 
which shows that  
$w\equiv 0$ and $\dot{w} \equiv 0$ a.e.\ in  $[0,T]$. 
But then, $w \equiv 0$ \emph{everywhere} in $[0,T]$. 
\end{proof}

\bigskip

\baselineskip=11pt {\small\noindent{\it Acknowledgements}: 
  R.R.\  has been partially supported by a MIUR-PRIN'10-11 grant
  for the project ``Calculus of Variations''.
  M.T.\ has been partially supported by the Deutsche
Forschungsgemeinschaft in the Priority Program 1748 "Reliable
simulation techniques in solid mechanics. Development of non-
standard discretization methods, mechanical and mathematical
analysis" within the project "Finite element approximation of functions of bounded variation and application to
models of damage, fracture, and plasticity".
   R.R.\ and M.T.\  are also grateful for the
  support from the  Gruppo Nazionale per  l'Analisi Matematica, la
  Probabilit\`a  e le loro Applicazioni (GNAMPA)
of the Istituto Nazionale di Alta Matematica (INdAM).
 This research was carried out during several
visits of R.R.\ at the Weierstrass Institute, supported by the European
Research Council through the ERC Advanced Grant
\emph{Analysis of Multiscale Systems Driven by Functionals (267802)},
and of M.T.\ at the University of Brescia; the kind hospitality at both institutions
is gratefully acknowledged.
}

\baselineskip=12pt

{\small 
\bibliographystyle{alpha}
\bibliography{marita_lit-v2,marita_ricky_lit}
}

\end{document}